\documentclass[11pt,reqno]{amsart}

\usepackage{amsmath}
\usepackage{amsthm}
\usepackage{graphicx}
\usepackage{mathtools}
\usepackage{amsfonts}
\usepackage{amssymb}
\usepackage{hyperref}
\usepackage{bm}
\usepackage[margin=1 in]{geometry}
\usepackage{enumerate}
\usepackage[normalem]{ulem}
\usepackage{tikz}
\usepackage{xcolor}
\usetikzlibrary{matrix,arrows,positioning,automata}
  \usepackage{cancel}
\usepackage{amsmath,centernot}

\DeclareFontFamily{U}{mathx}{}
\DeclareFontShape{U}{mathx}{m}{n}{<-> mathx10}{}
\DeclareSymbolFont{mathx}{U}{mathx}{m}{n}
\DeclareMathAccent{\widecheck}{0}{mathx}{"71}

\newcommand{\nll}{\centernot{\ll}}

\numberwithin{equation}{section}

\newtheorem{theorem}{Theorem}[section]
\newtheorem{lemma}[theorem]{Lemma}
\newtheorem{proposition}[theorem]{Proposition}
\newtheorem{corollary}[theorem]{Corollary}

\theoremstyle{definition}
\newtheorem{example}[theorem]{Example}\newtheorem{definition}[theorem]{Definition}
\newtheorem{remark}[theorem]{Remark}

\newtheorem{manualassuminner}{Assumption}
\newenvironment{manualassum}[1]{%
	\renewcommand\themanualassuminner{#1}%
	\manualassuminner
}{\endmanualassuminner}

\def\CC{{C}}

\def\BB{{\Gamma}}

\def\A{{\mathcal A}}

\def\diver{\mathrm{div}}
\def\E{{\mathbb E}}
\def\F{{\mathcal F}}
\def\FF{{\mathbb F}}

\def\N{{\mathbb N}}
\def\P{{\mathcal P}}
\def\PP{{\mathbb P}}

\def\R{{\mathbb R}}
\def\T{{\mathbb T}}
\def\tr{{\mathrm{Tr}}}

\def\V{{\mathcal V}}

\def\W{{\mathcal W}}
\def\X{{\mathcal X}}
\def\Y{{\mathcal Y}}
\def\Z{{\mathbb Z}}
\DeclareMathOperator*{\avg}{avg}
\DeclareMathOperator*{\esssup}{ess\,sup}

\newcommand{\Erdos}{Erd\"{o}s-R\'enyi}

\title[Quantitative non-exchangeable propagation of chaos]{Quantitative propagation of chaos for non-exchangeable diffusions via first-passage percolation}
\author{Daniel Lacker, Lane Chun Yeung, and Fuzhong Zhou} 
\address{Department of Industrial Engineering \& Operations Research, Columbia University}
\email{daniel.lacker@columbia.edu, fz2329@columbia.edu}
\address{Department of Mathematical Sciences, Carnegie Mellon University}
\email{lyeung@andrew.cmu.edu}
\thanks{D.L.\ and F.Z.\ acknowledge support from the NSF CAREER award DMS-2045328.}

\begin{document}
\begin{abstract}
This paper develops a non-asymptotic approach to mean field approximations for systems of $n$ diffusive particles interacting pairwise. The interaction strengths are not identical, making the particle system non-exchangeable. The marginal law of any subset of particles is compared to a suitably chosen product measure, and we find sharp relative entropy estimates between the two.
Building upon prior work of the first author in the exchangeable setting, we use  a generalized form of the BBGKY hierarchy to derive a hierarchy of differential inequalities for the relative entropies.
Our analysis of this complicated hierarchy exploits an unexpected but crucial connection with first-passage percolation, which lets us bound the marginal entropies in terms of expectations of functionals of this percolation process.
\end{abstract}

\maketitle


\section{Introduction}

Suppose a large number $n$ of particles are initialized at i.i.d. positions and then evolve according to some dynamics. The dynamics of a single particle consist of a base motion plus pairwise interaction forces exerted by each of the other particles. These interactions immediately correlate the particles' positions. The question we study in this work is: After some time $t > 0$, how strong is this correlation? 
When the joint distribution of particles is exchangeable, i.e., the interactions are symmetric, the problem just posed is known as the \emph{propagation of chaos} for mean field dynamics. The broader context for our work, discussed in detail below, is a recent literature on non-exchangeable extensions of this mean field paradigm. As we will see, the non-exchangeable setting exhibits far richer correlation structures with an intriguing dependence on the matrix of interaction strengths.

Our concrete setup is as follows, simplified somewhat for this introduction. The particles are indexed by $i \in [n]=\{1,\ldots,n\}$, take values in $\R^d$, and evolve according to
\begin{equation}
dX^i_t = \Big(b_0(X^i_t) +  \sum_{j=1}^n \xi_{ij} b(X^i_t,X^j_t)\Big)dt + \sigma dB^i_t, \qquad X^i_0 \stackrel{\text{i.i.d.}}{\sim}  P_0. \label{intro:SDE-n}
\end{equation}
Here $B^i$ are independent Brownian motions, $\sigma > 0$ is constant, and $b_0$ and $b$ are \emph{self-interaction} and \emph{interaction} functions, respectively, with precise assumptions given later.
The key feature is the $n \times n$ \emph{interaction matrix} $\xi$, with nonnegative entries $\xi_{ij}$ representing the influence of particle $j$ on $i$. We assume zero diagonal entries, $\xi_{ii}=0$, to distinguish the role of the self-interaction term $b_0$.

\subsection{The exchangeable (mean field) case}
When $\xi_{ij}=1/(n-1)$ for all $i \neq j$, we are in a classical setting of interacting diffusions of mean field type, and there is a well-established sense in which the particles are approximately i.i.d. This makes use of a limiting distribution $Q_t$, defined by the \emph{McKean-Vlasov equation}
\begin{equation}
dY_t = \bigg(b_0(Y_t) + \int_{\R^d} b(Y_t,y)\,Q_t(dy)\bigg)\,dt + \sigma dB_t, \qquad Q_t=\mathrm{Law}(Y_t), \ \ Y_0 \sim P_0. \label{intro:SDE-MV-MFcase}
\end{equation}
The phenomenon of \emph{propagation of chaos} is that, for $t > 0$ and $k$ fixed, the joint law $P^k_t$ of $(X^1_t,\ldots,X^k_t)$ converges weakly to the product measure $Q_t^{\otimes k}$ as $n\to\infty$.
By now there are many methods for proving propagation of chaos, reviewed thoroughly in \cite{chaintron2021propagation,chaintron2022propagation}, and we will discuss related literature in more detail in Section \ref{se:literature}. We focus here on one methodology, based on estimating the relative entropy (or Kullback-Leibler divergence) $H^k_t := H(P^k_t\,|\,Q_t^{\otimes k})$ for $k \le n$.

Relative entropy methods can be divided into two categories, \emph{global} and \emph{local}. 

Global methods estimate the entropy $H^n_t$ of the full $n$-particle configuration. The best possible global estimate in this setting is $H^n_t = O(1)$, shown for instance in \cite{jabin2016mean,jabin2018quantitative,jabir2019rate}. Propagation of chaos then follows from the subadditivity inequality $H^k_t\le (k/n)H^n_t = O(k/n)$ for $k \le n$. Global entropy methods have gained popularity in recent years in large part because they are powerful enough to handle physically relevant singular interaction functions \cite{jabin2018quantitative}.

Local methods, introduced recently by the first author  \cite{lacker2022hierarchies}, are less robust to singular interaction functions but yielded for the first time the optimal rate of convergence $H^k_t = O((k/n)^2)$. The optimality of this bound is justified by a matching lower bound in the Gaussian case, where $b_0$ and $b$ are linear. This reveals, surprisingly, that the subadditivity bound is suboptimal.
The proof in \cite{lacker2022hierarchies} uses (a form of) the \emph{BBGKY hierarchy}, which is a system of Fokker-Planck equations satisfied by $(P^k_t)_{t \ge 0}$ for each $k \le n-1$, in which the drift in the equation for $P^k$ depends on $P^{k+1}$. This is used to derive a hierarchy of differential inequalities,
\begin{equation}
\frac{d}{dt}H^k_t \le c_1 \frac{k^2}{n^2} + c_2 k (H^{k+1}_t - H^k_t), \quad k=1,\ldots,n-1, \qquad \frac{d}{dt}H^n_t \le c_1 . \label{intro:hierarchy-MF}
\end{equation}
where $c_1$ and $c_2$ are constants depending on $(b,b_0)$ but independent of $(n,k,t)$. 
The key assumption in \cite{lacker2022hierarchies} is that the pushforward under $Q_t$ of the centered interaction function $y \mapsto b(x,y) - \int b(x,\cdot)\,dQ_t$ is subgaussian, uniformly in $x$ and bounded time intervals, which notably includes bounded or Lipschitz $b$. We adopt a similar assumption in this work.

\subsection{The non-exchangeable setting} \label{se:intro-nonexchangeable}
In this work, we adapt the hierarchical approach of \cite{lacker2022hierarchies} to the non-exchangeable setting.

A first challenge of non-exchangeability  is the lack of an obvious choice of a reference measure, to replace the $Q_t$ arising from the McKean-Vlasov equation. One way to choose a reference measure would be to identify an alternative large-$n$ limit. This has been done under  structural assumptions on $\xi$ of an asymptotic nature, namely that it admits a suitable graphon limit (in the sense of \cite{lovasz2012large}), and we review some of this literature in Section \ref{se:literature} below.
We instead adopt a non-asymptotic perspective by working with a particular choice of reference measure termed the \emph{independent projection} in \cite{lacker2023independent}. It is described by the following SDE system, for $i \in [n]$:
\begin{equation}
dY^i_t = \bigg(b_0(Y^i_t) + \sum_{j=1}^n\xi_{ij}\int_{\R^d} b(Y^i_t,y)\,Q^j_t(dy)\bigg)dt + \sigma dB^i_t, \quad Q^i_t=\mathrm{Law}(Y^i_t), \ Y^i_0  \stackrel{\text{i.i.d.}}{\sim} P_0. \label{intro:indepproj}
\end{equation}
The paper \cite{lacker2023independent}  explains certain senses in which \eqref{intro:indepproj} can be considered a canonical way to approximate \eqref{intro:SDE-n} in distribution by a product measure.
In much of the literature on mean field \emph{dynamics}, the phrase ``mean field approximation" has an asymptotic connotation, associated with a  large-$n$ limit.
In this work we favor the non-asymptotic interpretation of the phrase, simply as the approximation of an $n$-dimensional probability measure by a \emph{product} measure. This interpretation is common in \emph{equilibrium} statistical physics as well as in the more recent literatures on variational inference in Bayesian statistics \cite{blei2017variational} and nonlinear large deviations \cite{chatterjee2016nonlinear}.

The dynamics of $(Q^i_t)_{t \ge 0}$ for $i \in [n]$ can be alternatively described in terms of a system of  $n$ coupled Fokker-Planck equations. This PDE system appeared in the recent paper \cite{jabin2021mean}, which studied large-$n$ limits of non-exchangeable systems like \eqref{intro:SDE-n}, under minimal assumptions on $\xi$.
They use this PDE system as a means of disentangling two problems which can be seen as separate. The first problem is to approximate the original $n$-particle system \eqref{intro:SDE-n} by one in which  the particles are \emph{independent}, with the system \eqref{intro:indepproj} being a natural choice.  The second problem is to identify a large-$n$ limit  for the    system \eqref{intro:indepproj}, taking advantage of the independence.
In the setting of \cite{jabin2021mean}, the first problem was straightforward because they were not after sharp quantitative estimates, and the second problem was the main source of difficulty.
In this paper, we focus exclusively on the first problem, because unlike the second it can be studied \emph{non-asymptotically}, without any need to identify a large-$n$ limit for $\xi$.
Moreover, there is a rich class of examples for which the second problem is vacuous, which is when the matrix $\xi$ is \emph{stochastic}, i.e., has constant row sums:
\begin{equation}
\sum_{j=1}^n\xi_{ij} =1, \quad \text{for all } i=1,\ldots,n. \label{intro:stochasticmatrix}
\end{equation}
In this case, as  noted in \cite[Remark 2.3]{jabin2021mean} and \cite[Remark 2.7]{lacker2023independent}, the independent projection reduces to the usual McKean-Vlasov equation; that is, $Q^i_t=Q_t$ for all $i$ where $Q_t$ is given by \eqref{intro:SDE-MV-MFcase}, and there is essentially no $n$-dependence  in \eqref{intro:indepproj}.
This is consistent with the recent folklore that the non-exchangeable system \eqref{intro:SDE-n} converges to the usual McKean-Vlasov when \eqref{intro:stochasticmatrix} holds and when the matrix $\xi$ is sufficiently dense; see
\cite{oliveira2019interacting} for a general theorem of this nature.
Our main results give the first sharp quantitative rates of convergence in this context.

An important special case of \eqref{intro:stochasticmatrix} is the following:

\begin{definition} \label{def:randomwalk-case}
We say \emph{the random walk case} to refer to the situation in which we are given a graph with vertex set $[n]$ and no isolated vertices, and the interaction matrix is $\xi_{ij}=1/\mathrm{deg}(i)$ when $(i,j)$ is an edge and $\xi_{ij}=0$ when $(i,j)$ is not an edge. Here $\mathrm{deg}(i)$ denotes the degree of vertex $i$.
If there is an isolated vertex $i$, we set $\xi_{ij}=0$ for all $j$ and note that \eqref{intro:stochasticmatrix}  does not hold unless one restricts attention to a connected component of the graph.
\end{definition}

In other words, in the random walk case, $\xi$ is the transition matrix of the simple random walk on the graph (if the graph is connected). Each particle interacts with the average of its neighbors in the graph. Note that when the graph is the complete graph we recover the usual mean field case. A notable sub-case is the following:

\begin{definition} \label{def:m-regular-case}
We say \emph{the $m$-regular graph case} to refer to the random walk case with the further restriction that the graph is $m$-regular, i.e., $\mathrm{deg}(i)=m$ is the same for all $i$. 
\end{definition}

Once a reference measure is chosen, a second challenge of non-exchangeability is that different choices of $k$ out of the $n$ particles can have different joint laws.
For a  set of indices $v \subset [n]$, let us write $P^v_t$ for the law of $(X^i_t)_{i \in v}$  and similarly $Q^v_t$ for the law of $(Y^i_t)_{i \in v}$, at a time $t \ge 0$. The main quantity we will study is the relative entropy
\[
H_t(v) := H(P^v_t\,|\,Q^v_t).
\]

\subsection{Summary of our results} \label{se:intro:summary}
Our main results are bounds on the entropies $H_t(v)$, many of which are sharp, which quantify the   approximate independence of the subcollection of particles $v \subset [n]$. 
We do not treat only the case of \eqref{intro:stochasticmatrix}, but an important standing assumption throughout the paper is that the row sums of $\xi$ are bounded:
\begin{equation}
\max_{i \in [n]}\sum_{j=1}^n\xi_{ij} \le 1 . \label{intro:sums-row}
\end{equation}
The constant 1 here is arbitrary, as any other constant could be absorbed into $b$. 
Let us summarize some highlights of our main results, which will be developed in full generality and detail in Section \ref{sec:main.results}, and with notable examples of interaction matrices discussed in Section \ref{sec:graph.examples}.  While we stress that our results are non-asymptotic, to understand them it is helpful to  imagine that we are in an asymptotic regime, given a sequence of $\xi$ of size $n \times n$ with $n \to\infty$, though we suppress the dependence of $\xi$ on $n$. Asymptotic notation like $k=o(n)$ should be interpreted accordingly. 

\subsubsection{Maximum entropy estimates} 
In Theorem \ref{theorem:max} we estimate the maximum entropy over all $k$-particle configurations, for each $k \in [n]$:
\begin{equation}
\widehat{H}^k_t :=\max_{v \subset [n], \ |v|=k} H_t(v) \lesssim   (\delta k + 1)(\delta k)^2, \qquad \text{where } \ \delta := \max_{i,j \in [n]}\xi_{ij}, \label{intro:max}
\end{equation}
where the constant hidden in $\lesssim$ depends on the details of $(b_0,b)$ but not on $k$, $n$, or $\xi$ subject to \eqref{intro:sums-row}.  The constant can also depend on the choice of $t > 0$, and in fact it also holds at the level of path-space distributions, though  we will show also that \eqref{intro:max} admits a uniform-in-time version whenever $Q^i_t$ satisfies a log-Sobolev inequality uniformly in $(i,t)$; the same remarks apply to the other results summarized in this section.

Hence,  the maximum entropy $\widehat{H}^k_t$ is controlled by the maximum entry $\delta$ of $\xi$. While the precise rate is not a priori obvious, the qualitative result is intuitive: In the regime $n\to\infty$, we cannot expect $\widehat{H}^k_t$ to vanish if $\delta$ does not, because a pair of particles $(i,j)$ with interaction strength $\xi_{ij}$ bounded away from zero cannot be expected to become asymptotically independent. Note that the number $k\le n$ of particles is in general allowed to grow with $n$, and $\widehat{H}^k_t$ vanishes as long as $k = o(1/\delta)$; this is in the spirit of what is sometimes called the \emph{the size of chaos} or \emph{increasing propagation of chaos} \cite{arous1999increasing}.  

The bound \eqref{intro:max} becomes simply $(\delta k)^2$ in the regime  $k=O(1/\delta)$. A Gaussian example (Remark \ref{re:sharpness-max}) shows that this is sharp by obtaining a matching lower bound of order $(\delta k)^2$ in many cases, such as in the regular graph case when $k=O($the number of vertices in the largest clique$)$.

The parameter $\delta$ may be viewed as a crude measure of \emph{denseness} of the matrix $\xi$. In the $m$-regular graph case (Definition \ref{def:m-regular-case}) we have $\delta=1/m$, so that $\widehat{H}^k_t =O((k/m)^2) \to 0$ as long as $m\to\infty$ and $k=o(m)$. The condition $m\to\infty$ means that the graph is \emph{dense} in a very mild sense, as there is no restriction on how quickly $m$ diverges relative to the number of vertices $n$. The sparse regime in which $m$ stays bounded as $n\to\infty$ is very different, and $\widehat{H}^k_t$ does not vanish; see Section \ref{se:literature} for further discussion.

\subsubsection{Average entropy estimates} 
The bound \eqref{intro:max} is useful in many cases but is blind to the heterogeneity of the interactions, focusing only on the strongest  (or worst-case) interaction strength $\delta$. Finer information is available if we relax the maximum to an average.
In Corollary \ref{co:avg.1way} we estimate the average entropy over all $k$-particle configurations, for each $k \in [n]$:
\begin{equation}
\overline{H}^k_t := \frac{1}{{n \choose k}}\sum_{v \subset [n], \ |v|=k} H_t(v) \lesssim (\delta k + 1)\frac{k^2}{n} \sum_{i=1}^n \delta_i^2, \qquad \text{where } \  \delta_i := \max_{j \in [n]}\xi_{ij}. \label{intro:avg1way}
\end{equation}
This reveals, in contrast with the maximal entropy $\widehat{H}^k_t$, that the average entropy $\overline{H}^k_t$ can be small even if  some of the rows of $\xi$ have large maximal entry. 

In the random walk case (Definition \ref{def:randomwalk-case}) this highlights an interesting  dichotomy of denseness thresholds. For the maximum entropy $\widehat{H}^k_t$ to vanish, $\min_i\mathrm{deg}(i)$ must diverge. 
This happens in an \Erdos\ random graph $G(n,p)$ even if $p=p_n$ is allowed to vanish, as long as $\liminf np/\log n > 1$ (the connectivity threshold). 
For the average entropy $\overline{H}^k_t$ to vanish, we need only that the \emph{typical degree} diverges in the sense that $(1/n)\sum_{i:\mathrm{deg}(i) \neq 0} \mathrm{deg}(i)^{-2} \to 0$. In the \Erdos\ graph this happens as long as $np \to \infty$ at any speed.  In summary, defining \emph{propagation of chaos} to mean that $\widehat{H}^k_t \to 0$ versus $\overline{H}^k_t \to 0$ leads to different (sharp) denseness thresholds.
This dichotomy between worst-case and typical behaviors is new to the non-exchangeable setting. A similar situation appeared in a recent study \cite[Section 2.3.1]{lacker2022case} of stochastic games on networks.

Our proof of \eqref{intro:avg1way} requires, however, an additional assumption that \emph{column sums} are bounded:
\begin{equation}
\max_{j \in [n]}\sum_{i=1}^n\xi_{ij} \le c. \label{intro:sums-column}
\end{equation}
The hidden constant in \eqref{intro:avg1way} depends on $c$. In the random walk case, \eqref{intro:sums-column} entails essentially that no vertex can have too many low-degree neighbors. 
We suspect that \eqref{intro:avg1way} is true even without assuming \eqref{intro:sums-column}. It is restrictive in the \Erdos \ case, as it again requires $\liminf np/\log n > 1$.

However, if we replace $\overline{H}^k_t$ by a \emph{non-uniform} average, then we can do away with the assumption \eqref{intro:sums-column}. A notable special case is when $\xi$ is the transition matrix of a Markov chain on $[n]$ (i.e., \eqref{intro:stochasticmatrix} holds) and $\pi \in \R^n$ is an invariant measure.  Focusing on the cases $k=1,2$, we show in Theorem \ref{theorem:avg-markov} that
\begin{equation}
\sum_{i=1}^n \pi_i H_t(\{i\}) \le \sum_{i,j=1}^n \pi_i \xi_{ij} H_t(\{i,j\}) \lesssim  \sum_{i=1}^n \pi_i\delta_i^2. \label{intro:avg1way-nonunif}
\end{equation}
In the random walk case we have $\pi_i=\mathrm{deg}(i)/\sum_j\mathrm{deg}(j)$, and the middle average can be written as $\widetilde{H}^2_t:=(1/|E|)\sum_{e \in E} H_t(e)$, where $E$ is the edge set of the graph. The right-hand side of \eqref{intro:avg1way-nonunif} then vanishes as long as the average degree diverges, which includes the \Erdos \ case all the way down to the optimal denseness threshold $np \to \infty$ (at any rate).  It may appear surprising that $\widetilde{H}^2_t$ vanishes under this minimal denseness condition. On the one hand, adjacent particles should be more strongly correlated than non-adjacent ones, and we would expect $\widetilde{H}^2_t$ to be  larger than $\overline{H}^2_t$, which averages over all pairs of vertices adjacent or not. On the other hand, note that $\pi$ gives higher weights to the high-degree vertices, at which more averaging occurs.

\subsubsection{Sharper average entropy estimates} 
The bound \eqref{intro:avg1way} is practical in many cases but not always sharp. Focusing for now on the case where $\xi$ is symmetric, we obtain in Theorem \ref{theorem:avg.2way.asym} the sharper bound
\begin{equation}
\overline{H}^k_t  \lesssim (\delta k + 1)\bigg(\frac{k^2}{n^2}\sum_{i,j=1}^n \xi_{ij}^2 + \frac{k}{n}\sum_{i=1}^n\Big(\sum_{j=1}^n\xi_{ij}^2\Big)^2 \bigg). \label{intro:avg}
\end{equation}
If $k=O(1/\delta)$ and $k = o(n)$, we explain in Remark \ref{re:sharpness-avg} that the estimate of \eqref{intro:avg} is sharp (for symmetric $\xi$). This sharpness is a consequence of a calculation in a Gaussian example, stated in Theorem \ref{theorem:Gaussian.avg.entropy}: 
\begin{equation}
\overline{H}^k_t \asymp \frac{k(k-1)}{n(n-1)}\sum_{i,j=1}^n  \xi_{ij}^2 + \frac{k (n - k)}{n (n - 1)}\sum_{i=1}^n \bigg(\sum_{j=1}^n  \xi_{ij}^2 \bigg)^2. \label{intro:Gaussian-avg}
\end{equation}
Here $\asymp$ means both $\lesssim$ and $\gtrsim$. 

A notable advantage of \eqref{intro:avg} over \eqref{intro:avg1way} is that the former admits a larger \emph{size of chaos}. In the regular graph case, these bounds respectively simplify to  $\overline{H}^k_t \lesssim (k/m+1)(k^2/nm+k/m^2)$ and $\overline{H}^k_t \lesssim (k/m+1)(k/m)^2$. The latter vanishes only when $k=o(m)$, whereas the former allows $k \gtrsim m$ as long as $k = o(\min\{(nm^2)^{1/3},m^{3/2}\})$

The above estimates, especially \eqref{intro:Gaussian-avg}, reveal a dramatic failure of the famous \emph{subadditivity inequality} to capture the correct behavior of the average entropy. The subadditivity of entropy states that
\begin{equation*}
\overline{H}^{k}_t \le (k/n)\overline{H}^n_t, \qquad 1 \le k \le n.
\end{equation*}
See \cite[Theorem 1]{dembo1991information} for this level of generality, where exchangeability is not assumed. Applying \eqref{intro:Gaussian-avg} with $k=n$ shows that $\overline{H}^n_{T} \asymp \sum_{ij}\xi_{ij}^2$. Using subaddivity then yields merely $\overline{H}^{k}_{T}  \lesssim (k/n)\sum_{ij}\xi_{ij}^2$, which completely misses the correct shape given by \eqref{intro:Gaussian-avg}.

\subsubsection{Setwise entropy estimates} 

We also obtain some estimates valid for each set $v \subset [n]$, without averages or maxima. In Theorem \ref{theorem:pointwise} we bound the entropy for each set $v \subset [n]$: 
\begin{equation}
H_t(v) \lesssim (\delta|v|+1)\bigg(\sum_{i,j\in v}\xi^2_{ij} + \delta \sum_{i,j \in v}(\xi^\top \xi+\xi\xi^\top)_{ij}  + \delta^2|v|\bigg). \label{intro:pointwise-v}
\end{equation} 
The right-hand side depends on not only the size but also the connectedness of the set $v$. This is seen most clearly in the $m$-regular graph case, discussed in detail in Section \ref{ex:regular}, where
\begin{equation}
H_t(v) \lesssim  \bigg(\frac{|v|}{m} + 1\bigg)\bigg(\frac{p_1(v)}{m^2}  + \frac{p_2(v)}{m^3}\bigg). \label{intro:pointwise-v-mregular}
\end{equation}
Here we define $p_\ell(v)$ to be the number of paths of length $\ell$ that start and end in $v$. 
The first term $p_1(v)$ can range from 0 to $|v|(|v| \wedge m)$, while the second term $p_2(v)$ can range from $|v|m$ to $|v|^2m$.
The smallest values are obtained when $v$ is \emph{disconnected}, in the sense that its vertices are nonadjacent and have no common neighbors.

\subsection{From the BBGKY hierarchy to first-passage percolation} \label{se:intro:FPPbound}

Our proofs proceed through a new but natural variant of the BBGKY hierarchy for the non-exchangeable setting. For each nonempty $v \subset [n]$, a Fokker-Planck equation can be written for the evolution $(P^v_t)_{t \ge 0}$ which depends on $(P^{v \cup \{j\}}_t)_{t \ge 0}$ for each $j \notin v$. This is a simple adaptation of the usual argument for deriving the BBGKY hierarchy in the exchangeable case, and we defer details to Section \ref{subsec:iterated.hierachy}. Adapting the methods of \cite{lacker2022hierarchies}, we then derive the following differential inequalities which generalize \eqref{intro:hierarchy-MF}:
\begin{equation}
\frac{d}{dt}H_t(v) \le \CC(v) +  \sum_{j \notin v} \A_{v \to j}(H_t(v\cup \{j\}) - H_t(v)), \quad v \subset [n], \label{intro:hierarchy-new}
\end{equation}
where, for certain constants $c_1$ and $c_2$ depending on $(b_0,b)$, and for any matrix $R$ belonging to a set $\mathcal{R}$ depending on $\xi$, we define
\begin{align*}
\CC(v) := c_1\sum_{i \in v}\Big(\sum_{j \in v} \xi_{ij}\Big)^2, \qquad \A_{v \to j}:= c_2\sum_{i \in v}\xi_{ij}.
\end{align*}
The hierarchy \eqref{intro:hierarchy-new} is indexed by \emph{subsets} rather than \emph{elements} of $[n]$, thanks to non-exchangeability, which makes it significantly more complex to analyze than \eqref{intro:hierarchy-MF}. In the exchangeable case, the analysis of the hierarchy \eqref{intro:hierarchy-MF} in \cite{lacker2022hierarchies}
started by applying Gronwall's inequality and iterating the resulting integral inequality. The resulting estimates involved convolutions of the exponential functions $t \mapsto e^{-c_2 k t}$, for $k =1,\ldots,n$, which were simplified by exploiting crucially the fact that the exponents $c_2 k$ are in arithmetic sequence.
Attempting to adapt this program to the hierarchy \eqref{intro:hierarchy-new}, one quickly encounters unwieldy exponential convolutions with arbitrary, unrelated exponents. The more recent paper \cite{hess2023higher} gave a simpler inductive approach in the exchangeable case, though it still relied on the arithmetic sequence of exponents.
 
What opens the door to a tractable analysis of the new hierarchy \eqref{intro:hierarchy-new}  is an unexpected appearance of a 
continuous-time Markov process $(\X_t)_{t \ge 0}$ taking values in the space of sets, $2^{[n]}$. This process, which we call \emph{the percolation process} for reasons explained below, is defined as follows: At each jump time, a single number from $[n] \setminus \X_t$ is added to $\X_t$, and numbers are never removed.  The transition rate from $v$ to $v\cup \{j\}$ is $\A_{v \to j}$, for each $v \subset [n]$ and $j \notin v$. 
In other words, the infinitesimal generator $\A$ of the process is an operator which acts on functions $F : 2^{[n]} \to \R$ via
\begin{equation*}
\A F(v) = \sum_{j \notin v}\A_{v \to j} \big(F(v \cup \{j\}) - F(v)\big), \quad v \subset [n].
\end{equation*} 
With this notation, we may write \eqref{intro:hierarchy-new} as a pointwise inequality between functions on $2^{[n]}$:
\begin{equation}
\frac{d}{dt}H_t \le \CC + \A H_t. \label{intro:ineq:Markov}
\end{equation}
Letting $\E_v[\cdot]$ denote the expectation for this Markov process under the initialization $\X_0^{R}=v$, we have the identity
\begin{equation*}
\E_v[F(\X_t)] = e^{t\A}F(v) , \quad v \subset [n], \ \ t > 0.
\end{equation*}
This formula emphasizes the important fact that the semigroup $e^{t\A}$ is monotone with respect to pointwise inequality.
Using this monotonicity, the inequality \eqref{intro:ineq:Markov} implies
\begin{equation*}
\frac{d}{dt}\Big( e^{t\A} H_{T-t}\Big) = e^{t\A} \Big(\A H_{T-t} + \frac{d}{dt}H_{T-t}\Big)  \ge - e^{t\A}\CC,
\end{equation*}
for any $T > t > 0$. We deduce for $v \subset [n]$ that
\begin{align}
H_{T}(v) &\le e^{T\A}H_{0}(v) + \int_0^T e^{t\A}\CC(v)\,dt \nonumber \\
	&= \E_v[H_{0}(\X_T)] + \int_0^T \E_v[\CC(\X_t)]\,dt .  \label{intro:Hv-bound}
\end{align}
This is essentially (an inequality form of) a Feynman-Kac formula for the Markov process $\X$. Note that previously in this introduction we have assumed the initial laws agree, $P_0=Q_0$, so that $H_0(\cdot)\equiv 0$ in the above inequality, but this can (and will) be easily generalized.

The percolation process has appeared in many guises in the literature. We adopt the term \emph{percolation} in light of its equivalence with the following model of \emph{first passage percolation} (FPP). Suppose $\xi$ is symmetric, and consider the (simple, undirected) graph $[n]$ with edge set $E=\{\{i,j\} \in [n]^2 : \xi_{ij} > 0\}$. Equip each edge $e=\{i,j\} \in E$ with an independent exponential random variable $\tau_e$ with rate $\xi_{ij}$. Any path in the graph is then assigned a weight, calculated by summing $\tau_e$ over those edges $e$ belonging to the path. The distance between two vertices $(i,j)$ is defined to be the minimal weight over all paths from $i$ to $j$. In this manner, the vertex set $[n]$ becomes a random metric space. Given an initial set $v \subset [n]$, we may use this random metric to define $\mathcal{B}_t$ as the set of points of distance at most $t$ from $v$. Then the process $\mathcal{B}_\cdot$ has the same distribution as our process $\X_\cdot$ initialized from $\X_0=v$. This follows from the memoryless property of the exponential distribution, as was first observed by Richardson \cite{richardson1973random}. 
The discrete-time counterpart of this continuous-time process is  a well known model of stochastic growth called the \emph{Eden model} \cite{eden1961two}.
A reasonable alternative name for $\X$ would be the \emph{infection process} in light of its connection with the Susceptible-Infected (SI) model of epidemiology: Given  an initial set $\X_0=v$  of infected nodes, an uninfected (susceptible) node $j$ then infected at rate $\A_{v \to j}$. The infected set grows but never shrinks (i.e., there is no recovery unlike the more common SIR model), and $[n]$ is an  absorbing state. 
The simplest case to understand is when $\xi$ is (a scalar multiple of) the adjacency matrix of a graph $G$ on vertex set $[n]$; then, the transition rate $\A_{v \to j}$ is proportional to the   number of neighbors of $j$ in $v$.

The inequality \eqref{intro:Hv-bound} reduces the problem of estimating the entropies $H_t(v)$ to the problem of estimating the key quantity  $\E_v[\CC(\X_t)]$ . The function $\CC(v)$ measures how strongly the set $v$ is connected internally. 
Existing results on FPP do not tell us much about  $\E_v[\CC(\X_t)]$, even in the nicest regular graph settings.
The canonical setting of FPP is the lattice $\Z^2$, or more generally $\Z^d$, rather than general graphs. The primary questions studied in the literature pertain to the \emph{limit shape} of $t^{-1}\X_t$ as $t\to\infty$, the fluctuations of passage times, and the existence and structure of geodesics \cite{auffinger201750}. On the other hand, our setting requires finite-time estimates of the connectivity of $\X_t$.
FPP (or the SI model) has been studied on large (finite) random graphs, though the main results again pertain to the long-time rather than transient behavior, or to the distance between typical vertices, or to the number of edges contained in the shortest path (hopcount) \cite{van2001first}.

Because of the lack of applicable prior results on FPP, a significant technical effort in this paper is in the analysis of $\E_v[\CC(\X_t)]$.
Our approach is somewhat inductive in nature. For a few sufficienty simple functions $F$ we have $\A F \le cF$ for a constant $c$, and then the inequality $(d/dt)\E_v[F(\X_t)]=\E_v[\A F(\X_t)] \le c\E_v[F(\X_t)]$ can be combined  with Gronwall's inequality.
For more complicated functions $F$, we look for bounds of the form $\A F \le cF + G$, where  $G$ is a function for which we already know how to estimate $\E_v[G(\X_t)]$.
We build up in this way toward estimates for certain tractable functions $F$ which bound $\CC$ from above.
These estimates are initially obtained in terms of complicated matrix expressions involving $\xi$ (Proposition \ref{pr:expectations}), and it takes further effort (Section \ref{se:proofs-concrete}) to simplify to the forms summarized in this introduction. A challenge in this final simplification is that spectral arguments are not helpful. The operator norm of $\xi$ is no smaller than 1 in many examples of interest (such as the regular graph case), and so the averaging effects of a dense matrix $\xi$ must be captured by other means. This is in line with the literature on graph limits which identifies the more appropriate norm to be the $\ell_\infty\to\ell_1$ operator norm (equivalently, the cut-norm).

In order to obtain our sharpest estimate \eqref{intro:avg} on average entropy, a non-trivial refinement of this method is needed. We in fact prove a family of inequalities of the form \eqref{intro:hierarchy-new}, where $\A_{v\to j}$ is replaced by $\A_{v\to j} := c_2\sum_{i \in v}R_{ij}$, and $R$ is any matrix from a certain family $\mathcal{R}$ which depends on $\xi$. For each such $R$ we may define a corresponding percolation process $\X^R$, and we may improve the inequality \eqref{intro:Hv-bound} by putting an infimum over $R \in \mathcal{R}$ on the right-hand side. The matrix $\xi$ itself  belongs to $\mathcal{R}$, and we use this choice in proving most of our main results. But for \eqref{intro:avg} we use a different choice (see Section \ref{se:proof:sharper}) which, in a sense, down-weights outliers in each row. 

\subsection{A note on the mean field case}
In fact, even in the well-understood mean field case $\xi_{ij}=1_{i \neq j}/(n-1)$, the above Markov process  perspective yields a simple alternative derivation of the optimal estimate $H^k_t = O((k/n)^2)$ from the hierarchy \eqref{intro:hierarchy-MF}.
Let $(\mathcal{Y}_t)_{t \ge 0}$ denote the \emph{Yule  process} process with rate $c_2$. This is the most classical pure-birth process, the continuous-time Markov chain which transitions from state $k$ to $k+1$ at rate $c_2k$, for each $k \in \N$. The infinitesimal generator  of the stopped process $\min(\Y_t,n)$ maps a function $F:[n] \to \R$ to the function $k \mapsto c_2 k(F(k+1)-F(k))1_{k < n}$. By the same  argument which leads from \eqref{intro:hierarchy-new} to \eqref{intro:Hv-bound}, the hierarchy \eqref{intro:hierarchy-MF} implies for $k \in [n]$ that
\begin{equation}
H^k_t \le \frac{c_1}{n^2} \int_0^t \E_k[(\min(\mathcal{Y}_s,n))^2]\,ds \le \frac{c_1}{n^2} \int_0^t \E_k[ \mathcal{Y}_s^2]\,ds. \label{intro:Hv-Yulebound}
\end{equation}
(We have omitted the time-zero entropy term, assumed to be zero here for simplicity.)
The distribution of $\mathcal{Y}_t$ given $\mathcal{Y}_0=k$ is known explicitly to be negative binomial \cite[Example 6.8]{ross2014introduction}; it is the same as the law of the number of trials until the $k^{\rm th}$ success, when trials of success probability $p=e^{-c_2t}$ are repeated independently. The second moment of this distribution is known explicitly and bounded by $2k^2/p^2$, and we recover $H^k_t = O((k/n)^2)$.
Even without knowing the explicit distribution, moments of $\Y_t$ can easily be estimated using the infinitesimal generator and Gronwall's inequality, e.g.,
\[
\frac{d}{dt}\E_k[\mathcal{Y}_t^2] = c_2\E_k[\mathcal{Y}_t((\mathcal{Y}_t+1)^2-\mathcal{Y}_t^2)1_{\{\Y_t < n\}}] \le 3c_2\E_k[ \mathcal{Y}_t^2 ].
\]

To connect this argument more clearly to our percolation process, notice in the mean field case that $\A_{v \to j} = c_2k/(n-1)$ whenever $|v|=k$ and $j \notin v$.
It follows that the cardinality process $|\X_t|$ is itself Markovian. 
Its transition $k\to k+1$ occurs at rate $c_2k(n-k)/(n-1)$, which is smaller than the corresponding rate $c_2 k$ for the Yule process. 
By scaling the exponential holding times one can therefore couple $\mathcal{Y}$ with $\X$ in such a way that $|\X_t| \le \mathcal{Y}_t$ a.s.

A final detail worth mentioning is that, in the mean field case, the term $C(v)$ in our new hierarchy \eqref{intro:hierarchy-new} becomes $C(v)=c_1k^2(k-1)/(n-1)^2$ for $|v|=k$, which is a factor of order $k$ larger than the corresponding term in \eqref{intro:hierarchy-MF}. In fact, the paper \cite{lacker2022hierarchies} initially obtains \eqref{intro:hierarchy-MF} with $k^3/n^2$ in place of $k^2/n^2$. The improvement from $k^3/n^2$ to $k^2/n^2$ in \cite{lacker2022hierarchies} requires a second pass through the argument in which certain covariance estimates are sharpened. A similar sharpening procedure appears in our arguments,  allowing us to improve an initial factor of $k$ to the factor $(\delta k + 1)$ which appeared in \eqref{intro:max}, \eqref{intro:avg1way}, and \eqref{intro:avg}.

\subsection{Related literature} \label{se:literature}

\subsubsection{Relative entropy methods, global and local}
The literature on mean field limits for exchangeable systems is vast, and there are many different techniques for proving propagation of chaos. For a comprehensive recent review, we refer to the two-volume survey \cite{chaintron2021propagation,chaintron2022propagation}. For our purposes, it is worth highlighting some recent progress on relative entropy methods, which can be divided roughly into \emph{global}  versus \emph{local} methods. Global entropy methods, based on estimating  $H_t([n])$ in our notation, were carried out in  \cite{arous1999increasing,jabin2016mean,jabir2019rate,lacker2018strong} for non-singular interactions. A breakthrough paper by Jabin-Wang \cite{jabin2018quantitative} revealed the power of entropy methods for singular interactions, which appear in many physically relevant models. This was developed further in \cite{bresch2023mean} in conjunction with the modulated energy method initiated by Duerinckx \cite{duerinckx2016mean} and Serfaty \cite{serfaty2020mean}. This has lead to significant progress on Riesz and Coulomb-type interactions, and we refer to \cite{de2023sharp,de2023attractive} for recent  results and further references.
 See also \cite{cattiaux2024entropy} for a recent probabilistic approach to singular interactions, based on path-space entropy methods yielding mostly qualitative results. An interesting recent contribution \cite{jackson2024concentration} shows how to derive concentration inequalities from global entropy estimates.
Entropy methods can also yield uniform-in-time bounds, usually requiring some form of logarithmic Sobolev inequality \cite{guillin2024uniform,rosenzweig2023modulated}.

These global entropy methods at best achieve estimates like $H^n_t = O(1)$, which by subadditivity leads only to the suboptimal $H^k_t = O(k/n)$. To show the optimal order of $(k/n)^2$, the local approach summarized above at \eqref{intro:hierarchy-MF} was developed by the first author in \cite{lacker2022hierarchies}.
The followup work \cite{lacker2023sharp} treated the uniform-in-time case, which was recently improved in \cite{monmarche2024time} via sharper estimates of the log-Sobolev constants along dynamics.
Going a step further, the paper \cite{hess2023higher} showed that the $n$-particle law $P_t$ admits a cumulant-type expansion in powers of $1/n$ around the product measure $Q_t^{\otimes n}$, and they use hierarchical methods to prove optimal $L^2$ estimates on each term in the expansion.
The main advantage of the local approach is that it can achieve the optimal rate, though it did not at first appear to handle singularities as well as global methods.
The paper \cite{han2022entropic} made some progress in this direction, treating mild $L^p$-for-large-$p$ singularities but under other restrictive assumptions.
A recent breakthrough \cite{wang2024sharp} showed how to combine the methods of \cite{lacker2022hierarchies} and \cite{jabin2018quantitative} in order to achieve the optimal entropy estimate for models with singular interaction functions in $W^{-1,\infty}$, at least for high temperature or short time.
The optimal entropy estimate was obtained recently in \cite{hu2024mimicking} for systems driven by fractional Brownian motion.
Let us mention also \cite{bresch2022new}, which adopted a different local perspective based on propagating weighted $L^p$-norm estimates along the BBGKY hierarchy and was able to rigorously derive the (singular)  Vlasov-Poisson-Fokker-Planck equation on short time horizon.

We should stress that our results, like \cite{lacker2022hierarchies}, cannot handle deterministic dynamics, due to the reliance on nondegenerate noise in estimating the entropies $H^k_t$. This is drawback is shared by most works using relative entropy, except perhaps when the mean field limit $Q_t$ is sufficiently regular  \cite{jabin2016mean}. Beyond relative entropy, there are many other techniques for proving propagation of chaos, some of which work in the deterministic setting. However, so far, the sharp rate of local propagation of chaos has  been obtained only for systems with noise, and only by the analysis of relative entropy (or its cousin, the chi-squared divergence \cite{hess2023higher}).

\subsubsection{Non-exchangeable systems}

The literature on interacting particle systems with heterogeneous interactions has exploded in the past decade, motivated by a wide range of disciplines in which network structures play an important role and cannot be reasonably neglected \cite{newman2018networks}. We focus the subsequent discussion on the mathematical study of continuous-time and mostly stochastic dynamics, of the form \eqref{intro:SDE-n}.

An early thread of this literature focused on the question of \emph{universality}: For what (sequences of) $n \times n$ interaction matrices $\xi$ does the $n$-particle system \eqref{intro:SDE-n} converge to the usual McKean-Vlasov limit \eqref{intro:SDE-MV-MFcase} as $n\to\infty$? This was answered first in \cite{bhamidi2019weakly,delattre2016note,coppini2020law} for \Erdos\ graphs $G(n,p)$ (and other exchangeable random graph models), where $\xi$ is $1/np$ times the adjacency matrix, culminating with \cite{oliveira2019interacting} obtaining the minimal denseness condition $np \to \infty$.
More generally, if $\xi$ is sufficiently dense and has row sums close  to 1, we should expect to achieve the usual McKean-Vlasov limit. Our results quantify this,   at least when row sums equal 1, because the right-hand sides of \eqref{intro:avg1way}, \eqref{intro:avg}, and \eqref{intro:Gaussian-avg} can be interpreted as measuring the denseness of $\xi$.

For many (sequences of) interaction matrices the McKean-Vlasov equation is not the correct large-$n$ limit. Alternative limits have been derived using various concepts from the theory of dense graph limits. Some representative papers in this direction include  \cite{medvedev2014nonlinear,chiba2016mean,medvedev2018continuum,bayraktar2023graphon,bet2024weakly}, which take advantage of the well-developed theory of  graphons  \cite{lovasz2012large} and their $L^p$ extensions \cite{borgs2019I,borgs2019II}.
Only recently have some papers \cite{bayraktar2022stationarity,bris2022note} made this large-$n$ analysis uniform in time, which is more difficult in the absence of exchangeability perhaps due to the lack of a gradient flow structure, though see \cite{peszek2023heterogeneous} for an interesting new perspective on the latter point.
Nowadays, the large-$n$ limit theories for non-exchangeable systems are evolving hand-in-hand with modern graph limit theories. The recent papers  \cite{kuehn2022vlasov,gkogkas2022graphop} build on operator-theoretic graph limits (``graphops") proposed in \cite{backhausz2022action} which unify dense and sparse regimes. The very recent \cite{ayi2024mean} builds on hypergraphons originating from \cite{elek2012measure}. The paper \cite{jabin2021mean} even developed its own new tailor-made notion of extended graphons.

Sparse interactions behave differently, such as those induced by graphs with bounded degree. There is not enough averaging taking place in \eqref{intro:SDE-n}, and we cannot expect nearby particles to become asymptotically independent.
A completely different phenomenology arises in the sparse regime, and much remains to be understood. The papers \cite{oliveira2020interacting,lacker2023local} show how to derive large-$n$ limits using the notion of \emph{local weak convergence}, a.k.a.\ \emph{Benjamini-Schramm convergence} \cite{benjamini2011recurrence}, a graph limit theory well-suited to sparse settings. The companion paper \cite{lacker2023marginal} of \cite{lacker2023local} identifies a new substitute for the McKean-Vlasov equation in the sparse regime; notably, even under the constant row sum condition \eqref{intro:stochasticmatrix}, one does not get the usual McKean-Vlasov equation in the large-$n$ limit.
 See \cite{ramanan2023interacting} for a survey including more recent progress.

All of the above perspectives require some asymptotic structural assumption on the interaction matrix $\xi$, in contrast with our decidedly non-asymptotic approach relying on the independent projection \eqref{intro:indepproj}. The recent paper \cite{jackson2023approximately} also employs the independent projection for a non-asymptotic analysis of mean field approximations, but in the context of stochastic control problems, and focusing on global estimates on the full $n$-particle system.

The important recent paper \cite{jabin2021mean} warrants further discussion.
As discussed in Section \ref{se:intro-nonexchangeable}, we can see \cite{jabin2021mean} as addressing two separate problems. The first problem is the  approximation of the original $n$-particle system \eqref{intro:SDE-n} by the independent projection \eqref{intro:indepproj}, and the second is to identify large-$n$ limits for the independent projection.
The first problem was straightforward to address in \cite{jabin2021mean}, using the Lipschitz assumption on the interaction kernel, and being content with suboptimal rates of approximation \cite[Proposition 2.2]{jabin2021mean}.
The second problem was the main focus and difficulty in  \cite{jabin2021mean}, due to the minimal denseness assumptions imposed on $\xi$.
In contrast, our main goal is a sharp quantitative solution of the first problem. 
The main assumptions on $\xi$ are different as well. In \cite{jabin2021mean} it was assumed that the row sums \emph{and column sums} of $\xi$ are $O(1)$ and that the maximal entry is $o(1)$. We share their requirement of bounded row sums, but our estimates of the maximal entropy $\widehat{H}^k_t$ do not require bounded column sums, and our estimates of the average entropy $\overline{H}^k_t$ do not require $\max_{ij}\xi_{ij}$ to be small.
The proofs in \cite{jabin2021mean} also adopted a hierarchical technique, developed further in \cite{jabin2023mean}, but of a completely different nature from ours, not dealing directly with the marginals of \eqref{intro:SDE-n} but rather using tree-indexed observables modeled on the notion of homomorphism density from graph theory.

\subsection{Outline of the rest of the paper}

The next Section \ref{sec:main.results} gives a detailed presentation of our most general setting and main results. Section \ref{sec:graph.examples} illustrates how they specialize for certain natural classes of interaction matrices $\xi$.
The proofs of the main results occupy the remaining sections. Section \ref{se:hierarchyproofs} proves the main bound \eqref{intro:Hv-bound} in which the percolation process $\X_t$ first appears.
Section \ref{se:percolationproofs} then explains how to estimate various expectations of functions of $\X_t$, which are put to use in Section \ref{se:proofs-concrete} in order to derive our most user-friendly bounds which were summarized in Section \ref{se:intro:summary}. The final Section \ref{se:gaussianproofs} carries out the calculations for a Gaussian example presented in Section \ref{se:gaussian}.

\subsection*{Acknowledgement}

D.L.\ is grateful to Louigi Addario-Berry for discussions on first-passage percolation and for pointing out its equivalence with the Markov process $(\X_t)$.

\section{Main results} \label{sec:main.results}

\subsection{Notation}

The number $n$ of particles is fixed throughout the paper, as is the dimension $d$.
Let $[n]:=\{1,2, \dots, n\}$. For $v\subset [n]$, we denote  the cardinality of $v$ by $|v|$. 

Given any topological space $E$, let $\P(E)$ be the space of Borel probability measures on $E$. For  $\mu\in\P(E)$ and measurable function $\phi$ on $E$, let $\langle \mu,\phi\rangle$ denote  the integral $\int_{E}\phi\,  d\mu$ when well defined.
For $Q \in \P(E^n)$, let $Q^v \in \P(E^v)$ denote the marginal law of the coordinates in $v$.
For brevity, when $ v =\{j\} $ is a singleton we omit the bracket and write simply $Q^j$.

For any $\mu,\nu\in \P(E)$, the relative entropy is defined as usual by
\[
H(\nu \, |\, \mu) := \int_E \frac{d \nu}{d \mu } \log \frac{d \nu}{d \mu } \, d \mu, \text{ if } \nu \ll \mu, \quad H(\nu \, |\, \mu) =\infty \ \text{if } \nu \nll \mu.
\]
For $\mu,\nu \in \P(\R^k)$, the relative Fisher information between $ \mu $ and $ \nu $ is defined as usual by
\[I(\nu \, |\, \mu) := \int_{\R^k} \Big| \nabla \log \frac{d\nu}{d \mu} \Big|^2 d \nu,
\]
where we set $ I(\nu \, |\, \mu) := \infty $ if $ \nu \nll \mu $ or if the weak gradient $ \nabla \log d \nu /d \mu $ does not exist in $ L^2(\nu) $. The Wasserstein distance is defined by
\begin{align*}
	\W_2(\mu, \nu) := \inf_\pi \left(\int_{\R^k \times \R^k} |x - y|^2 \pi (dx, dy)\right) ^{1/2},
\end{align*}
where the infimum is taken over all $\pi \in \P(\R^k \times \R^k)$ with marginals $\mu$ and $\nu$.

We will use some less standard notation for probability measures on continuous path space.
For $Q$ in $\P(C([0,\infty);\R^d))$ or $\P(C([0,T];\R^d))$ and $0 \le t \le T$, let $Q_t \in \P(\R^d)$ denote the time-$t$ marginal, i.e., the pushforward of $Q$ by the evaluation map $x \mapsto x_t$.
Let $Q_{[t]} \in \P(C([0,t];\R^d))$ denote the law of the path up to time $t$, i.e., the pushforward of $Q$ by the restriction map $x\mapsto x|_{[0,t]}$.
For $Q \in \P( C([0,\infty);(\R^d)^n) )$ and $v \subset [n]$ we will write $Q^v_t$ for the time-$t$ marginal law of the coordinates in $v$ under $Q$, and we define $Q^v_{[t]}$ similarly.

\subsection{The interacting particle system}
The $n$-particle system $X_t=(X^1_t,\ldots,X^n_t)$ we study is governed by the following system of stochastic differential equations (SDEs):
\begin{equation}  \label{eq.SDE.n.particle.sys}
	d X^{i}_t = \Big(b_0^i(t, X^{i}_t) + \sum_{j \neq i} \xi_{ij} b^{ij}(t, X^{i}_t, X^{j}_t)\Big) dt + \sigma dB^i_t, \quad i \in [n],
\end{equation}
where $ B^1, \ldots, B^n $ are independent $ d $-dimensional Brownian motions. Let $ P \in \P( C([0,\infty);(\R^d)^n) )$ denote the law of a weak solution $ (X^1, \dots, X^n) $ of \eqref{eq.SDE.n.particle.sys}, started from some given initial distribution $P_0 $. 
Here $ \xi$ is an $n \times n$ matrix with non-negative entries and zeros on the diagonal.
The functions $b_0^i : [0, \infty) \times \R^d \to \R^d$ and $b^{ij}: [0, \infty) \times \R^d \times \R^d \to \R^d$ are Borel measurable, with more precise assumptions given below. 
Note that \eqref{eq.SDE.n.particle.sys} generalizes the model \eqref{intro:SDE-n} by allowing $b_0^i$ and $b^{ij}$ to be heterogeneous, which causes no additional difficulty in our arguments. The assumptions below on these functions are all uniform with respect to $(i,j)$, so that we may safely interpret $\xi_{ij}$ as capturing solely the \emph{scale} or \emph{strength} of the interaction between particles $i$ and $j$, viewed as distinct from the detailed shape of the interaction function $b^{ij}$.

Following the terminology of  \cite{lacker2023independent}, we define the \emph{independent projection} as the solution $Y_t=(Y^1_t,\ldots,Y^n_t)$ to the following McKean-Vlasov equation
	\begin{equation}
		\left\{\begin{aligned}
			 \label{eq.independent.projection.sys}
			&d Y^{i}_t =  \Big(b_0^i(t, Y^{i}_t) +  \sum_{j \neq i} \xi_{ij} \langle Q^j_t, b^{ij}(t, Y^{i}_t, \cdot)\rangle \,  \Big) dt + \sigma dB^i_t, \quad i \in [n] \\
			& Q_t = \mathrm{Law}(Y_t), \quad t\ge 0
		\end{aligned}\right.
	\end{equation} 
We write $ Q \in \P( C([0,\infty);(\R^d)^n) )$ for the law of a weak solution $ (Y^1, \ldots, Y^n) $ of \eqref{eq.independent.projection.sys}, initialized from some product measure $Q_0=Q^1_0 \otimes \cdots \otimes Q^n_0$. When the SDE \eqref{eq.independent.projection.sys} is well-posed, the coordinates $Y^1,\ldots,Y^n$ are independent, because the drift of $Y^i$ depends only on $Y^i$ and not the other coordinates.
Our main results will be estimates on the relative entropies 
\begin{equation}
H_t(v) := H(P^{v}_t \, | \, Q^{v}_t), \qquad H_{[t]}(v) := H(P^{v}_{[t]} \, | \, Q^{v}_{[t]}), \quad v \subset [n], \ t \ge 0. \label{def:entropies}
\end{equation}
Recall that for any  $t \ge 0$ and $v \subset [n]$ we write $P^{v}_{[t]} \in \P(C([0,t];(\R^d)^v))$ for the law of the path up to time $t$ of the coordinates in $v$ under $P$; that is, for the law of $(X^i_s)_{s \in [0,t], \, i \in v}$. Similarly, we write $P^v_t$ for the time-$t$ law of $(X^i_t)_{i \in v}$. We write $Q^v_{[t]}$ and $Q^v_t$ for the analogous marginal laws under $Q$.

\subsection{Assumptions and examples}

Our first set of assumptions will drive our estimates on the path-space entropies $H_{[t]}(v)$, for bounded time intervals.
Following \cite{lacker2022hierarchies}, rather than making direct assumptions on $(b_0^i,b^{ij})$, we make the following implicit assumptions which emphasize the key ingredients in the method.

\begin{manualassum}{A} \label{asssump:common}

{\ }
Let $T \in [0,\infty]$.
	\begin{enumerate}[(i)] 
		\item Well-posedness: The SDEs \eqref{eq.SDE.n.particle.sys} and \eqref{eq.independent.projection.sys} admit unique in law weak solutions from any initial distribution, in the time interval $[0,T]$.
		\item Square integrability of interaction function: 
		\begin{align*}
			M := \max_{i,j\in[n]} \esssup_{t\in (0,T)} \int_{(\R^d)^n} \big| b^{ij}(t, x_i, x_j) - \big\langle Q^{j}_t, b^{ij}(t, x_i, \cdot)\big\rangle \big|^2 P_t(dx) < \infty.
		\end{align*}
		\item Transport-type inequality: There exists $0 < \gamma < \infty$ such that
		\begin{align} \label{cond.transport.type.ineq}
			\big| \big< \nu - Q^i_t, b^{ij}(t, x, \cdot) \big> \big|^2 \leq \gamma H( \nu \, | \,Q^i_t ), \quad \forall i\in[n],\,  x \in \R^d,\,\nu \in \P(\R^d),\, t \in (0,T).
		\end{align}
		\item
		The  $n\times n$ matrix $\xi =\left(\xi_{ij}\right)_{i,j=1}^{n}$ has nonnegative entries, zero diagonal entries $\xi_{ii}=0$, and bounded row sums:
		\begin{equation} \tag{rows}
			\max_{1\leq i\leq n} \sum_{j=1}^n \xi_{ij} \leq 1. \label{cond:row.sum}
		\end{equation}
	\end{enumerate}	
\end{manualassum}

\begin{remark} \label{re:row.sum}
The right-hand side of \eqref{cond:row.sum} can be generalized from 1 to any other constant, say $c > 0$. By changing the interaction matrix to $\xi/c$ and the interaction functions to $c b^{ij}$, we can reduce to the case \eqref{cond:row.sum}, with the constants $(\gamma,M)$ scaled accordingly to $(c^2\gamma,c^2M)$.
The restriction that $\xi$ has nonnegative entries is made purely to avoid notational clutter, and it can be removed as long as $\xi_{ij}$ is replaced by $|\xi_{ij}|$ in \eqref{cond:row.sum} and in all of the results to follow.
\end{remark}

\begin{example}[Bounded drift] \label{ex:bounded}
Suppose $b^{ij}$ is bounded  and $b_0^i$ is such that the SDE $dZ_t^i = b_0^i(t,Z_t^i)dt + \sigma dB_t^i$ is unique in law from any initial position (which holds, e.g., if $b_0^i$ is bounded or Lipschitz). Then Assumption \ref{asssump:common} holds. The well-posedness of the independent projection follows from known arguments for McKean-Vlasov equations \cite[Theorem 2.5]{lacker2018strong} or \cite[Theorem 2]{mishura2020existence}. Conditions (ii) and (iii) hold with $\gamma = 2\max_{ij}\||b^{ij}|^2\|_\infty$ and $M=2\gamma$.
\end{example}

\begin{example}[Lipschitz drift] \label{ex:lipschitz}
Let $T < \infty$.
Suppose that $b_0^i$ and $b^{ij}$ are Lipschitz, uniformly in $(i,j)$, and that the initial laws $Q_0$ and $P_0$ admit finite second moments. Assume also the following transport inequality: there exists $0 \le \gamma_0 < \infty$ such that
\begin{align*}
\W_2^2(\nu,Q^i_0) \le \gamma_0 H(\nu\,|\,Q^i_0), \quad \forall i \in [n], \ \nu \in \P(\R^d).
\end{align*}
then Assumption \ref{asssump:common} holds. The well-posedness of the independent projection is a straightforward consequence of classical results on McKean-Vlasov equations \cite[Proposition 4.1]{lacker2023independent}. It can be shown exactly as in \cite[Corollary 2.7]{lacker2022hierarchies} that parts (ii,iii) of Assumption \ref{asssump:common} hold, with explicit ($n$-independent) bounds on $\gamma$ and $M$.
\end{example}

\begin{remark}
Examples \ref{ex:bounded} and \ref{ex:lipschitz} do not exhaust the scope of Assumption \ref{asssump:common}. We refer to  \cite[Section 2B]{lacker2022hierarchies} for further discussion, particularly for the most unusual condition \eqref{cond.transport.type.ineq}. In particular, we highlight Remarks 2.12 and 4.5 in \cite{lacker2023sharp} for an explanation of how the arguments extend with minimal effort to kinetic (second-order) models. We could also handle path-dependent coefficients, except in our uniform-in-time results.
\end{remark}

Our second and stronger set of assumptions will allow us to obtain uniform-in-time estimates, but (unsurprisingly) only for the time-marginal entropy $H_t(v)$. The following is adapted from \cite{lacker2023sharp}:

\begin{manualassum}{U} \label{asssump:uniform.in.time}
{ \ }
	\begin{enumerate}[(i)] 
		\item Assumption \ref{asssump:common} holds with $ T = \infty $. 
		\item Log-Sobolev inequality (LSI): There exists a constant $0 < \eta < \infty$ such that 
		\begin{align*}
			H(\nu \, |\, Q^i_t) \le \eta I(\nu \,|\, Q^i_t), \quad \forall \nu \in \P(\R^d), \,\, i \in [n], \,\, t \ge 0. 
		\end{align*} 
		\item High-temperature/large noise: It holds that $\sigma^4 > 24 \eta \gamma $.
		\item 
		For each $(t,x) \in [0, \infty) \times \R^d$ and $i \in [n]$, we have $b^{ij}(t, x, \cdot) \in L^1(\R^d,Q^i_t)$.  The functions $b_0^i$ and $(t,x) \mapsto \langle Q^i_t, b^{ij}(t, x, \cdot)\rangle$ are locally bounded, for each $i \in [n]$.
		Finally, for each $p, t > 0$, 
		\begin{align} \label{assump.unit.int}
		\begin{split}
			\max_{i,j \in [n], i \neq j}\int_0^t \int_{(\R^d)^n} \left(\big| b^{ij}(s, x^i, x^j) \big|^p + \big| \langle Q^j_t, b^{ij}(s, x^i, \cdot) \rangle \big|^p \right) P_s (dx)ds &< \infty,\\
		\max_{i,j \in [n], i \neq j} \sup_{s \in [0, t)} \int_{(\R^d)^n} \left(\big| b_0^i(s, x^i) \big|^2 + \big| b^{ij}(s, x^i, x^j) \big|^2\right) P_s (dx) &< \infty.
		\end{split}
		\end{align}
	\end{enumerate}	
\end{manualassum}

The essential parts of Assumption \ref{asssump:uniform.in.time} are parts (i--iii). As in \cite[Assumption (E)]{lacker2023sharp}, the condition (iv) is purely technical, used only qualitatively to justify an entropy estimate; the values of the integrals play no role in our quantitative bounds. The high-temperature constraint in (iii) is important, as explained in \cite[Remark 2.2]{lacker2023sharp}, and uniform-in-time propagation of chaos can fail for small $\sigma^4$. We have not tried to optimize the constant $24\eta\gamma$, and we certainly do not expect to improve upon \cite{lacker2023sharp} in which the threshold was already likely suboptimal, as it could not reach all the way to criticality in the Kuramoto model \cite[Example 2.10]{lacker2023sharp}.

\begin{example}[Convex potentials] \label{ex:convex}
Assume $ b_0^i(t,x) = -\nabla U(x) $ and $ b^{ij}(t,x,y) = - \nabla W(x - y) $,   where $U$ and $W$ are twice continuously differentiable functions satisfying the following:
\begin{itemize} 
	\item $W$ is convex and $U$ is strongly convex , i.e., there exists some $\lambda > 0$ such that $ 	\nabla^2 U \succeq \lambda I$.
	\item $\nabla W$ is bounded, and both $\nabla U$ and $\nabla W$ are Lipschitz. 
\end{itemize}
Suppose the interaction matrix $\xi$ is symmetric, and $P_0$  admits finite  moments of all orders. Assume also the following log-Sobolev inequality: there exists $0 \le \eta_0 < \infty$ such that
\begin{equation*}
H(\nu\,|\,Q^i_0) \le \eta_0 I(\nu\,|\,Q^i_0), \quad \forall i \in [n], \ \nu \in \P(\R^d).
\end{equation*}
Then Assumption \ref{asssump:uniform.in.time} is satisfied, with $\eta = \max(\eta_0/4, \sigma^2/\lambda)$, $\gamma = 2\||\nabla W|^2\|_\infty$, and $M =2\gamma$. The proof is a straightforward modification of that of \cite[Corollary 2.7]{lacker2023sharp}, and we give some details in Section \ref{se:proofs:convexpotentials}.
We doubt that the boundedness condition on $\nabla W$ is necessary, but to relax it would require showing $\max_{i \in [n]}\sup_{t \ge 0}\E|X^i_t|^2 < \infty$, uniformly in $n$, which seems to be a delicate task in the absence of exchangeability.
\end{example}

\begin{example}[Small interactions on the torus] \label{ex:torus}
Suppose the state space is the flat torus $\T^d = \R^d /\Z^d$ instead of $\R^d$.  Take $b_0^i \equiv 0$ and $b^{ij}(t,x, y) = K(x - y)$ for some Lipschitz $K: \R^d \to \R^d$.
Let $\lambda \ge 1$, and assume $Q^i_0$ admits a smooth density bounded in $[\lambda^{-1},\lambda]$, for each $i \in [n]$.
 Finally, assume that $\diver K$ is small in the sense that 
	\begin{equation} \label{eq:K.smallness.assum}
		\Vert \diver K \Vert_\infty  <   2\sigma^2\pi^2 \big/ \big(1 + \sqrt{2 \log \lambda}  \big).
	\end{equation} 
	Then  Assumption  \ref{asssump:uniform.in.time} is satisfied. The proof is a modification of that of \cite[Corollary 2.9 and Lemma 5.1]{lacker2023sharp}, and we give the details in Section \ref{se:proofs:torus}. We can trivially take $\gamma = 2\||K|^2\|_\infty$ and $M= 2\gamma$ by Pinsker's inequality, and the constant $\eta$  can be taken to be
	\begin{equation*}
	\eta = \frac{\lambda^2}{8\pi^2}\bigg(1-\frac{\sqrt{2 \log \lambda} \Vert \diver K \Vert_\infty}{ 2( 2\sigma^2\pi^2 - \Vert \diver K \Vert_\infty ) }\bigg)^{-1},
	\end{equation*}
	which is simply $\eta=\lambda^2/8\pi^2$ if $K$ is divergence-free. 
\end{example}

\subsection{The first-passage percolation bound} \label{sec:CTMC.main.result}

In this section we describe our most general estimates on the relative entropies $H_t(v)$ and $H_{[t]}(v)$ defined in \eqref{def:entropies}. It is stated in Proposition \ref{pr:CTMC} below in terms of what we call \emph{the percolation process} associated with the matrix $R \in \mathcal{R}$. Here and throughout,  we denote 
\begin{align} \label{eq:def.control.set}
 \mathcal{R} := \bigg\{ R \in \R^{n \times n}_+ : \sum_{j=1}^n \frac{\xi_{ij}^2 }{R_{ij}} \le 2, \ \forall i \in [n]\bigg\},
\end{align}
with the convention that $\xi_{ij}^2/R_{ij}:=0$ if $\xi_{ij}^2=R_{ij}=0$.
We will make use of the following quantities:
\begin{align} \label{eq:def.C(v).D(v)}
	\CC(v) := \frac{M}{\sigma^2} \sum_{i \in v} \bigg(\sum_{j \in v} \xi_{ij} \bigg)^2, \quad \A_{v \to j}^{R} := \frac{2\gamma}{\sigma^2} \sum_{i \in v} R_{ij},  \qquad \forall v\subset [n], \ j \in [n] \setminus v.
\end{align}
The percolation process is a continuous-time Markov chain $\X^{R}$ on the state space $2^{[n]}$ of subsets of $[n]$. Its rate matrix $\A^{R}(v,u)$ is defined for $u,v \in 2^{[n]}$ by
\begin{equation}
\A^{R}(v,u) = \begin{cases}\A^{R}_{v \to j} &\text{if } u=v \cup \{j\},  \text{ for some } j \in [n] \setminus v \\  - \sum_{j \notin v} \A^{R}_{v \to j} &\text{if } u=v \\   0 &\text{otherwise}. \end{cases} \label{rate.matrix}
\end{equation}
The key structural feature of $\A^{R}$ is that it is a rate matrix, in the sense that $\sum_v \A^{R}(u,v)=0$ for each $u$, and the off-diagonal entries $v \neq u$ are nonnegative. We naturally view $\A^{R}$ as an operator acting on functions $F : 2^{[n]}\to\R$, 
\[
\A^{R} F(v) = \sum_{u \in 2^{[n]}}\A^{R}(v,u)F(u) = \sum_{j \notin v}\A^{R}_{v\to j}\big(F(v \cup \{j\}) - F(v)\big).
\]
Let $\E_v[\cdot]$ denote expectation under the initialization $\X^{R}_0=v$, and note the stochastic representation $\E_v[F(\X^{R}_t)] = e^{t\A^{R}}F(v)$ for $t \ge 0$.
We prove the following in Section \ref{se:hierarchyproofs}:

\begin{proposition} \label{pr:CTMC}
Assume  $H_0([n]) < \infty$.
\begin{enumerate}[(i)]
\item If Assumption \ref{asssump:common} holds for $T < \infty$, then
	\begin{align}
		H_{[T]}(v) &\le\inf_{R \in \mathcal{R}}\E_v\bigg[H_0(\X^{R}_T) + \int_0^T \CC(\X^{R}_t)\, dt\bigg].  \label{ctmc.hierarchy.1}
	\end{align}
	\item If Assumption \ref{asssump:uniform.in.time} holds, then for all $t > 0$, 
	\begin{align}  \label{ctmc.hierarchy.uit.1}
		H_t(v) & \le \inf_{R \in \mathcal{R}} \E_v \left[ e^{-\sigma^2t/4\eta} H_0(\X^{R}_t) + \int_0^t e^{-\sigma^2s/4\eta } \CC(\X^{R}_s) \, ds\right].
	\end{align} 
\end{enumerate} 
\end{proposition}

\subsection{Concrete bounds} \label{se:concrete}

In this section we give an assortment of more practical bounds on the entropies $H_{[t]}(v)$ and $H_t(v)$ which we deduce from the general Proposition \ref{pr:CTMC}. Proofs are given in Section \ref{se:proofs-concrete}. Here we emphasize results which hold for general matrices $\xi$, and Section \ref{sec:graph.examples} will specialize the results to various classes of $\xi$. We start with the maximum entropy over sets $v \subset [n]$ of a given size.
For $k \in [n]$ and $t \ge 0$ define
\begin{align*}
\widehat{H}^{k}_{[t]} = \max_{|v| =k} H_{[t]}(v), \qquad \widehat{H}^{k}_t = \max_{|v| =k} H_t(v).
\end{align*}
Throughout the section we will make use of the following parameters:
\begin{equation}
\delta:=\max_{i,j \in [n]}\xi_{ij}, \qquad \delta_i:=\max_{j\in [n]}\xi_{ij} \label{def:delta=maxentry}
\end{equation}
Our first result on maximum entropy was announced at \eqref{intro:max}:

\begin{theorem}[Maximum entropy] \label{theorem:max}
Suppose the following initial chaoticity assumption holds:
\begin{align} \label{cond.chaotic.initial.condition}
\widehat{H}^{k}_0 \leq C_0 (\delta k + 1)(\delta k)^2, \quad \text{for all } k \in [n],
\end{align}
for some   constant $C_0$.
If Assumption \ref{asssump:common} holds for $T < \infty$, then 
\begin{equation} \label{eq:max-mainbound}
\widehat{H}^{k}_{[T]} \le C (\delta k + 1)(\delta k)^2, \quad \text{for all } k \in [n],
\end{equation}
for a constant  $C$ depending only on $(C_0, \gamma, M,\sigma, T)$. 
If Assumption \ref{asssump:uniform.in.time} holds, then $\sup_{t \ge 0}\widehat{H}^{k}_t$ is bounded by the same quantity as in \eqref{eq:max-mainbound}, with a constant  $C$ depending only on $(C_0, \gamma, M,\sigma, \eta)$. 
\end{theorem}

In the regime $k=O(1/\delta)$, the bound \eqref{eq:max-mainbound} becomes $\widehat{H}^{k}_{[T]}  \lesssim (\delta k)^2$, which cannot be improved. Indeed, in a Gaussian example in Remark \ref{re:sharpness-max} we obtain a matching lower bound.
Moreover, the initial chaoticity assumption \eqref{cond.chaotic.initial.condition} is sharp, in the sense that replacing it with a stronger assumption does not lead to a stronger conclusion in \eqref{eq:max-mainbound}. Indeed, in the same Gaussian example we have i.i.d.\ initial positions $P_0=Q_0$, so $\widehat{H}^{k}_0=0$, and nonetheless $\widehat{H}^{k}_t \asymp (\delta k)^2$ for any $t > 0$.
See \cite{lacker2022quantitative} for a natural class of non-trivial examples of exchangeable distributions $P_0$ and $Q_0$ on $(\R^d)^n$, with $Q_0$ being a product measure, such that $H(P_0^v \,|\,Q_0^v) = O((k/n)^2)$ for all $v \subset [n]$ with $|v|=k$; if $\delta \gtrsim 1/n$, as it is in all of our examples, then \eqref{cond.chaotic.initial.condition} holds.

Our next results pertain to the average entropy, which behaves quite differently from the maximum and can be small even when $\delta$ is not.
For $k \in [n]$ and $t \ge 0$ define
\begin{align*}
	\overline{H}^{k}_{[t]} := \frac{1}{\binom{n}{k}} \sum_{|v|=k} H_{[t]}(v), \qquad \overline{H}^k_t := \frac{1}{\binom{n}{k}} \sum_{|v|=k} H_t(v).
\end{align*}  
That is, we are averaging over all $v \subset [n]$ of cardinality $k$. 
For some of these results, we will require an additional assumption that the column sums (not just row sums) of $\xi$ are bounded by 1:
\begin{equation} \tag{columns}
\max_{j \in [n]}\sum_{i=1}^n \xi_{ij}\le 1. \label{asmp:columnsum}
\end{equation}
The following was announced at \eqref{intro:avg1way}, and is a corollary of the subsequent Theorem \ref{theorem:avg-markov}:

\begin{corollary}[Average entropy] \label{co:avg.1way} 
Assume \eqref{asmp:columnsum} holds.
Suppose the following initial chaoticity assumption holds: 
	\begin{equation} 
		 \widehat{H}^k_0 \leq C_0 (\delta k + 1)  \frac{k^2 }{n} \sum_{i=1}^n\delta_i^2 , \quad \text{for all } k \in [n], \label{asmp:avg1way-init}
	\end{equation}
	for some finite constant $C_0$.
	If Assumption \ref{asssump:common} holds for $T < \infty$, then 
\begin{equation}  \label{eq:avg.1way.asym}
\overline{H}^{k}_{[T]} \leq C(\delta k + 1)\frac{k^2 }{n} \sum_{i=1}^n\delta_i^2 , \quad \text{for all } k \in [n],  
\end{equation}
for a constant  $C$ depending only on $(C_0, \gamma, M,\sigma, T)$. 
If Assumption \ref{asssump:uniform.in.time} holds, then $\sup_{t \ge 0}\overline{H}^{k}_t$ is bounded by the same quantity as in \eqref{eq:avg.1way.asym}, with a constant  $C$ depending only on $(C_0, \gamma, M,\sigma, \eta)$.  
\end{corollary}

It is worth stressing that there is a mismatch between the initial chaoticity assumption \eqref{asmp:avg1way-init}, imposed on the \emph{maximum} entropy $\widehat{H}^k_0$, and the conclusion \eqref{eq:avg.1way.asym} for the average entropy $\overline{H}^k_{[T]}$.
If the assumption \eqref{asmp:avg1way-init} is weakened by changing $\widehat{H}^k_0$ to $\overline{H}^k_0$, it is not clear if \eqref{eq:avg.1way.asym} still holds.

As discussed around \eqref{intro:avg1way-nonunif}, we can remove the assumption of bounded column sums if we instead work with \emph{weighted} averages. 

\begin{theorem}[Weighted average entropy] \label{theorem:avg-markov}
Suppose a vector $\pi \in \R^n$ has nonnegative entries and satisfies $\pi^\top\xi \le \pi^\top$ coordinatewise, as well as $\sum_{i=1}^n\pi_i \le 1$. 
Suppose the following initial chaoticity assumption holds: 
\begin{equation} \label{asmp:initial-markov}
\widehat{H}^k_0 \le C_0( \delta k + 1) k^2 \sum_{i=1}^n \pi_i  \delta_i^2 , \quad \text{ for all } k \in [n],
\end{equation}
for some finite constant $C_0$.
Suppose we are given $k \in [n]$ and a random element $\V$ of $\{v \subset [n] : |v| \le k\}$ such that $\PP(i \in \V) \le k\pi_i$ for all $i \in [n]$. 
If Assumption \ref{asssump:common} holds for $T < \infty$, then
\begin{equation} \label{eq:markov-mainbound}
\E[H_{[T]}(\V)] \le C(\delta k + 1)k^2 \sum_{i=1}^n \pi_i  \delta_i^2 ,
\end{equation}
for a constant  $C$ depending only on $(C_0, \gamma, M,\sigma, T)$. 
If Assumption \ref{asssump:uniform.in.time} holds, then  $\sup_{t \ge 0}\E[H_t(\V)]$ is bounded by the same quantity as in \eqref{eq:markov-mainbound}, with a constant  $C$ depending only on $(C_0, \gamma, M,\sigma, \eta)$. 
\end{theorem}

There are two main examples to have in mind for Theorem \ref{theorem:avg-markov}. The first is the case of uniform averaging, where $\pi_i=1/n$ for all $i \in [n]$. The condition $\pi^\top \xi \le \pi^\top$ then means that the column sums of $\xi$ are all bounded by 1. The random set $\V$ can be taken to be uniform over $\{v \subset [n] : |v|=k\}$, meaning $\E[H_{[T]}(\V)]=\overline{H}^k_{[T]}$, and we thus deduce Corollary \ref{co:avg.1way} as an immediate corollary of Theorem \ref{theorem:avg-markov}.
The second example is when $\xi$ is the transition matrix of a Markov chain on $[n]$, and $\pi$ is an invariant measure, meaning $\pi^\top \xi = \pi^\top$ and $\sum_i\pi_i=1$. Assume as usual that $\xi_{ii}=0$ for all $i$. Consider any random set $\V=\{Z_1,\ldots,Z_k\}$ (where any repeated elements are merged), where the marginals are $Z_i \sim \pi$ for each $i \in [n]$. The union bound then yields $\PP(i \in \V) \le k\pi_i$. 
This includes the case where $Z_1,\ldots,Z_k$ are i.i.d.\ $\sim \pi$, or the case where $(Z_1,\ldots,Z_k)$ is a trajectory from the Markov chain in stationarity. In the latter case,
\begin{equation*}
\PP(\V=\{i,j\}) = \pi_i\xi_{ij} + \pi_j\xi_{ji}
\end{equation*}
for each $i,j \in [n]$, which shows that the claim \eqref{intro:avg1way-nonunif} follows from Theorem \ref{theorem:avg-markov}.

Returning to uniform (unweighted) averages, the bound of Corollary \ref{co:avg.1way} can be pushed a bit further to sharpen the row-max dependence to certain row-averages. This result was announced at \eqref{intro:avg} in the case of symmetric interaction matrix $\xi$:

\begin{theorem} [Sharper average entropy] \label{theorem:avg.2way.asym}
Assume \eqref{asmp:columnsum} holds. 
Suppose the following initial chaoticity assumption holds: 
	\begin{equation}  \label{eq: init.chao.assum}
		\widehat{H}_0^k \leq C_0  (\delta k + 1)\bigg(\frac{ k^2}{n^2}\sum_{i,j=1}^n\xi_{ij}^2 + \frac{ k}{n}\sum_{i=1}^n\bigg(\sum_{j=1}^n (\xi_{ij}^2 +\xi_{ji}^2)\bigg)^2\bigg) , \quad \text{for all } k \in [n].
	\end{equation}
	for some finite constant $C_0$.
	If Assumption \ref{asssump:common} holds for $T < \infty$, then 
	\begin{equation} \label{eq:avg.2way.asym}
	\overline{H}^{k}_{[T]} \le C 
	 (\delta k + 1)\bigg(\frac{ k^2}{n^2}\sum_{i,j=1}^n\xi_{ij}^2 + \frac{ k}{n} \sum_{i=1}^n\bigg(\sum_{j=1}^n (\xi_{ij}^2 +\xi_{ji}^2)\bigg)^2\bigg), \quad \text{for all } k \in [n],
	\end{equation}
	for a constant  $C$ depending only on $(C_0, \gamma, M,\sigma, T)$. 
If Assumption \ref{asssump:uniform.in.time} holds, then $\sup_{t \ge 0}\overline{H}^{k}_t$ is bounded by the same quantity as in \eqref{eq:avg.2way.asym}, with a constant  $C$ depending only on $(C_0, \gamma, M,\sigma, \eta)$. 
\end{theorem}

\begin{remark} \label{rmk:sym}
The bound \eqref{eq:avg.1way.asym} is weaker than \eqref{eq:avg.2way.asym} when $\xi$ is symmetric. Indeed, 
using symmetry of $\xi$, we have the following simple estimates for the terms on the right-hand side of \eqref{eq:avg.2way.asym}:
\begin{align}
\frac{1}{n^2}\sum_{i,j=1}^n\xi_{ij}^2 & \le \frac{1}{n^2}\sum_{i,j=1}^n \delta_i^2 =\frac{1}{n }\sum_{i =1}^n \delta_i^2, \nonumber \\
\sum_{i=1}^n\bigg(\sum_{j=1}^n (\xi_{ij}^2 +\xi_{ji}^2)\bigg)^2  & = 4\sum_{i=1}^n\bigg(\sum_{j=1}^n \xi_{ij}^2\bigg)^2 \le   4\sum_{i=1}^n\bigg(\sum_{j=1}^n\delta_i\xi_{ij}\bigg)^2 \le 4\sum_{i=1}^n \delta_i^2,  \label{p(xi)} 
\end{align}
with the last step using \eqref{cond:row.sum}. Without symmetry of $\xi$, however, the left-hand side of \eqref{p(xi)} is not controlled by $\sum_i \delta_i^2$, and Corollary \ref{co:avg.1way} and \ref{theorem:avg.2way.asym} are not directly comparable.
\end{remark}

\begin{remark}
Note by convexity of relative entropy that 
\begin{equation*}
H\bigg(\frac{1}{{n \choose k}}\sum_{v \subset [n], \ |v|=k} P^v_t \,\bigg|\, \frac{1}{{n \choose k}}\sum_{v \subset [n], \ |v|=k} Q^v_t\bigg) \le \overline{H}^k_t.
\end{equation*}
The first measure on the left-hand side  is exactly the (exchangeable) law of $(X^{\pi(1)}_t,\ldots,X^{\pi(k)}_t)$, where $\pi$ is a uniformly random permutation of $[n]$, independent of $X$. Similarly for the second measure.
Hence, our bounds on $\overline{H}^k_t$ immediately apply to the symmetrized laws.
\end{remark}

\begin{remark}
A generalization of Theorem \ref{theorem:avg.2way.asym} to weighted averaging is possible, analogous to how Theorem \ref{theorem:avg-markov} generalizes Corollary \ref{co:avg.1way}. With $\pi$ and $\V$ as in Theorem \ref{theorem:avg-markov}, assume also that $\PP(i,j \in \V) \le k^2\pi_i\pi_j$ for all distinct $i,j \in [n]$. Then, if the initial chaoticity assumption \eqref{eq: init.chao.assum} is changed accordingly, we have
\begin{align*}
\E[H_{[T]}(\V)] \le C(\delta k + 1)\bigg(k^2 \sum_{i,j=1}^n \pi_i\pi_j \xi_{ij}^2 +  k\sum_{i=1}^n\pi_i\bigg(\sum_{j=1}^n (\xi_{ij}^2 +\xi_{ji}^2)\bigg)^2\bigg).
\end{align*}
This does not seem to shed new light on our main examples, so we omit the details.
\end{remark}

We lastly present \emph{setwise} bounds, for each $v \subset [n]$, without taking any average or maximum. The following was announced at \eqref{intro:pointwise-v}:

\begin{theorem}[Setwise entropy]  \label{theorem:pointwise}
Assume \eqref{asmp:columnsum} holds. 
Define
	\begin{equation}
		q_{\xi}(v) = (\delta|v|+1)\bigg(\sum_{i,j\in v}\xi^2_{ij} + \delta \sum_{i,j \in v}(\xi^\top \xi+\xi\xi^\top)_{ij}  + \delta^2|v|\bigg), \quad v \subset [n]. \label{eq:q_xi}
	\end{equation}
	Suppose the following initial chaoticity assumption holds:
	\begin{equation}\label{eq:pointwise.init.cond}
		H_{0}(v) \le  C_0 q_{\xi}(v) , \quad \text{for all } v \subset [n],
	\end{equation}
	for some finite constant $C_0$.
If Assumption \ref{asssump:common} holds for $T < \infty$, then 
\begin{equation}\label{eq:pointwise}
H_{[T]}(v) \le Cq_{\xi}(v), \quad \text{for all } v \subset [n],
\end{equation}
for a constant  $C$ depending only on $(C_0, \gamma, M,\sigma, T)$.  If Assumption \ref{asssump:uniform.in.time} holds,  then $\sup_{t \ge 0}H_t(v) \le Cq_{\xi}(v)$, with a constant  $C$ depending only on $(C_0, \gamma, M, \sigma,\eta)$.  
\end{theorem}

The quantity $q_\xi(v)$ depends not only on the size of $v$ but also on its structure, through the two summations over $v$. 
It is sharp enough to recover the maximum entropy bounds, in the sense that $q_\xi(v) \lesssim (\delta|v|+1)(\delta|v|)^2$ under assumptions \eqref{cond:row.sum} and \eqref{asmp:columnsum}, though it is not sharp enough to recover the average entropy bounds. 
That said, Theorem \ref{theorem:max} is not a corollary of Theorem \ref{theorem:pointwise}, because the  initial chaoticity assumption is stronger in the latter.

\begin{remark}
In certain cases, our entropy bounds transfer to  squared  Wasserstein distance via a Talagrand inequality. For example, in the Lipschitz setting of Example \ref{ex:lipschitz}, the measure $Q^i_{[T]}$ can be shown as in \cite{lacker2022hierarchies} to satisfy the transport inequality $\W_2^2(\cdot,Q^i_{[T]}) \le C H(\cdot\,|\,Q^i_{[T]})$ for a constant independent of $i$.
The quadratic transport inequality tensorizes \cite[Proposition 1.9]{gozlan2010transport}, and so  $\W_2^2(\cdot,Q^v_{[T]}) \le C H(\cdot\,|\,Q^v_{[T]})$ for each $v \subset [n]$, with the same constant $C$.
In the uniform-in-time case, by the Otto-Villani  theorem \cite{otto2000generalization} (see also \cite[Theorem 8.12]{gozlan2010transport}), the log-Sobolev inequality of Assumption \ref{asssump:uniform.in.time}(ii) implies the quadratic transport inequality  $\W^2_2(\cdot, Q^i_t) \le 4 \eta H(\cdot \, | \, Q^i_t)$, for all $i$ and $t$, which tensorizes in the same manner.
\end{remark}

\subsection{Reversed entropy} \label{se:reversedentropy}

Different results can be obtained for $\overleftarrow{H}_{[t]}(v) := H(Q^v_{[t]}\,|\,P^v_{[t]})$, in which the order of the arguments of relative entropy is reversed compared to $H_{[t]}(v)$ defined in \eqref{def:entropies}.
As in the prior papers \cite{lacker2022hierarchies,lacker2023sharp}, the results are somewhat easier to obtain, but only under the stronger assumption that $b$ is bounded; see \cite[Remark 4.12]{lacker2022hierarchies} for ideas on relaxing this assumption.
In our setting, under the assumption that $b$ is bounded, the reversed entropy $\overleftarrow{H}_{[t]}(v)$ satisfies all of the same bounds as in Theorems  \ref{theorem:max} and  \ref{theorem:avg.2way.asym}, with the only change being that the prefactor $(\delta k+1)$ is removed (both in the conclusions and the time-zero assumptions). The same is true for Theorem \ref{theorem:pointwise}, with the factor $\delta|v|+1$ removed from the definition of $q_\xi(v)$.
The proof is somewhat easier, with Remarks \ref{re:reversedentropy1} and \ref{re:reversedentropy2} explaining the differences.

If one is only interested in estimates on the total variation $\|P^v_{[t]} - Q^v_{[t]}\|_{\mathrm{TV}}$, then this can be derived from Pinsker's inequality regardless of the order of arguments in relative entropy. In this sense, the reversed entropy estimate yields a sharper result for total variation, by removing the $\delta k + 1$ factor. Of course, this factor is inconsequential when $k=O(1/\delta)$, for instance when $k$ is fixed as $n\to\infty$. In the mean field case where $\xi_{ij}=1/(n-1)$, we have $1/\delta=n-1$, and so it is automatic that $k=O(1/\delta)$. But $k=O(1/\delta)$ is a restriction in the non-exchangeable setting, such as in the $m$-regular graph case where it requires $k = O(m)$. In other words, we can obtain a larger \emph{size of chaos} by working with reversed entropy.

\subsection{Sharpness, and a Gaussian example} \label{se:gaussian}
In this section we discuss a simple Gaussian example. Particularly sharp estimates are available, including lower bounds, which make it a useful test case.
Consider the following $n$-particle system with linear drift:
\begin{equation}
	d X^{i}_t = \sum_{j \neq i} \xi_{ij} X^{j}_t  dt + dB^i_t, \quad X^i_0 = 0, \quad i \in [n]. \label{def:SDEgaussian}
\end{equation}
As usual, $\xi$ is a matrix with non-negative entries and zero diagonal. 
The law $ P_t $ of $(X^1_t,\ldots,X^n_t)$ is the centered Gaussian with covariance  matrix
\begin{align*}
	\Sigma_t := \int_0^t e^{ s \xi} e^{s\xi^\top}ds.
\end{align*}
The independent projection $Y_t$ defined in \eqref{eq.independent.projection.sys} satisfies
\begin{equation}
	d Y^{i}_t = \sum_{j \neq i} \xi_{ij} \E[Y^j_t]dt + dB^i_t, \quad Y^i_0 = 0, \quad i \in [n]. \label{def:SDEgaussian-indproj}
\end{equation}
Taking expectations, we find that necessarily $\E[Y^j_t]=0$ for all $j \in [n]$, and so $Y^i \equiv B^i$.
That is, the law $Q_t$ is the centered Gaussian measure with covariance matrix $tI$. Thus both $P_t$ and $Q_t$ are  centered Gaussian measures. 
A well known exact formula for the relative entropy between Gaussians gives
\begin{equation*}
H(P^v_t\,|\,Q^v_t) = \frac12 \tr \, h(t^{-1}\Sigma^v_t-I),
\end{equation*}
where we define $h(x)=x-\log(1+x)$, and we write $A^v$ for the submatrix of an $n \times n$ matrix $A$ corresponding to those rows and columns indexed by $v \subset [n]$. Noting that $h(0)=h'(0)=0$, we approximate $h(x)$ to leading order by a quadratic. In particular, letting $\rho=\|\xi\|_{\mathrm{op}}$, we will show
\begin{equation}
H(P^v_t\,|\,Q^v_t) \le  e^{6 \rho t}\tr\bigg( \frac{1}{t} \int_0^t (e^{ s \xi} e^{s\xi^\top} - I)^v\,ds \bigg)^2. \label{ineq:Gaussianupper}
\end{equation}
For small enough $t$, specifically $t \le \log(2)/2\rho$, we get a lower bound of the same order,
\begin{equation*}
H(P^v_t\,|\,Q^v_t) \ge \frac16\tr\bigg( \frac{1}{t} \int_0^t (e^{ s \xi} e^{s\xi^\top} - I)^v\,ds \bigg)^2.
\end{equation*}
We do not consider $t \le \log(2)/2\rho$ to be a significant limitation. By the data processing inequality, note that $H(P^v_{[T]}\,|\,Q^v_{[T]}) \ge H(P^v_t\,|\,Q^v_t)$ for $T \ge t \ge 0$. Hence, any lower bound on $H(P^v_t\,|\,Q^v_t)$ for small time applies also to $H(P^v_{[T]}\,|\,Q^v_{[T]})$ on any longer time horizon.

Without further simplification, the right-hand side of \eqref{ineq:Gaussianupper} admits a network-science interpretation. If $\xi$ is symmetric for simplicity, then expanding out the trace and exponential yields
\begin{align*}
\tr\bigg( \frac{1}{t} \int_0^t (e^{2s\xi} - I)^v\,ds \bigg)^2 
	&= \sum_{i,j \in v}\bigg( \frac{1}{t} \int_0^t \sum_{\ell=1}^\infty \frac{(2s)^\ell}{\ell!} (\xi^\ell)_{ij} \,ds \bigg)^2.
\end{align*}
In the language of network science \cite{estrada2008communicability}, the innermost summation is a measure of the \emph{communicability} of the nodes $i$ and $j$. The reasoning behind this terminology is that if $\xi$ is the adjacency matrix of a graph, then $(\xi^\ell)_{ij}$ counts the number of length-$\ell$ paths from $i$ to $j$, and the power series gives a weighted count over all paths between vertices in $v$. 

It is difficult to simplify the right-hand side of \eqref{ineq:Gaussianupper} in general, but after taking averages over $v$ of size $k$ we obtain a sharp estimate.
Let us stress that in the following theorem we require a bound on the spectral norm of $\xi$, rather than  row or column sums as in our results in Section \ref{se:concrete}.

\begin{theorem} 	\label{theorem:Gaussian.avg.entropy}
	Consider the Gaussian setting of this section. Define
	\begin{equation} \label{eq:quantity.for.Gaussian}
	D_T(\xi) := \sum_{i=1}^n \bigg(\sum_{m=2}^\infty \frac{T^m}{(m+1)!} (\xi^m)_{ii}\bigg)^2.
\end{equation}
For $0 < T \le \log(2) /2 \rho$ and $k \in [n]$, we have
	\begin{equation} \label{eq:Gaussian.avg.ent}
		\overline{H}^{k}_{T} \asymp  \frac{k(k-1)}{n(n-1)}\sum_{i,j=1}^n  \xi_{ij}^2 +  \frac{k (n - k)}{n (n - 1)}\bigg(D_T(\xi) + \sum_{i=1}^n \bigg(\sum_{j=1}^n  \xi_{ij}^2 \bigg)^2   \bigg),
\end{equation}
where the hidden  constants depend only on  $T$ and $\rho$. Moreover, we have
\begin{equation} \label{eq:quantity.for.Gaussian.bounded}
		\sum_{i=1}^n \bigg(\sum_{j=1}^n \xi_{ij}\xi_{ji} \bigg)^2 \lesssim
		D_T(\xi)
		\lesssim \sum_{i=1}^n \bigg(\sum_{j=1}^n \xi_{ij}^2 \bigg)^2 + \sum_{i=1}^n \bigg(\sum_{j=1}^n \xi_{ji}^2 \bigg)^2,
	\end{equation}
	and thus if $\xi$ is symmetric then the $D_T(\xi)$ term can be discarded from \eqref{eq:Gaussian.avg.ent}.
\end{theorem}

In spite of \eqref{eq:quantity.for.Gaussian.bounded}, it appears that the behavior of $D_T(\xi)$ cannot be precisely captured by any low-degree polynomial of $\xi$. In particular, the different terms appearing in \eqref{eq:quantity.for.Gaussian.bounded} may differ wildly in size for asymmetric $\xi$. For example, consider the lower-triangular matrix given by $\xi_{ij}=1$ if $j=1$ and $i \ge 2$, and $\xi_{ij}=0$ otherwise. Then $(\xi^m)_{ii}=0$ for all positive integers $m$, so $D_T(\xi)=0$. On the other hand, $\sum_{i=1}^n (\sum_{j=1}^n \xi_{ij}^2 )^2 = n-1$ and $\sum_{i=1}^n (\sum_{j=1}^n \xi_{ji}^2 )^2= (n-1)^2$.

\begin{remark}[Sharpness of Theorem \ref{theorem:avg.2way.asym}] \label{re:sharpness-avg}
Theorem \ref{theorem:Gaussian.avg.entropy} indicates that our result in the general setting in Theorem \ref{theorem:avg.2way.asym} is sharp in certain regimes. Let us focus on the case of symmetric $\xi$, for simplicity. For $k = o(n)$, the right-hand side of \eqref{eq:Gaussian.avg.ent} is  of the same order as
	\begin{align} \label{eq:Gaussian.avg.ent.lower.bound.order}
		\frac{k^2}{n^2} \sum_{i,j=1}^n \xi_{ij}^2 + \frac{k}{n} \sum_{i=1}^n \bigg(\sum_{j=1}^n \xi_{ij}^2\bigg)^2.
	\end{align}
Specializing Theorem \ref{theorem:avg.2way.asym}  to symmetric $\xi$, when $k= O(1/\delta)$ (e.g., if $k$ is a fixed constant),  the upper bound  \eqref{eq:avg.2way.asym} therein is  of the same order as \eqref{eq:Gaussian.avg.ent.lower.bound.order}. In this sense, Theorem \ref{theorem:avg.2way.asym} is sharp in the case of symmetric $\xi$. Note that in the $m$-regular graph case the quantity \eqref{eq:Gaussian.avg.ent.lower.bound.order} becomes $k^2/nm + k/m^2$.
\end{remark}

For the maximal and setwise entropy bounds given in Theorems \ref{theorem:max} and \ref{theorem:pointwise} , we must restrict the class of $\xi$ further in order to claim sharpness. This will make use of the following lower bound.

\begin{proposition} \label{pr:Gaussiancase-clique}
In the Gaussian setting of this section, for $T \le \log(2)/2\rho$ we have
\begin{equation}
H_T(v) \ge \frac{T^2}{12}\sum_{i,j \in v}\xi_{ij}^2, \quad \forall v \subset [n]. \label{pr:GaussianLB-pointwise}
\end{equation}
\end{proposition}

\begin{proposition} \label{pr:Gaussiancase-max}
In the Gaussian setting of this section, if $\xi$ has row sums bounded by 1, then 
\begin{equation}
H_T(v) \le e^{10 \rho T} \delta^2|v|^2, \quad \forall v \subset [n], \label{pr:GaussianUB-max}
\end{equation}
where we set $\delta=\max_{i,j \in [n]}\xi_{ij}$ as usual.
\end{proposition}

Note that the average of the right-hand side of \eqref{pr:GaussianLB-pointwise} over all $v \subset [n]$ with $|v|=k$ is exactly $\frac{T^2}{12}\frac{k(k-1)}{n(n-1)}\sum_{i,j \in [n]}\xi_{ij}^2$, which only recovers the first term in the bounds of Theorem \ref{theorem:Gaussian.avg.entropy}. Hence, the inequality \eqref{pr:GaussianLB-pointwise} cannot admit a matching upper bound for every $v \subset [n]$. However, it is sharp for well-connected sets $v$:

\begin{remark}[Sharpness of Theorem \ref{theorem:max}] \label{re:sharpness-max}
Suppose $v \subset [n]$ is such that $\xi_{ij}=\delta$ for all distinct $i,j \in v$. For example, this holds in the regular graph case if $v$ is a clique. Then Proposition \ref{pr:Gaussiancase-clique} becomes $H_T(v) \ge (T^2/12)\delta^2|v|(|v|-1)$, which is of the same order as the upper bound of Proposition \ref{pr:Gaussiancase-max}.
In the regime $k = O(1/\delta)$, this matches the  upper bound  $\widehat{H}^k_{[T]} = O(\delta^2 k^2)$ obtained in the general (non-Gaussian) case in Theorem \ref{theorem:max}.
\end{remark}

\section{Examples of interaction matrices} \label{sec:graph.examples}
In this section, we illustrate how the main results in Section \ref{sec:main.results} specialize in some noteworthy classes of interaction matrix $\xi$, mostly arising from simple undirected graphs.

Throughout this section, we continue to write $a \lesssim b$ to mean that $a \le Cb$ for some constant $C$ which can depend on the constants from Assumption \ref{asssump:common} but not on $n$, $k$, or $v \subset [n]$. The constant may also depend on $T$, except when Assumption \ref{asssump:uniform.in.time} holds.
While we do not index our matrix $\xi$ by $n$, in the example in this section we have in mind an asymptotic regime of a sequence of $\xi$ of size $n \times n$ with $n \to\infty$. Asymptotic notations like $k=o(n)$ should be interpreted accordingly. The number $k\le n$ of particles is in general allowed to grow with $n$, except when stated otherwise.

In each of the following examples, we take for granted that Assumption \ref{asssump:common} holds, except possibly \eqref{cond:row.sum} which we will justify when it is not obvious. This way, we may focus our attention on the effects of different choices of interaction matrix $\xi$. For the same reason we shall assume that $P_0 = Q_0$, so the time-zero assumptions such as \eqref{cond.chaotic.initial.condition} are trivially satisfied with $C_0=0$.

\subsection{The regular graph case}\label{ex:regular}
We begin by summarizing the $m$-regular graph case from Definition \ref{def:m-regular-case}, which was already discussed to some extent in the introduction with details omitted. Clearly the row and column sums of $\xi$ are all equal to 1, and $\delta=\max_{ij}\xi_{ij} =1/m$. Applying Theorem \ref{theorem:max},
\begin{equation}
\widehat{H}^{k}_{[T]} \lesssim (k/m)^2 + (k/m)^3  , \qquad \text{for } 1 \le k \le n, \label{ineq:ex:mreg-max}
\end{equation}
which is of course $O((k/m)^2)$ when $k \le m$.
Note that the classical exchangeable setting is recovered when $m=n-1$, which yields  $\widehat{H}^{k}_{[T]} \lesssim  (k/n)^2$, recovering the main result of \cite{lacker2022hierarchies}.

Estimating the average entropy, we get a slightly sharper estimate from Theorem \ref{theorem:avg.2way.asym} than from Corollary \ref{co:avg.1way} (as expected from Remark \ref{rmk:sym}). 
In fact, noting that $\delta_i=\delta=1/m$ for all $i$, Corollary \ref{co:avg.1way} simply bounds $\overline{H}^{k}_{[T]}$ by the same right-hand side as \eqref{ineq:ex:mreg-max}, which of course also follows trivially from the inequality  $\overline{H}^{k}_{[T]} \le \widehat{H}^{k}_{[T]} $.
To use Theorem \ref{theorem:avg.2way.asym}, we compute
\begin{align*}
 \sum_{i,j=1}^n\xi_{ij}^2 = \frac{n}{m}, \qquad 
\sum_{i=1}^n\bigg(\sum_{j=1}^n\xi_{ij}^2\bigg)^2 =  \frac{n}{m^2}..
\end{align*}
Combined with $\delta=1/m$, applying Theorem \ref{theorem:avg.2way.asym} yields
\begin{equation}
\overline{H}^{k}_{[T]} \lesssim \bigg( \frac{k}{m} + 1\bigg)\bigg(\frac{k^2}{nm} + \frac{k}{m^2}\bigg) . 
 \label{ineq:ex:mreg-avg}
\end{equation}
For $k=O(1)$, this bound is again of order $1/m^2$, but when $k$ is allowed to grow with $n$ it reveals an interesting new detail compared to the preceding bounds. Specifically, unlike the previous bounds, \eqref{ineq:ex:mreg-avg} can vanish even in cases where $k$ is larger than $m$; precisely it vanishes when $k = o(\min(m^{3/2},(nm^2)^{1/3}))$, for instance when $m = O(n^{2/5})$ and $k = o(m^{3/2})$. In the reversed entropy case discussed in Section \ref{se:reversedentropy}, the prefactor of $(k/m+1)$ disappears, and thus the size of chaos is even larger: one can take $k=o(\min(m^2,(nm)^{1/2}))$, for example $k=o(m^2)$ when $m \le n^{1/3}$.

To apply the setwise entropy estimate of Theorem \ref{theorem:pointwise},  it will be helpful to write $\xi=(1/m)A$, where $A$ is the adjacency matrix of the underlying $m$-regular graph. Then
\begin{equation}
q_{\xi}(v) = \Big(\frac{|v|}{m}+1\Big)\bigg(\frac{1}{m^2}\sum_{i,j \in v} A_{ij} +  \frac{2}{m^3}\sum_{i,j \in v} (A^2)_{ij} + \frac{|v|}{m^2}\bigg). \label{ineq:ex:mreg-setwise}
\end{equation}
The two summations on the right-hand side count, respectively, the number of edges in $v$ and the number of paths of length two which start and end in $v$. The latter is at least $m|v|$, as seen by retaining only the $i=j$ terms in the sum. Thus, the last term  $|v|/m^2$ of \eqref{ineq:ex:mreg-setwise} is dominated by the second to last term.
Hence, Theorem \ref{theorem:pointwise} implies 
\begin{equation}\label{eq:regular.graph.pointwise}
H_{[T]}(v) \lesssim q_{\xi}(v) \lesssim \bigg(\frac{|v|}{m} + 1\bigg)\bigg(\frac{1}{m^2} \sum_{i,j \in v} A_{ij} + \frac{1}{m^3}\sum_{i,j \in v} (A^2)_{ij}\bigg), \quad v \subset [n].
\end{equation}
This yields the bound announced in \eqref{intro:pointwise-v-mregular}.
Two extreme cases illustrate the range of values this can take, depending on how connected the set $v$ is. If $v$ is highly disconnected, in the sense that there are no paths of length one or two between distinct vertices in $v$, then  \eqref{eq:regular.graph.pointwise} becomes
\begin{align*}
H_{[T]}(v) \lesssim \bigg(\frac{|v|}{m} + 1\bigg)\frac{|v|}{m^2},
\end{align*}
which is small as long as $|v|=o(m^{3/2})$.
If instead $v$ is highly connected, for instance a clique (which in particular implies $|v| \le m$), then there are $|v|(|v|-1)$  directed edges in $v$, and \eqref{eq:regular.graph.pointwise} becomes
\begin{align*}
H_{[T]}(v) \lesssim \bigg(\frac{|v|}{m} + 1\bigg)\frac{|v|^2}{m^2},
\end{align*}
which  is small if $|v|=o(m)$, and is the same order as the maximal entropy $\widehat{H}^{k}_{[T]}$ when $|v|=k$. In summary, the size of $H_{[T]}(v)$ is controlled by a tradeoff between the size of $v$ and its connectedness.

\subsection{The random walk case} \label{ex:random.walk.graph}
Recall the random walk case of Definition \ref{def:randomwalk-case}, and abbreviate $m_i=\mathrm{deg}(i)$ for the degree of vertex $i$.
That is,  $\xi_{ij}=(1/m_i)1_{i \sim j}$.
Assume the graph has at least one edge, to avoid the trivial case $\xi=0$.
Note that $\xi$ is asymmetric except in the regular graph case.
The row sum condition \eqref{cond:row.sum} is clearly satisfied. We have
	\begin{equation*}
		\delta = \frac{1}{m_*}, \text{ where } m_* := \min_{i \in [n]} m_i 1_{m_i>0}, \quad \text{ and } \quad
		 \delta_i  = \frac{1}{m_i }1_{m_i>0}.
	\end{equation*}
Applying Theorem \ref{theorem:max}, we deduce
	\begin{equation*}
		\widehat{H}^{k}_{[T]} \lesssim (k/m_*)^2 + (k/m_*)^3 ,
	\end{equation*}
	which is of course $O((k/m_*)^2)$ when $k =O(m_*)$.
In other words, the maximal entropy is controlled by the minimum degree.  
If we have bounded column sums, which here means that
	\begin{equation}
		\max_{i \in [n]} \sum_{j \sim i, \, m_j \neq 0} \frac{1 }{m_j} \le 1, \label{eq:random.walk.asmp-columnsum}
	\end{equation}
	then we can apply Corollary \ref{co:avg.1way} to get the sharper bound
	\begin{align}
		\overline{H}^{k}_{[T]} \lesssim \Big(\frac{k }{m_*} + 1\Big) \frac{k^2}{n} \sum_{i=1, \, m_i \neq 0}^n \frac{1 }{ m_i^2 } . \label{eq:random.walk.avg.1way}
	\end{align}
	Note as in Remark \ref{re:row.sum} that if the right-hand side of \eqref{eq:random.walk.asmp-columnsum} is a constant other than 1, we could change it to 1 by rescaling $b$ in proportion. We skip the application of Theorem \ref{theorem:avg.2way.asym}, which we did not find particularly enlightening in this example.  
	
	Even if column sums are not bounded as in \eqref{eq:random.walk.asmp-columnsum}, we can apply Theorem \ref{theorem:avg-markov} to estimate certain weighted averages. Indeed,  the natural choice of $\pi$ is $\pi_i=m_i/\sum_jm_j$, which is the invariant measure of the simple random walk on the graph. The relevant quantity in the bound of Theorem \ref{theorem:avg-markov} is
	\begin{equation}
	\sum_{i=1}^n\pi_i \delta_i^2 = \frac{1}{\sum_i m_i}\sum_{i: m_i \neq 0} \frac{1}{m_i}. \label{ineq:nonunif-avg-rw}
	\end{equation}
	This vanishes as long as the average degree diverges, $(1/n)\sum_im_i \to \infty$. There are two natural choices for the random set $\V$ from Theorem \ref{theorem:avg-markov}. For $k=2$, and assuming the graph is connected, an interesting choice is to take $\V=\{Z_1,Z_2\}$, where $Z_1\sim\pi$ and $Z_2$ is a uniform random neighbor of $Z_1$. The bound \eqref{eq:markov-mainbound} then becomes 
	\begin{equation}
\frac{1}{\sum_im_i}\sum_i m_i H_{[T]}(\{i\}) \le \frac{1}{\sum_im_i}\sum_{i=1}^n \sum_{j \sim i}  H_{[T]}(\{i,j\}) \lesssim  \frac{1}{\sum_i m_i}\sum_{i} \frac{1}{m_i}, \label{intro:avg1way-nonunif-rw}
\end{equation}
where the first inequality $\E[H_{[T]}(\{Z_1\})] \le \E[H_{[T]}(\{Z_1,Z_2\})]$ is due to the data processing inequality. (This is a special case of \eqref{intro:avg1way-nonunif}, except that \eqref{intro:avg1way-nonunif-rw} involves path-space entropies instead of time-marginals.)
Letting $E$ denote the set of (undirected) edges of the graph, note that $\sum_im_i=2|E|$, and so the middle term in \eqref{intro:avg1way-nonunif-rw} is $(1/|E|)\sum_{e \in E}{H_{[T]}(e)}$.  An alternative choice, for any $k \in [n]$, is $\V=\{Z_1,\ldots,Z_k\}$ where $Z_i$ are i.i.d $\sim \pi$. The bound \eqref{eq:markov-mainbound} then becomes 
\begin{align*}
\sum_{i_1,\ldots,i_k=1}^n \bigg(\prod_{j=1}^k\pi_{i_j}\bigg) H_{[T]}(\{i_1,\ldots,i_k\}) \lesssim \bigg(\frac{k}{m_*} + 1\bigg) \frac{k^2}{\sum_i m_i}\sum_{i} \frac{1}{m_i} .
\end{align*}
Note that there may be repeated terms among $\{i_1,\ldots,i_k\}$, which are collapsed into the set of distinct entries in $H_{[T]}(\{i_1,\ldots,i_k\})$.

This example can be applied to many models of random graphs, in the \emph{quenched} sense where the realization of the random graph determines $\xi$ as input to our main theorems.
For example, consider the \Erdos\ graph with edge probability $p$. Here $p$ is allowed to depend on $n$, but we suppress this dependence.
There are two denseness thresholds of interest. First, when $np\to\infty$, at any speed, the right-hand side of \eqref{ineq:nonunif-avg-rw} is of order $\asymp 1/(np)^2$ with high probability, as is easily seen using the multiplicative Chernoff bound, and thus  $ \E[H_{[T]}(\V)] \to 0$ for $k=O(1)$.
Second, the minimum degree $m_*$ diverges if  $\liminf_{n\to\infty} np/\log n > 1$ (see \cite[Lemma 6.5.2]{durrett2010random}), which is exactly the threshold for connectivity. It is only in this regime that the maximum entropy $\widehat{H}^k_T \to 0$ for $k$ fixed as $n\to\infty$.

\subsection{Scaled adjacency matrix} \label{ex:AdjMatrix}
Our next example is inspired by recent literature on universality for Ising and Potts models on graphs \cite{basak2017universality}. Suppose $A$ is the adjacency matrix of a graph $G$ with vertex set $[n]$ and nonempty edge set $E$. Let $\delta > 0$ be a scalar, and set $\xi= \delta A$, which is consistent with our usual notation $ \delta = \max_{ij}\xi_{ij}$.
We have two cases in mind:
\begin{enumerate}[(I)]
\item $G$ is non-random, and $\delta=n/2|E|$ is the reciprocal of the average degree.
\item $G$ is (a realization of) the \Erdos\ graph with edge probability $p$, and $\delta = 1/np$.
\end{enumerate}
This is the natural scaling which ensures that the average row sum is 1, or $(1/n)\sum_{i,j=1}^n\xi_{ij}=1$, in expectation in case (II).
Then $\xi$ is symmetric, and its maximal row sum is $\delta m^*$, where $m^*$ is the maximal degree of the graph (not the minimum degree which was denoted $m_*$ in Section \ref{ex:random.walk.graph}).  The bounded row sum assumption \eqref{cond:row.sum}, or rather its relaxation in Remark \ref{re:row.sum}, is valid as long as $m^* \lesssim 1/\delta$. In case (I) above, this means that the maximal degree is of the same order as the average degree, or $m^* \asymp \overline{m} := (1/n)\sum_i m_i$. In case (II), this means that $m^* \lesssim np$, which holds with high probability as $n\to\infty$, even if $p$ is allowed to vanish as $n\to\infty$, as long as $\liminf np/\log n > 0$; this follows easily from the multiplicative Chernoff bound.

The maximum entropy bound of Theorem \ref{theorem:max} is easy to apply. In this case, $\max_j \xi_{ij} = \delta$ for any non-isolated vertex $i$, so the average entropy bound of Corollary \ref{co:avg.1way} yields no improvement over Theorem \ref{theorem:max}.
To apply 
Theorem \ref{theorem:avg.2way.asym}, we compute 
\begin{align*}
\sum_{i,j=1}^n\xi_{ij}^2 &= \delta^2 \sum_{i=1}^nm_i, \ \ \text{ and } \ \ 
 \sum_{i=1}^n\bigg(\sum_{j=1}^n (\xi_{ij}^2 +\xi_{ji}^2)\bigg)^2\bigg) = 4\sum_{i=1}^n\bigg(\sum_{j=1}^n \xi_{ij}^2\bigg)^2 =  4\delta^4\sum_{i =1}^nm_i^2,
\end{align*}
where $m_i$ again denotes the degree of vertex $i$. Then Theorem \ref{theorem:avg.2way.asym} yields
\begin{equation}
\overline{H}^k_{[T]} \lesssim (\delta k + 1)\bigg(  \delta^2 \frac{k^2}{n^2}      \sum_{i=1}^nm_i + \delta^4\frac{  k}{n }\sum_{i =1}^nm_i^2 \bigg). \label{eq:adjmat-avg}
\end{equation}
Turning to the setwise entropy estimate of Theorem \ref{theorem:pointwise}, assuming $v \subset [n]$ is of size $|v| = O(1/\delta)$ again for simplicity, we have
\begin{align*}
H_{[T]}(v) \lesssim q_\xi(v) \lesssim \delta^2 \sum_{i,j \in v}A_{ij} + \delta^3 \sum_{i,j \in v}(A^2)_{ij} + |v|\delta^2.
\end{align*}
This generalizes \eqref{intro:pointwise-v-mregular} beyond the regular graph setting. Let us summarize how these bounds specialize in the two  cases mentioned above:
\begin{enumerate}[(I)]
\item Letting $\overline{m}=1/\delta=\frac{1}{n}\sum_{i=1}^n m_i$ denote the average degree, we simplify the above to 
\begin{equation} \label{eq:adj-caseI}
\widehat{H}^k_{[T]}  \lesssim (k/ \overline{m})^2 + (k/ \overline{m})^3, \qquad \overline{H}^k_{[T]}  \lesssim \bigg(  \frac{k}{\overline{m}}+1\bigg)\bigg(\frac{k^2}{n\overline{m}} +  \frac{k}{\overline{m}^2} \bigg) . 
\end{equation}

Here we used the fact that $\overline{m} \asymp m^*$, which implies $(1/n)\sum_i m_i^2 \asymp \overline{m}^2$ by Jensen's inequality. This bound behaves like \eqref{ineq:ex:mreg-avg} from the $m$-regular graph case, except with $m$ replaced by $\overline{m}$.  
As in the regular graph case, the average entropy bound can accommodate larger values of $k$, although both bounds in \eqref{eq:adj-caseI} are of the same order $1/\overline{m}^2$ when $k=O(1)$. 
\item In the \Erdos\ case with $\liminf np/\log n > 0$, we get $\widehat{H}^k_{[T]} \lesssim (k/np)^2 + (k/np)^3$ with high probability. The bound on $\overline{H}^k_{[T]}$ in \eqref{eq:adjmat-avg}  behaves in expectation exactly like \eqref{ineq:ex:mreg-avg} except with $m$ replaced by $np$. 
\end{enumerate}

\subsection{Rank-one matrices}
 
Suppose  $\alpha,\beta\in\R^n$ have nonnegative entries, and $\xi_{ij}=\alpha_i\beta_j$ for $i \neq j$, with $\xi_{ii}=0$ as usual. Then the corresponding $ n $-particle system is given by
	\begin{equation*}  
	d X^{i}_t = \bigg(b^i_0(t, X^{i}_t) + \alpha_i\sum_{j \neq i}  \beta_j b^{ij}(t, X^{i}_t, X^{j}_t)\bigg) dt + \sigma dB^i_t, \quad i = 1, \dots ,n.
	\end{equation*}
This class of examples arises naturally in the classical $n$-body problem of masses interacting via gravitational force, where $\alpha_i=\beta_i$ is the mass of the $i$th body. We mention this as a motivating class of examples, though our results do not technically apply to gravitational interactions, which are typically described by  noiseless second-order (kinetic) models with singular interactions.

Write $|\cdot|_p$ for the $\ell_p$ norm on $\R^n$, for $p \in [1,\infty]$. The row sums of $\xi$ are bounded by 1 (as required by our main theorems) if $|\beta|_1|\alpha|_\infty \le 1$. For the results of ours which required bounded column sums, we would also assume $|\alpha|_1|\beta|_\infty \le 1$.

To apply our main theorems, we compute
\begin{align*}
\delta = |\alpha|_\infty |\beta|_\infty  , \qquad \delta_i = \alpha_i |\beta|_\infty. 
\end{align*}
Using Theorem \ref{theorem:max}, the maximum entropy  is bounded by
\begin{equation}
\widehat{H}^k_{[T]} \lesssim  (k |\alpha|_\infty  |\beta|_\infty)^3+  (k |\alpha|_\infty  |\beta|_\infty)^2. \label{ex:rankone-max}
\end{equation}
If $|\alpha|_1|\beta|_\infty \le 1$ so that the column sum condition is satisfied, then Corollary \ref{co:avg.1way} yields
\begin{equation}
\overline{H}^k_{[T]} \lesssim  \big(k|\alpha|_\infty |\beta|_\infty+1\big) k^2 |\beta|_\infty^2 \frac{|\alpha|_2^2}{n}. \label{ex:rankone-avg1}
\end{equation} 
To relax the restriction $|\alpha|_1|\beta|_\infty \le 1$, we can apply the weighted average entropy bounds of Theorem \ref{theorem:avg-markov}. If we instead assume that $\alpha \cdot \beta \le 1$ (with the constant 1 being arbitrary, as usual), we can take $\pi=\beta/|\beta|_1$ in Theorem \ref{theorem:avg-markov}. The relevant quantity in Theorem \ref{theorem:avg-markov} is
\begin{align*}
\sum_{i=1}^n \pi_i\delta_i^2 = |\beta|_\infty^2  \frac{\sum_i\beta_i \alpha_i^2}{\sum_i\beta_i}.
\end{align*}

An example worth mentioning is when $\alpha_i=1/n$ for all $i$, which was studied in \cite{wang2022mean}. Therein, a mean field limit for the weighted empirical measure $(1/n)\sum_{i=1}^n\beta_i\delta_{X^i_t}$ was shown for models with singular interactions, notably leading to particle approximations for the PDE of a passive scalar advected by the 2D Navier-Stokes equation.
Our results yield different information about this same setup (though we do not allow for singular interactions as they do).
It is assumed in \cite{wang2022mean} that $(1/n) \sum_i \beta_i^r = O(1)$ for some $r \in (1,\infty)$ or $|\beta|_\infty = O(1)$. We need only assume that $(1/n) \sum_i \beta_i  = O(1)$ and  $|\beta|_\infty = o(n)$.
Then the right-hand side of \eqref{ex:rankone-max} becomes $|\beta|_\infty^3 (k/n)^3+|\beta|_\infty^2 (k/n)^2$, which vanishes  if $k=O(1)$.
Note that $\sum_{j=1}^n\xi_{ij} = (1/n)\sum_{j \neq i} \beta_j$ is close to the constant $(1/n)|\beta|_1$, and we thus expect the independent projection to be close to i.i.d copies of the McKean-Vlasov equation, with drift scaled by this constant $(1/n)|\beta|_1$.

\subsection{Sequential propagation of chaos} \label{se:sequentialpoc}

The recent paper \cite{du2023sequential} studies  the case  where $\xi$ is lower-triangular, motivated by computational considerations, so that each particle $i$ in sequence is influenced by a weighted average only  over the previous particles $j < i$. 
A notable special case of their more general setup is where the weights are uniform, so $\xi_{ij} = 1_{j < i}/(i-1)$. Note in this case that the row sums equal 1 as in \eqref{intro:stochasticmatrix}, so we expect the usual McKean-Vlasov equation in the limit. For Lipschitz $(b^0,b)$ it was shown in \cite[Theorem 2.1]{du2023sequential} that the expected squared Wasserstein distance between the $n$-particle empirical measure and the McKean-Vlasov limit is $O(n^{-2/(d+4)})$. 
 This is for the time-marginal laws, whereas for path-space laws they replace $n^{-2/(d+4)}$ by $\log\log n / \log n$.

The only result of ours that meaningfully applies is Theorem \ref{theorem:avg-markov}. Indeed, $\delta=\max_{ij}\xi_{ij}= 1$, which makes our maximum entropy bound of Theorem \ref{theorem:max} uninformative. 
Moreover, the maximal column sum $\max_j\sum_i\xi_{ij}=\sum_{i=1}^{n-1}1/i$ is of order $\log n$, so Corollary \ref{co:avg.1way} and Theorem \ref{theorem:avg.2way.asym} do not apply.
To apply Theorem \ref{theorem:avg-markov}, note first that $\delta_i=\max_j \xi_{ij} = 1_{i > 1}/(i-1)$. We must identify $\pi$ satisfying $\pi^\top \xi \le \pi^\top$ coordinatewise and $\sum_i \pi_i \le 1$.
Here are two examples. The first is degenerate: If $\pi=(1,0,\ldots,0)$ then $\pi^\top \xi = 0 \le \pi^\top$, which forces $\V\subset \{1\}$ to be nonrandom in light of the requirement $\PP(i \in \V) \le k\pi_i$ for all $i$. Since $\delta_1=0$ the right-hand side of \eqref{eq:markov-mainbound} vanishes. This makes sense because the independent projection has the same first marginal as the particle system itself, due to the first row of $\xi$ being zero.

The interesting non-degenerate example is $\pi_i=c/i$ for $c=1/\sum_{i=1}^n (1/i)$.
Then
\begin{align*}
(\pi^\top \xi)_j &=	\sum_{i=j+1}^{n} \frac{\pi_i}{i-1} =  \sum_{i=j+1}^{n } \frac{c}{(i-1)i} = c\sum_{i=j+1}^{n }\bigg(\frac{1}{i-1} - \frac{1}{i} \bigg)  = c\bigg(\frac{1}{j} - \frac{1}{n }\bigg) \le \frac{c}{j}=\pi_j.
\end{align*}
The relevant quantity from Theorem \ref{theorem:avg-markov} becomes
\begin{align*}
\sum_{i=1}^n\pi_i \delta_i^2 = \sum_{i=2}^{n }\frac{c}{i(i-1)^2} \le c \le \frac{1}{\log n}.
\end{align*}
Thus, for $\V$ as in Theorem \ref{theorem:avg-markov}, we get
\begin{align*}
	\E[H_{[T]}(\V)] \lesssim  k^3/\log n .
\end{align*}
We do not know if this is sharp, and it is difficult to compare directly with the aforementioned results of \cite{du2023sequential}, but it is natural to expect this weighted average to vanish slowly as $n\to\infty$. In fact, it is surprising that it converges at all, because the heavier weights $\pi_i=c/i$ are given to the low-index particles for which not as much averaging occurs.

\section{From the particle system to the percolation process} \label{se:hierarchyproofs}
This section is devoted to the proof of Proposition \ref{pr:CTMC}, which bounds the entropies $H_{[t]}(v)$ and $H_{t}(v)$ in terms of the percolation process. To this end, we first derive in Section \ref{subsec:iterated.hierachy} the hierarchy of differential inequalities satisfied by these entropies, stated in \eqref{intro:hierarchy-new} in the introduction. Section \ref{subsec:proof.thm.2.7} then shows how to deduce Proposition \ref{pr:CTMC} from these hierarchies.  

The following shorthand notation will be useful: For $v \subset [n]$ and $j\in [n]\backslash v$, let $vj:=v\cup\{j\}$.

\subsection{The hierarchy of differential inequalities} \label{subsec:iterated.hierachy}

Our first lemma pertains to the path-space entropies $H_{[t]}(v)$, following and adapting the strategy developed in \cite{lacker2022hierarchies} for the exchangeable case; see specifically the proof of Theorem 2.2 therein up to equation (4-18).
Recall in the following the definitions \eqref{eq:def.C(v).D(v)} related to the percolation process.

\begin{lemma}\label{lem: init.ent.est}
Suppose Assumption \ref{asssump:common} holds.  Suppose $H_0([n]) < \infty$.
Let $v \subset [n]$. 
	\begin{enumerate}[(i)]
	\item The map $t \mapsto H_{[t]}(v)$  is absolutely continuous, and  for a.e.\ $t \in [0,T]$,
	\begin{equation} \label{ineq:lambda.CS.diffineq}
		\frac{d}{dt} H_{[t]}(v)  \leq \CC(v) +  \inf_{R \in \mathcal{R}} \sum_{j \notin v} \A^{R}_{v\to j} \left(H_{[t]}(vj)  - H_{[t]}(v)\right), 
	\end{equation}
	By convention, the final term of \eqref{ineq:lambda.CS.diffineq} is zero if $v=[n]$. 
	\item If it holds for some constant $h_3$ that 
	\begin{equation}
		 H_{[T]}(v) \le h_3, \quad \text{for all } v\subset [n] \text{ with } |v|=3, \label{asmp:k=3}
		\end{equation}
		 then \eqref{ineq:lambda.CS.diffineq} holds with $\CC(v)$ replaced by 
		\begin{equation}
		 \widehat{\CC}(v) := \frac{ \sqrt{\gamma M h_3}}{\sigma^2} \sum_{i \in v}\bigg(\sum_{j \in v}\xi_{ij}\bigg)^2 + \frac{M}{\sigma^2} \sum_{i,j \in v} \xi_{ij}^2. \label{def:CChat}
	\end{equation} 
	\end{enumerate}
	\end{lemma}

At first we will apply part (i). As in \cite{lacker2022hierarchies}, after we have a good bound on $h_3$ from a first pass through the argument, we will apply part (ii) and repeat the argument to sharpen the results.

	\begin{proof}[Proof of Lemma \ref{lem: init.ent.est}]
	We begin by treating the case of $v=[n]$ separately, in part for transparency and in part for the technical purpose of implying that  $H_{[T]}(v) < \infty$ for all $v \subset [n]$. We will first apply \cite[Lemma 4.4(iii)]{lacker2022hierarchies}, a well known entropy estimate based on Girsanov's theorem. As a preparation, we first show that the assumptions therein are satisfied. Thanks to the well-posedness  in Assumption \ref{asssump:common}(i), we only need to verify the integrability condition in \cite[Equation (4.9)]{lacker2022hierarchies}, which in our context requires showing
		\begin{align*} 
			\sum_{i=1}^n  \int_0^T  \int_{(\R^d)^n} \bigg| \sum_{j=1}^n\xi_{ij} \left( b^{ij}(t, x_i, x_j) - \big< Q^j_t, b^{ij}(t, x_i, \cdot) \big> \right) \bigg|^2 P_t(dx) \, dt  < \infty, \\
			\sum_{i=1}^n \int_0^T \int_{(\R^d)^n}  \bigg| \sum_{j=1}^n\xi_{ij} \left( b^{ij}(t, y_i, y_j) - \big< Q^j_t, b^{ij}(t, y_i, \cdot) \big> \right) \bigg|^2 Q_t(dy) \, dt  < \infty. 
		\end{align*}
		The first of these two claims  is a straightforward consequence of  Assumption \ref{asssump:common}(ii). The second follows from the fact that,
by \cite[Lemma 2.3 and Remark 2.4(i)]{lacker2022hierarchies}, our Assumption \ref{asssump:common}(iii) implies the following much stronger exponential square-integrability: there exists a $  \kappa > 0$ such that 
		\begin{align*}
			\sup_{(t, x)\in[0,T]\times\R^d} \int_{\R^d} \exp\left( \kappa\left| b^{ij}(t,y_i,y_j)- \big<Q^j_t, b^{ij}(t,y_i,\cdot)\big> \right|^2 \right) Q^j_t(dy_j) <\infty.
		\end{align*}
Having checked the assumptions, we may now finally apply \cite[Lemma 4.4(iii)]{lacker2022hierarchies} to find 
		\begin{align*}
			\frac{d}{dt} H_{[t]}([n]) & = \frac{1}{2 \sigma^2} \sum_{i=1}^n \E \bigg[ \Big|  \sum_{j=1}^n\xi_{ij} 
			\left( b^{ij}(t, X^i_t, X^j_t) - \big\langle Q^{j}_t, b^{ij}(t, X^i_t, \cdot)\big\rangle \right) \Big|^2\bigg] \leq \frac{M}{2\sigma^2}\sum_{i=1}^n\Big(\sum_{j=1}^n\xi_{ij}\Big)^2,
		\end{align*}
		where we used the Cauchy-Schwarz inequality and Assumption \ref{asssump:common}(ii) in the last step. Since $H_0([n]) < \infty$, we deduce that $H_{[T]}([n]) <\infty$ as claimed. 
		
		Next, we identify the dynamics for any subset $v \in [n]$ of particles. For $x \in C([0,T];\R^d)$ write $x_{[t]}=(x_s)_{s \in [0,t]} \in C([0,t];\R^d)$ for the path up to time $t \le T$, and similarly for $x \in C([0,T];(\R^d)^v)$.
		Write $\FF^v=(\F^v_t)_{t \in [0,T]}$ for the filtration generated by the particles in $v$, i.e., $\F^v_t$ is generated by the random variable $X^v_{[t]}$.
		For any $i \in v$ and $j \notin v $, there exists a progressively measurable function $ \widehat{b}^v_{ij} : [0, T] \times  C([0,T];(\R^d)^v)\to \R^d $ such that
		\begin{equation}
			\widehat{b}^v_{ij} (t, X^v) = \E[b^{ij}(t, X^i_t, X^j_t) \, | \, \F^v_t], \quad a.s., \ a.e. \  t \in [0,T].  \label{pf:bhat-hierarchy}
		\end{equation}
		For any $ i \in v $, we compute the conditional expectation of the drift of $X^i_t$ given $\F^v_t$:
		\begin{align*}
			\E &\Big[b_0^{i}(t, X^i_t) +  \sum_{j \neq i} \xi_{ij} b^{ij}(t, X^i_t, X^j_t) \,\Big|\, \F^v_t\Big] = b_0^{i} (t,  X^i_t ) + \sum_{j \in v}\xi_{ij} b^{ij}(t, X^i_t ,  X^j_t ) +\sum_{j \notin v} \xi_{ij} 	\widehat{b}^v_{ij} (t, X^v ).
		\end{align*}
By a projection argument \cite[Lemma 4.1]{lacker2022hierarchies}, we may change the Brownian motions so that  this conditional expectation becomes the drift of $X^i_t$, for each $i \in v$. Precisely,  there exist independent $\FF^v$-Brownian motions $(\widehat{B}^i)_{i \in v}$ such that 
\begin{equation} \label{mimicking_sde_x}
d X^{i}_t = \Big(b_0^{i}(t, X^i_t) +  \sum_{ j \in v} \xi_{ij} b^{ij}(t, X^i_t, X^j_t) + \sum_{j \notin v} \xi_{ij} 	\widehat{b}^v_{ij} (t, X^v )\Big) dt + \sigma d \widehat{B}^i_t, \quad i \in v.
\end{equation}
For the independent projection \eqref{eq.independent.projection.sys}, the particles in $ v $ solve the SDE system
		 \begin{align} \label{eq.independent.size.k.subsys}
	 		d Y^{i}_t = \Big(b_0^{i}(t, Y^{i}_t) + \sum_{j =1}^n \xi_{ij} \big\langle Q^{j}_t\, ,\, b^{ij}(t, Y^{i}_t, \cdot) \big\rangle \Big)dt + \sigma dB^i_t, \quad i \in v.
		 \end{align}
With these dynamics identified, we will apply the entropy identity \cite[Lemma 4.4 (ii)]{lacker2022hierarchies} to  \eqref{mimicking_sde_x} and \eqref{eq.independent.size.k.subsys}. To justify this, note that by the data processing inequality that $H_{[t]}(v) \leq H_{[T]}(v) \leq H_{[T]}([n])$ holds for any subset $v\subset [n]$, and $H_{[T]}([n])$ is finite as was shown in the first part of the proof.  Thus,
		\begin{align}
			\frac{d}{dt} H_{[t]}(v)  & = \frac{1}{2 \sigma^2} \sum_{i \in v} \E \bigg[ \bigg| 
			\sum_{j \in v}\xi_{ij} b^{ij}(t, X^i_t, X^j_t) + \sum_{j \notin v} \xi_{ij} \widehat{b}^v_{ij} (t, X^v) -\sum_{j \neq i} \xi_{ij}\left\langle Q^{j}_t, b^{ij}(t, X^i_t, \cdot)\right\rangle  \bigg|^2 \bigg] \nonumber\\
			& \leq \frac{1}{\sigma^2} \sum_{i \in v} \E \bigg[\bigg| \sum_{j \in v}\xi_{ij} \left( b^{ij}(t, X^i_t, X^j_t) - \big\langle Q^{j}_t, b^{ij}(t, X^i_t, \cdot) \big\rangle \right) \bigg|^2\bigg] \nonumber\\
			& \quad + \frac{1}{\sigma^2} \sum_{i \in v} \E \bigg[\bigg| \sum_{j \notin v}\xi_{ij} \left( \widehat{b}^v_{ij} (t, X^v)- \left\langle Q^{j}_t, b^{ij}(t, X^i_t, \cdot) \right\rangle \right) \bigg|^2\bigg] \nonumber \\
			&  =: \text{I} + \text{II}. \label{def.I.II}
		\end{align}

To control term I, we simply use the Cauchy-Schwarz inequality and recall the definition of $M$ from Assumption \ref{asssump:common}(ii):
		\begin{align} \label{eq:old.treatment.term.I}
			\text{I} & \leq \frac{1}{\sigma^2} \sum_{i \in v} \bigg( \sum_{j \in v } \xi_{ij} \bigg)
			\bigg(\sum_{j \in v } \xi_{ij} \E\Big[\left| b^{ij}(t, X^i_t, X^j_t) - \left<Q^{j}_t, b^{ij}(t, X^i_t, \cdot)\right> \right|^2\Big] \bigg) \\ \nonumber
			& \leq  \frac{M}{\sigma^2} \sum_{i\in v}\bigg( \sum_{j \in v } \xi_{ij}\bigg)^{2} = \CC(v).
		\end{align}
For term II, we introduce some additional notation.  Let $P^{j|v}_{t; X^v_{[t]}} (d x^j_t)$ denote a version of the regular conditional law of $X^j_{t}$ given $X^v_{[t]}$, and let $P^{j|v}_{[t]; X^v_{[t]}}(dx^j_{[t]})$ denote a version of the regular conditional law of $X^j_{[t]}$ given $X^v_{[t]}$. Then the assumed transport-type inequality  \eqref{cond.transport.type.ineq}  implies
		\begin{align*}
			 \big| \widehat{b}^v_{ij} (t,  X^v )- \big< Q^{j}_t, b^{ij}(t, X^i_t, \cdot) \big> \big|^2 &=  \big| \big< P^{j|v}_{t; X^v_{[t]}} - Q^{j}_t, b^{ij}(t, X^i_t, \cdot) \big> \big|^2 \\&  \le \gamma H \big(P^{j|v}_{t; X^v_{[t]}} \, |\, Q^{j}_t \big) \le \gamma H \big(P^{j|v}_{[t]; X^v_{[t]}} \, |\, Q^{j}_{[t]} \big), \quad a.s.,
		\end{align*}
		where we used the data-processing inequality in the last step. 
		Recalling that $\E$ denotes expectation under $P$, the chain rule for relative entropy implies
		\begin{align*}
			\E \big[H \big(P^{j|v}_{[t]; X^v_{[t]}} \, |\, Q^{j}_{[t]} \big)\big] = H_{[t]}(vj) - H_{[t]}(v).
		\end{align*} 
		Therefore, using the triangle inequality for the $L^2$ norm, we see that
		\begin{align} 
			\text{II} &\leq \frac{1}{\sigma^2} \sum_{i \in v} 
			\bigg( \sum_{j \notin v}\xi_{ij} \sqrt{  \E \Big[\Big| \widehat{b}^v_{ij} (t, X^v)- \left\langle Q^{j}_t, b^{ij}(t, X^i_t, \cdot) \right\rangle \Big|^2\Big]} \bigg)^2. \nonumber \\ 
			&\le \frac{\gamma}{\sigma^2} \sum_{i \in v} 
			\bigg( \sum_{j \notin v}\xi_{ij} \sqrt{ H_{[t]}(vj) - H_{[t]}(v) } \bigg)^2. \label{pf:stochcontrol1}
		\end{align}
	We can then apply the following identity, valid for any $(a_1,\ldots,a_n)$
	and $(h_1,\ldots,h_n)$ in $\R^n_+ := [0,\infty)^n$:
	\begin{align}
		\bigg(\sum_{j=1}^n a_j\sqrt{h_j}\bigg)^2
		&= \inf\bigg\{2 \sum_{j=1}^n r_j h_j
		: r \in \R^n_+, \ \sum_{j=1}^n \frac{a_j^2}{r_j} \le 2\bigg\}.
		\label{identity:sumsqrt}
	\end{align}
	Indeed, for any $r \in \R^n_+$ satisfying $\sum_j(a_j^2/r_j) \le 2$,
	we have by Cauchy-Schwarz that
	\begin{align*}
		\bigg(\sum_{j=1}^n a_j\sqrt{h_j}\bigg)^2
		&\le \bigg(\sum_{j=1}^n r_j h_j \bigg)
		\bigg(\sum_{j=1}^n \frac{a_j^2}{r_j}\bigg)
		\le 2 \sum_{j=1}^n r_j h_j .
	\end{align*}
		and equality is obtained by taking $r_j=(a_j /2\sqrt{h_j})\sum_{j=1}^na_j\sqrt{h_j}$. Now, for each $i \in v$ we may apply \eqref{identity:sumsqrt} in \eqref{pf:stochcontrol1} with $h_j= H_{[t]}(vj) - H_{[t]}(v)$ and $a_j=\xi_{ij}$, and we find
		\begin{equation*}
			\text{II} \leq \frac{2\gamma}{\sigma^2} \inf_{R \in \mathcal{R}}  \sum_{i \in v}  \sum_{j \notin v} R_{ij}\big(H_{[t]}(vj) - H_{[t]}(v) \big).
		\end{equation*}
 where $\mathcal{R}$ is the constraint set given in \eqref{eq:def.control.set}.
	
		
		Combining the bounds on terms I and term II, we arrive at \eqref{ineq:lambda.CS.diffineq}.

		We next prove part (ii). We bound II in the same way, but we improve the bound of term I in \eqref{eq:old.treatment.term.I} by instead expanding the square to get the identity  I$=$I(a)$+$I(b), where
\begin{align*}
	\text{I(a)}
	&= \frac{1}{\sigma^2}\!\!\! \sum_{i,j,r \in v, \, j \neq r }\!\!\!\! \xi_{ij} \xi_{i r}\E \Big[\Big( b^{ij}(t, X^i_t, X^j_t) - \left\langle Q^{j}_t, b^{ij}(t, X^i_t, \cdot) \right\rangle \Big) \cdot \Big( b^{ir}(t, X^i_t, X^r_t) - \left\langle Q^{r}_t, b^{ir}(t, X^i_t, \cdot) \right\rangle \Big) \Big]\\
	\text{I(b)} & =\frac{1}{\sigma^2} \sum_{i,j \in v} \xi_{ij}^2 \E \Big[\Big| b^{ij}(t, X^i_t, X^j_t) - \left\langle Q^{j}_t, b^{ij}(t, X^i_t, \cdot) \right\rangle \Big|^2\Big].
\end{align*}
Recall that the diagonal entries of $\xi$ are zero, so the terms in the sums vanish if $i \in \{j,r\}$.
Using the above notation for conditional measures, we condition on $ (X^i_t, X^j_t) $ and use the Cauchy-Schwarz inequality to get, for distinct $i,j,r \in v$, 
\begin{align*}
	\E &\Big[\Big( b^{ij}(t, X^i_t, X^j_t) - \left\langle Q^{j}_t, b^{ij}(t, X^i_t, \cdot) \right\rangle \Big) \cdot \Big( b^{ir}(t, X^i_t, X^r_t) - \left\langle Q^{r}_t, b^{ir}(t, X^i_t, \cdot) \right\rangle \Big) \Big] \\
	&= \E \Big[\left( b^{ij}(t, X^i_t, X^j_t) - \left\langle Q^{j}_t, b^{ij}(t, X^i_t, \cdot) \right\rangle \right) \cdot  \Big\langle P^{r | \{i,j\}}_{t;X^{\{i,j\}}_{[t]}} - Q^{r}_t, b^{ir}(t, X^i_t, \cdot) \Big\rangle  \Big]\\ 
		&\le   \sqrt{M} \, \E \Big[ \big| \big\langle P^{r | \{i,j\}}_{t;X^{\{i,j\}}_{[t]}} - Q^{r}_t, b^{ir}(t, X^i_t, \cdot) \big\rangle \big|^2  \Big]^{1/2}.
\end{align*}
Apply the assumption \eqref{cond.transport.type.ineq}, followed by the data processing inequality and the chain rule of relative entropy, to bound the above further by
\begin{align*}
\sqrt{\gamma M } \, \E\Big[H\big(P^{r | \{i,j\}}_{t;X^{\{i,j\}}_{[t]}} \,\big|\, Q^r_t\big)\Big]^{1/2} &\le \sqrt{\gamma M}\Big(H_{[t]}(\{i,j,r\}) - H_{[t]}(\{i,j\})\Big)^{1/2} \le 	 \sqrt{\gamma M h_3} .
\end{align*}
Therefore,  
\begin{align*}
	\text{I(a)}&\le \frac{ \sqrt{\gamma M h_3}}{\sigma^2} \sum_{i,j,r \in v} \xi_{ij} \xi_{i r} = \frac{ \sqrt{\gamma M h_3}}{\sigma^2} \sum_{i \in v}\bigg(\sum_{j \in v}\xi_{ij}\bigg)^2.
\end{align*}
For I(b) we have the simple bound 
\begin{align*}
	\text{I(b)} \le \frac{M}{\sigma^2} \sum_{i,j \in v} \xi_{ij}^2.
\end{align*} 
Put it together to complete the proof.
		\end{proof}

\begin{remark} \label{re:reversedentropy1}
In Section \ref{se:reversedentropy} we mentioned the case of the reversed entropies $\overleftarrow{H}_{[t]}(v)=H(Q^v_{[t]}\,|\,P^v_{[t]})$, under the stronger assumption of bounded $b^{ij}$. For the reversed entropies we obtain the same hierarchy \eqref{ineq:lambda.CS.diffineq}, except with $C(v)$ replaced by $\widetilde{C}(v)=(M/\sigma^2)\sum_{i,j \in v} \xi_{ij}^2$. Indeed, the proof proceeds in the same manner, but with the particles $X^i_t$ replaced throughout by the independent projection $Y^i_t$, which ultimately results in I$=$I(b) because I(a) vanishes by independence. See Remark \ref{re:reversedentropy2} for the downstream implications of this.
\end{remark}
		
		We next give the analogous result for the time-marginal entropies $H_t(v)$, following the strategy of \cite[Section 3.3]{lacker2023sharp}. There are many parallels with the proof of Lemma \ref{lem: init.ent.est}.
		
		\begin{lemma}\label{lem: init.ent.est-unif}
	Suppose Assumption \ref{asssump:uniform.in.time} holds. Suppose $H_0([n]) < \infty$.
Let $v \subset [n]$. 
	\begin{enumerate}[(i)]
	\item For every $t \ge 0$,
	\begin{equation} \label{ineq:lambda.CS.diffineq.uit}
		H_t(v) - H_s(v) \le C(v)(t - s) + \sum_{j \notin v} \A^{R}_{v \to j} \int_s^t \Big(H_u(vj) - H_u(v)\Big) du- \frac{\sigma^2}{4 \eta} \int_s^t H_u(v) du.
	\end{equation}
	By convention, the second-to-last term of \eqref{ineq:lambda.CS.diffineq.uit} is zero if $v=[n]$. 
	\item If it holds for some constant $h_3$ that 
	\begin{equation}
		 \sup_{t \ge 0}H_{t}(v) \le h_3, \quad \text{for all } v\subset [n] \text{ with } |v|=3, \label{asmp:k=3.uit}
		\end{equation}
		 then \eqref{ineq:lambda.CS.diffineq.uit} holds with $\CC$ replaced by $\widehat{\CC}$ defined in \eqref{def:CChat}.
	\end{enumerate}
	\end{lemma}
		\begin{proof}
We first apply a projection argument, to express $X^v_t=(X^i_t)_{i \in v}$ as the solution of a Markovian SDE. At the level of the Fokker-Planck PDEs, this is a marginalization argument exactly like that used in deriving the BBGKY hierarchy. To parallel the previous proof, we favor a stochastic perspective, applying the mimicking theorem \cite[Corollary 3.7]{brunickshreve}. First, let us define the Markovian analogue of \eqref{pf:bhat-hierarchy}: For any $i \in v$ and $j \notin v $, there exists a Borel function $ \widehat{b}^v_{ij} : [0, T] \times  (\R^d)^v\to \R^d $ such that
		\begin{equation*}
			\widehat{b}^v_{ij} (t, X^v_t) = \E[b^{ij}(t, X^i_t, X^j_t) \, | \, X^v_t], \quad a.s., \ a.e. \  t > 0. 
		\end{equation*}
Then, by \cite[Corollary 3.7]{brunickshreve}, there exists  a weak solution $\widehat{X}^v=(\widehat{X}^i)_{i \in v}$ of the Markovian analogue of the SDE \eqref{mimicking_sde_x},
\begin{equation}  \label{mimicking_sde_x-Markov}
d \widehat{X}^{i}_t = \Big(b_0^{i}(t, \widehat{X}^i_t) +  \sum_{ j \in v} \xi_{ij} b^{ij}(t, \widehat{X}^i_t, \widehat{X}^j_t) + \sum_{j \notin v} \xi_{ij} 	\widehat{b}^v_{ij} (t, \widehat{X}^v_t )\Big) dt + \sigma d \widehat{B}^i_t, \quad i \in v,
\end{equation}
defined on a possibly different probability space with different Brownian motions, and with the crucial property that $\widehat{X}^v_t$ has the same law as $X^v_t$, for each $t \ge 0$.

We next make use of a well known calculation of the time-derivative of the relative entropy between the laws of two Markovian diffusion processes. To summarize formally how this works, suppose we are given solutions  of two different SDEs taking values in some Euclidean space, $dZ^i_t=a^i(t,Z^i_t)dt + \sigma dB^i_t$, for $i=1,2$. Let $\rho^i_t$ be the law of $Z^i_t$. Then, using the Fokker-Planck equation satisfied by $\rho^i$, one has the formal computation
\begin{align}
\frac{d}{dt}H(\rho^1_t\,|\,\rho^2_t) &= \int   \bigg( (a^1(t,z)-a^2(t,z)) \cdot \nabla\log\frac{d\rho^1_t}{d\rho^2_t}(z) - \frac{\sigma^2}{2}\Big|\nabla\log\frac{d\rho^1_t}{d\rho^2_t}(z)\Big|^2\bigg)\,\rho^1_t(dz) \nonumber \\
	&\le \frac{1}{\sigma^2} \int  |a^1(t,z)-a^2(t,z)|^2\,\rho^1_t(dz) - \frac{\sigma^2}{4}I(\rho^1_t\,|\,\rho^2_t). \label{pf:entropyinequality-timemarginal}
\end{align}
We refer to \cite[Lemma 3.1]{lacker2023sharp} for a rigorous version of the integrated form of this inequality (and further references), under mild local integrability conditions on $a^1$ and $a^2$ of a technical nature. We apply it with $a^1$ being the drift of $\widehat{X}^v$ as in \eqref{mimicking_sde_x-Markov}, and with $a^2$ being the drift of the dynamics for $Y^v$ which was recalled in \eqref{eq.independent.size.k.subsys}. The technical conditions were straightforward to check in \cite[Section 3.3]{lacker2023sharp}, and they are equally straightforward here, so we omit the details.
Applying the integrated form of \eqref{pf:entropyinequality-timemarginal} (that is, \cite[Lemma 3.1]{lacker2023sharp}) then yields
		\begin{align*}
			H_t(v) - H_s(v)  \le   \frac{1}{\sigma^2} \int_s^t  \sum_{i \in v} &\E \Biggl[ \Bigg| 
			\sum_{ j \in v}\xi_{ij} b^{ij}(u, X^i_u, X^j_u) + \sum_{j \notin v} \xi_{ij} \widehat{b}^v_{ij} (u, X^v_u) \\
			& \qquad  - \sum_{j=1}^n \xi_{ij}\left\langle Q^{j}_u, b^{ij}(u, X^i_u, \cdot)\right\rangle  \Bigg|^2\Biggr]  du -\frac{\sigma^2}{4} \int_s^t I(P^v_u \, | \, Q^v_u) du .
		\end{align*}
		The expectation term is estimated exactly as in the proof of Lemma \ref{lem: init.ent.est}. For the Fisher information,   Assumption \ref{asssump:uniform.in.time}(iv) together with tensorization of the log-Sobolev inequality \cite[Proposition 5.2.7]{bakry2014analysis} implies that $ H_u(v) \le \eta I(P^v_u \, | \, Q^v_u)$. Putting it together proves part (i). Part (ii) follows by improving the estimate on I, in exactly the same manner as in the proof of Lemma \ref{lem: init.ent.est}(ii).
	\end{proof}

\subsection{A note on a direct proof of the maximum entropy bound of Theorem \ref{theorem:max}}
We have reached the point in our arguments where the percolation process will make an appearance. However, we take a moment in this short section to point out that the percolation is not really needed if one is just interested in the maximum entropy bound of Theorem \ref{theorem:max}. Indeed, in this case we may reduce the analysis to a hierarchy of differential inequalities indexed by $[n]$ rather than $2^{[n]}$, and then appeal to the results of \cite{lacker2022hierarchies}:

\begin{proof}[Proof sketch of Theorem \ref{theorem:max} avoiding the percolation process]
Starting from Lemma \ref{lem: init.ent.est}(i), fix $v \subset [n]$ with $|v|=k$. Using $H_{[t]}(vj)\le \widehat{H}^{k+1}_{[t]}$ and the definition of $\mathcal{A}^{R}_{v\to j}$ in \eqref{eq:def.C(v).D(v)}, we obtain
		\begin{align*}
			\frac{d}{dt} H_{[t]}(v)
			&\le \CC(v)
			+ \big(\widehat{H}^{k+1}_{[t]} - H_{[t]}(v)\big)
			\inf_{R \in \mathcal{R}} \sum_{j \notin v} \A^{R}_{v\to j}
			\notag\\
			&\lesssim \delta^2 k^3 			+ \frac{2\gamma k}{\sigma^2}\big(\widehat{H}^{k+1}_{[t]} - H_{[t]}(v)\big).
		\end{align*}
		In the last step we used the simple inequality $C(v) \lesssim \delta^2k^3$, and we bounded the infimum by choosing $R = \xi$ and using \eqref{cond:row.sum}. 
		Applying Gr\"{o}nwall's inequality and taking the maximum over all $|v|=k$ yields
		\begin{align*}
			\widehat{H}_{[t]}^{k}
			&\lesssim
			e^{-\frac{2\gamma k}{\sigma^2}t}\widehat{H}_{[0]}^{k}
			+ \int_0^t e^{-\frac{2\gamma k}{\sigma^2}(t-s)}
			\Big( \delta^2k^2 + \frac{2\gamma k}{\sigma^2}\widehat{H}^{k+1}_{[s]} \Big)\,ds.
		\end{align*}
		Iterating this linear hierarchy exactly as in \cite{lacker2022hierarchies} leads to $\widehat{H}_{[t]}^{k} \lesssim \delta^2k^3$. This implies $h_3 := \widehat{H}_{[t]}^3  \lesssim \delta^2$, and we can apply Lemma \ref{lem: init.ent.est}(ii) along with $\widehat{C}(v) \lesssim \delta^3 k^3 + \delta^2 k^2$ for $|v|=k$;  repeating  the above argument then leads to $\widehat{H}_{[t]}^{k} \lesssim \delta^3k^3 + \delta^2k^2$.
\end{proof}

A primary motivation for our introduction of the percolation process is that this reduction to an $[n]$-indexed hierarchy appears to fail to sharply capture the average entropy. Indeed, let us argue that such a reduction would contradict the lower bound obtained in the Gaussian case, Theorem \ref{theorem:Gaussian.avg.entropy}:
Averaging \eqref{ineq:lambda.CS.diffineq}  gives
		\begin{align*}
			\frac{d}{dt}\,\overline{H}_{[t]}^k
			\lesssim
			\frac{1}{{n \choose k}} \sum_{v\subset [n] : |v|=k}\bigg(\CC(v)
			+
			 \inf_{R \in \mathcal{R}}\sum_{j \notin v}\sum_{i \in v}
			R_{ij}\big(H_{[t]}(vj) - H_{[t]}(v)\big)\bigg).
		\end{align*}
		We can estimate the average over $C(v)$ as
		\begin{align*}
		\frac{1}{{n \choose k}} \sum_{v\subset [n] : |v|=k} C(v) &\lesssim \frac{1}{{n \choose k}} \sum_{v\subset [n] : |v|=k}\sum_{i \in v} \sum_{j,\ell \in \in v} \xi_{ij}\xi_{i\ell} \\
				&= \frac{k(k-1)}{n(n-1)}\sum_{i,j=1}^n \xi_{ij}^2 + \frac{k(k-1)(k-2)}{n(n-1)(n-2)}\sum_{i,j,\ell=1}^n \xi_{ij}\xi_{i\ell} 1_{j \neq \ell}.
		\end{align*}
		In the $m$-regular graph case, this is of order $k^2/nm+k^3/n^2$.
		Now, suppose, for the sake of contradiction, that there exists $c>0$ such that for all $t$ and $k$,
		\begin{align*}
			\frac{1}{{n \choose k}} \sum_{v\subset [n] : |v|=k}\inf_{R \in \mathcal{R}}\sum_{j \notin v}\sum_{i\in v}
			R_{ij}\big(H_{[t]}(vj) - H_{[t]}(v)\big)
			\le ck \big(\overline{H}_{[t]}^{k+1} - \overline{H}_{[t]}^{k}\big).
		\end{align*}
		Then the same analysis as in \cite{lacker2022hierarchies} would yield, in the $m$-regular graph case,
		\begin{align*}
			\overline{H}_{[t]}^{k} \lesssim k^2/nm + k^3/n^2.
		\end{align*}
		This would contradict the Gaussian lower bound in Theorem \ref{theorem:Gaussian.avg.entropy}, which was of the order $k^2/nm + k/m^2$ as noted in Remark \ref{re:sharpness-avg}.

\subsection{Proof of a refinement of Proposition \ref{pr:CTMC}} \label{subsec:proof.thm.2.7}
We prove next a refinement of Proposition \ref{pr:CTMC}, which takes into account estimates on $h_3$ as in the preceding lemmas.

\begin{proposition} \label{pr:CTMC-selfimproved}
Assume $H_0([n]) < \infty$.
\begin{enumerate}[(i)]
\item If Assumption \ref{asssump:common} holds for $T < \infty$, then
	\begin{equation*}
	H_{[T]}(v) \le\inf_{R \in \mathcal{R}}\E_v\bigg[H_0(\X^{R}_T) + \int_0^T \CC(\X^{R}_t)\, dt\bigg].
	\end{equation*}
	If also $H_{[T]}(v) \le h_3$ for all $v \subset [n]$ with $|v|=3$, then 
\begin{equation*}
	H_{[T]}(v) \le\inf_{R \in \mathcal{R}}\E_v\bigg[H_0(\X^{R}_T) + \int_0^T \widehat{\CC}(\X^{R}_t)\, dt\bigg].
\end{equation*}
	where $\widehat{\CC}$ was defined in \eqref{def:CChat}.
	\item If Assumption \ref{asssump:uniform.in.time} holds, then for all $t > 0$, 
	\begin{equation*} 
	H_t(v)  \le \inf_{R \in \mathcal{R}} \E_v \left[ e^{-\sigma^2t/4\eta} H_0(\X^{R}_t) + \int_0^t e^{-\sigma^2s/4\eta } \CC(\X^{R}_s) \, ds\right].
	\end{equation*} 
	If also $\sup_{t \ge 0}H_t(v) \le h_3$ for all $v \subset [n]$ with $|v|=3$, then 
	\begin{equation*} 
	H_t(v)  \le \inf_{R \in \mathcal{R}} \E_v \left[ e^{-\sigma^2t/4\eta} H_0(\X^{R}_t) + \int_0^t e^{-\sigma^2s/4\eta } \widehat{\CC}(\X^{R}_s) \, ds\right].
\end{equation*} 
\end{enumerate} 
\end{proposition}
\begin{proof}
We essentially repeat here the argument given in Section \ref{se:intro:FPPbound}. Begin with (i). Recall the definition of the operator $\A^{R}$ from \eqref{rate.matrix}, which acts on a function $F:\R^{2^{[n]}} \to \R$ via
\begin{equation}
\A^{R} F(v) =  \sum_{j \notin v} \A^{R}_{v \to j}(F(vj) - F(v)). \label{def:Aoperator}
\end{equation}
We may then write the inequality \eqref{ineq:lambda.CS.diffineq} in Lemma \ref{lem: init.ent.est}(i) as a pointwise inequality between functions:
\begin{equation}
\frac{d}{dt} H_{[t]}   \leq \CC + \A^{R} H_{[t]} . \label{pf:hierarchy-FPP1}
\end{equation}
As mentioned before, $\A^{R}$ is the rate matrix of a (continuous-time) Markov process, in the sense that  its row sums are zero and its off-diagonal entries are nonnegative. In particular, the associated semigroup $e^{t\A^{R}}$ leaves invariant the set of nonnegative functions on $2^{[n]}$. Hence, by reversing time and applying \eqref{pf:hierarchy-FPP1}, we have
\begin{align*}
\frac{d}{dt}\Big( e^{t\A^{R}} H_{[T-t]}\Big) = e^{t\A^{R}} \Big(\A^{R} H_{[T-t]} + \frac{d}{dt}H_{[T-t]}\Big)  \ge - e^{t\A^{R}}\CC,
\end{align*}
Integrate this to find
\begin{align*}
e^{T\A^{R}}H_{[0]} \ge H_{[T]} - \int_0^T e^{t\A^{R}}\CC\,dt.
\end{align*}
Recall the probabilistic expression $e^{t\A^{R}}F(v)=\E_v[F(\X^{R}_t)]$ for the semigroup, where $\X^{R}$ is the percolation process and $\E_v$ denotes expectation starting from $\X^{R}_0=v$. Hence, rearranging the previous inequality yields the first claim in (i). For the second claim, we simply apply part (ii) of Lemma \ref{lem: init.ent.est}(ii) instead of part (i), and repeat the argument. 
	
The proof of part (ii) is similar. As a technical point, Lemma \ref{lem: init.ent.est-unif} does not exactly provide a differential inequality, because we do not know a priori that $t \mapsto H_t(v)$ is differentiable.
If it were differentiable, we could write \eqref{ineq:lambda.CS.diffineq.uit} in Lemma \ref{lem: init.ent.est-unif}(i) as the following pointwise inequality between functions,
\begin{align*}
\frac{d}{dt} H_{t}   \leq \CC + \A^{R} H_{t} - \frac{\sigma^2}{4\eta} H_t.
\end{align*}
Hence,
\begin{align*}
\frac{d}{dt}\Big(e^{-(\sigma^2/4\eta)t} e^{t\A^{R}} H_{ T-t }\Big) = e^{-(\sigma^2/4\eta)t}e^{t\A^{R}} \Big(\A^{R} H_{ T-t } -  \frac{\sigma^2}{4\eta} H_{T-t} + \frac{d}{dt}H_{ T-t }\Big)  \ge - e^{-(\sigma^2/4\eta)t} e^{t\A^{R}}\CC,
\end{align*}
which we integrate to find
\begin{align*}
e^{-(\sigma^2/4\eta)T} e^{T\A^{R}}H_0 \ge H_T - \int_0^T e^{-(\sigma^2/4\eta)t}e^{t\A^{R}}\CC\,dt.
\end{align*}
In probabilistic notation, this yields \eqref{ctmc.hierarchy.uit.1}. To address the issue that $t \mapsto H_t(v)$ might not be differentiable, we simply mollify, taking limits easily in light of the uniform bound $\sup_{t \in [0,T]}H_t(v) \le H_{[T]}(v) < \infty$ for any $T > 0$. 
\end{proof}

\section{Expectation estimates for the percolation process} \label{se:percolationproofs}

We have now completed the proof of Proposition \ref{pr:CTMC-selfimproved}, which bounds the entropies $H_{[t]}(v)$ and $H_t(v)$ in terms of quantities of the form $ \E_v[F(\X^{R}_T)]$, with $\X^{R}$ being the percolation process. Recall that these expectations can be expressed in terms of the semigroup of the percolation process,
\begin{equation}
\E_v[F(\X^{R}_t)] = e^{t\A^{R}}F(v) =  \sum_{m=0}^\infty \frac{t^m}{m!}(\A{^{R}})^m F(v). \label{eq:semigroup-series}
\end{equation}
In this section we  estimate the expectations for eight functions $F$. In Section \ref{se:proofs-concrete}, we will put these estimates to use in order to prove the theorems stated in Section \ref{se:concrete}.
The functions $F$ of interest to us are those which arise from bounding $\CC$ as well as $\widehat{\CC}$, which were defined respectively in \eqref{eq:def.C(v).D(v)} and \eqref{def:CChat}. To write these functions succinctly, we will use the notation $1_v$ to denote the $n$-vector with ones for the coordinates in $v \subset [n]$ and zeroes otherwise, and we define $\widehat{\xi}_{ij}=\xi_{ij}^2$ as the entrywise (Hadamard) square of $\xi$. 
\begin{itemize}
\item The bound on the maximum entropy in Theorem \ref{theorem:max} starts by using the crude bound $\xi_{ij}\le\delta$ for all $i,j$:
\begin{equation*}
\CC(v) = \frac{M}{\sigma^2} \sum_{i\in v}\bigg( \sum_{j \in v } \xi_{ij}\bigg)^{2} \le \frac{M}{\sigma^2} \delta^2 |v|^3,
\end{equation*}
where we recall that $\delta=\max_{ij}\xi_{ij}$.
This leads us to study the quantity $\E_v|\X_t|^3$, which turns out to require first estimating $\E_v|\X_t|^2$ and $\E_v|\X_t|$.
\item The bound on the average entropy in Corollary \ref{co:avg.1way} starts from the sharper bound
\begin{equation*}
\CC(v) \le \frac{M}{\sigma^2} |v|^2 \sum_{i\in v}\delta_i^2 = \frac{M}{\sigma^2} |v|^2 \langle 1_v,x\rangle, \qquad x=(\delta_1^2,\ldots,\delta_n^2),
\end{equation*}
where we recall that $\delta_i=\max_j\xi_{ij}$ is the row-maximum. This leads us to study the quantity $\E_v[|\X_t|^2\langle 1_{\X_t}, x\rangle]$, which turns out to require first estimating $\E_v[|\X_t| \langle 1_{\X_t}, x\rangle]$ and $\E_v[ \langle 1_{\X_t}, x\rangle]$.
\item The sharper bound on the average entropy in Theorem \ref{theorem:avg.2way.asym} starts from the Cauchy-Schwarz inequality,
\begin{equation*}
\CC(v) \le \frac{M}{\sigma^2}|v| \sum_{i,j\in v} \xi_{ij}^2 = \frac{M}{\sigma^2}|v| \langle 1_v , \widehat{\xi} 1_v \rangle .
\end{equation*}
This leads us to study the quantity $\E_v[|\X_t| \langle 1_{\X_t} , \widehat{\xi} 1_{\X_t} \rangle ]$, and which turns out to require first estimating $\E_v[ \langle 1_{\X_t} , \widehat{\xi} 1_{\X_t} \rangle ]$.
\end{itemize}

For an $n \times n$ matrix $G=(G_{ij})$, we write $G_{\mathrm{diag}}$ for the vector of diagonal entries of $G$:
\begin{equation}
G_{\mathrm{diag}} = (G_{11},\ldots,G_{nn}).
\end{equation}
Recall the constant of $0 < \gamma < \infty$ in Assumption \ref{asssump:common}(iii).

\begin{proposition} \label{pr:expectations}
Consider a vector $x \in \R^n$ and an $n \times n$ matrix $G$, both having nonnegative entries. Assume that the matrix $R$ has row sums bounded by $1$, i.e., 
			\begin{equation} \tag{R-rows}
		\max_{1\leq i\leq n} \sum_{j=1}^n R_{ij} \leq 1. \label{cond:R row.sum}
	\end{equation}  Then we have the following estimates, for any $t \ge 0$ and $v \subset [n]$:
\begin{enumerate}[(i)]
\item Polynomial in $|\X^{R}_t|$:
\begin{alignat*}{3}
	&(\mathrm{a}) \quad && \E_v|\X^{R}_t|	&&		\le e^{2\gamma t/\sigma^2 }|v|  \\
	&(\mathrm{b}) \quad && \E_v|\X^{R}_t|^2	&&	\le 2 e^{4\gamma t/\sigma^2}|v|^2 \\
	&(\mathrm{c}) \quad && \E_v|\X^{R}_t|^3	&&	\le 8  e^{6\gamma t/\sigma^2} |v|^3 \hphantom{...................................................................................................}
\end{alignat*}
\item Linear functions of $1_{\X^{R}_t}$:
\begin{alignat*}{3}
&(\mathrm{a}) \quad  && \E_v[\langle 1_{\X^{R}_t} , x\rangle]  && \le \langle 1_v, e^{2\gamma t R /\sigma^2}x \rangle \\
&(\mathrm{b}) \quad  && \E_v[|\X^{R}_t|\langle 1_{\X^{R}_t} , x\rangle ] && \le |v|\langle 1_v,e^{2\gamma t (I+R )/\sigma^2}(I+R)x\rangle   \\
&(\mathrm{c}) \quad &&\E_v[|\X^{R}_t|^2\langle 1_{\X^{R}_t} , x\rangle ] && \le 2|v|^2 \big\langle 1_v, e^{2\gamma t (2I+R )/\sigma^2}(I+R )^2 x\big\rangle
\hphantom{.......................................................}
\end{alignat*}
\item Quadratic functions of $1_{\X^{R}_t}$: Letting $G_t=e^{2\gamma t R/\sigma^2} G e^{2\gamma t R^\top/\sigma^2}$,
\begin{alignat*}{3}
&(\mathrm{a}) \quad  && \E_v[ \langle 1_{\X^{R}_t}, G 1_{\X^{R}_t}\rangle ] && \le \big\langle 1_v, G_t 1_v\big\rangle + \frac{\gamma}{\sigma^2} \int_0^t  \big\langle 1_v, R e^{2\gamma (t-s)R/\sigma^2} (G_s)_{\mathrm{diag}}\big\rangle \,ds   \\
&(\mathrm{b}) \quad  && \E_v[|\X^{R}_t|\langle 1_{\X^{R}_t}, G 1_{\X^{R}_t}\rangle] && \le |v| e^{2\gamma t/\sigma^2 } \big\langle 1_v, (R G_t + G_t R^\top + G_t)1_v\big\rangle   \\
& \  && \  && \quad + \frac{2\gamma}{\sigma^2} |v| e^{2\gamma t/\sigma^2} \int_0^t  \big\langle 1_v,e^{2\gamma (t-s)  R /\sigma^2}(I+R)R(R G_s + G_sR^\top + 2 G_s)_{\mathrm{diag}}\big\rangle   \,ds
\end{alignat*}
\end{enumerate} 
\end{proposition}

In fact, part (i) follows from (ii) by taking $x=1$ to be the all-ones vector and using the assumption \eqref{cond:R row.sum}. Similarly, part (ii) follows from (iii) by taking $G=x1^\top$. 
Nonetheless, we give separate proofs for each claim, because the earlier ones are shorter  and serve as good warmups.
The rest of the section is devoted to the proof of Proposition \ref{pr:expectations}. 
Our approach will start from the formula
\begin{equation}
\frac{d}{dt}\E_v[ F(\X^{R}_t)] = \frac{d}{dt} e^{t \A^{R}} F(v) = e^{t\A^{R}} \A^{R} F(v) = \E_v[\A^{R} F(\X^{R}_t)]. \label{eq:generator}
\end{equation}
Then, we will try to bound $\A^{R} F$ from above in terms of $F$ itself, or other functions for which we have already computed expectations, so that we obtain  an estimate of $\E_v[ F(\X^{R}_t)]$ using Gronwall's inequality. 
We will use repeatedly the basic formula
\begin{equation}
\A^{R} F(v) = \frac{2\gamma}{\sigma^2}\sum_{j \notin v}\Big(\sum_{i \in v} R_{ij}\Big) (F(vj)-F(v)), \label{eq:AF-basic}
\end{equation}
which comes from the definition of $\A^{R}_{v\to j}$ in \eqref{eq:def.C(v).D(v)}.
Moreover, a convenient abuse of notation will be to write $\A^{R}[F(v)]$ in place of $\A^{R} F(v)$. For example, $\A^{R} [|v|^2]$ will stand for $\A^{R} F(v)$, where $F(v)=|v|^2$.

\subsection{Polynomials}

In this section we prove part (i) of Proposition \ref{pr:expectations}. We begin with a more general lemma.

\begin{lemma} \label{le:AF-polynomial}
Let $\ell \ge 1$. Then, for $v \subset [n]$,
\begin{align}
		\A^{R}[|v|^\ell] \le \frac{2\gamma}{\sigma^2}   |v|\big((|v|+1)^{\ell}-|v|^\ell\big). \label{pf:polynomial1}
		\end{align}
\end{lemma}
\begin{proof}
To avoid notational clutter, we assume without loss of generality that $\sigma = \sqrt{2}$. The general case follows by replacing $\gamma $ with $2\gamma/\sigma^2$. Let $F(v)=|v|^\ell$.  For $j \notin v$ we have $F(vj) - F(v) = (|v|+1)^\ell - |v|^\ell$.
We then apply \eqref{eq:AF-basic} and recall from Assumption  \eqref{cond:R row.sum}  that row sums of $R$ are bounded by $1$:
		\begin{align*}
			\A^{R}  F(v) & = \gamma \sum_{j \notin v}\bigg( \sum_{i \in v}R_{ij}\bigg)\big((|v|+1)^\ell - |v|^\ell\big)  \le \gamma |v| \big((|v|+1)^\ell - |v|^\ell\big). \qedhere
		\end{align*}  
\end{proof}

Using Lemma \ref{le:AF-polynomial} with $\ell=1$, we have $\A^{R}[|v|] \le \gamma|v|$, and thus from \eqref{eq:generator} we deduce
\begin{equation*}
\frac{d}{dt}\E_v|\X^{R}_t| \le \gamma \E_v|\X^{R}_t|.
\end{equation*}
Since $\E_v|\X^{R}_0|=|v|$, from Gronwall's inequality we get $\E_v|\X^{R}_t| \le e^{\gamma t}|v|$, which is Proposition \ref{pr:expectations}(ia). To prove Proposition \ref{pr:expectations}(ib), we apply Lemma \ref{le:AF-polynomial} with $\ell=2$ to get $\A^{R}[|v|^2] \le \gamma |v|(2|v|+1)$, which we plug into \eqref{eq:generator} to find
\begin{equation*}
\frac{d}{dt}\E_v|\X^{R}_t|^2 \le 2\gamma \E_v|\X^{R}_t|^2 +  \gamma \E_v|\X^{R}_t|.
\end{equation*}
Using Gronwall's inequality and Proposition \ref{pr:expectations}(ia),
\begin{align*}
\E_v|\X^{R}_t|^2 &\le e^{2\gamma t}\E_v|\X^{R}_0|^2 +  \gamma \int_0^t e^{2\gamma(t-s)}\E_v|\X^{R}_s|\,ds \\
	&\le e^{2\gamma t}|v|^2  +  \gamma |v| \int_0^t e^{2\gamma(t-s)}e^{\gamma s}\,ds \\
	&= e^{2\gamma t}|v|^2  +  e^{2\gamma t}(1-e^{-\gamma t})|v| .
\end{align*}
Proposition \ref{pr:expectations}(ib) follows quickly. To prove Proposition \ref{pr:expectations}(ic), we apply Lemma \ref{le:AF-polynomial} with $\ell=3$ to get $\A^{R}[|v|^3] \le  \gamma |v|(3|v|^2+3|v|+1)$, which we plug into \eqref{eq:generator} to find
\begin{equation*}
\frac{d}{dt}\E_v|\X^{R}_t|^3 \le 3\gamma \E_v|\X^{R}_t|^3 + 3\gamma \E_v|\X^{R}_t|^2 + \gamma \E_v|\X^{R}_t|.
\end{equation*}
By Gronwall's inequality and parts (ia,b),
\begin{align*}
\E_v|\X^{R}_t|^3 &\le e^{3\gamma t}\E_v|\X^{R}_0|^3 +  \gamma \int_0^t e^{3\gamma (t-s)}\big( 3\E_v|\X^{R}_s|^2 + \E_v|\X^{R}_s|\big)\,ds \\
	&\le e^{3\gamma t}|v|^3 + \gamma \int_0^t e^{3\gamma (t-s)}\big( 6e^{2\gamma s}|v|^2 + e^{\gamma s}|v| \big)\,ds \\
	&= e^{3\gamma t}|v|^3 +   e^{3\gamma t}\big(6|v|^2(1-e^{-\gamma t}) +  \tfrac12|v|(1-e^{-2\gamma t})\big).
\end{align*}
Discarding terms yields Proposition \ref{pr:expectations}(ic).

\begin{remark}
In the proof of Lemma \ref{le:AF-polynomial}, and below, we repeatedly bound $\sum_{j \notin v}$ by $\sum_{j \in [n]}$. Our rough intuition is that this does not lose too much  because we view $|v|$ as much smaller than $n$. From a practical standpoint, it is hard to imagine obtaining a tractable estimate without using such a bound, as it is what lets us close the Gronwall loop.
\end{remark}

\subsection{Linear functions}

In this section we prove part (ii) of Proposition \ref{pr:expectations}, and we again begin with a lemma.

\begin{lemma} \label{le:AF-linear}
Let $x \in \R^n$ have nonnegative entries, and let $\ell \ge 0$ be an integer. Let $v \subset [n]$. 
\begin{equation}
\A^{R}[|v|^\ell \langle 1_v,x\rangle ] 
 \le \frac{2\gamma}{\sigma^2}(|v|+1)^\ell \langle 1_v, R x\rangle + \frac{2\gamma}{\sigma^2} |v| \big((|v|+1)^{\ell} - |v|^\ell\big)\langle 1_v, x\rangle . \label{pf:linear1}
\end{equation}
\end{lemma}
\begin{proof}
As in the proof of Lemma \ref{le:AF-polynomial}, we assume without loss of generality that $\sigma = \sqrt{2 }$. Let $F(v) = |v|^\ell \langle 1_v,x\rangle$. For $j\notin v$ we have
		\begin{align*}
			F(vj) - F(v) &= (|v|+1)^\ell\sum_{i\in vj}x_i - |v|^\ell\sum_{i\in v}x_i \\
				&= (|v|+1)^\ell x_j +  \big((|v|+1)^{\ell} - |v|^\ell\big)\sum_{i\in v}x_i.
		\end{align*}
		Plugging this into \eqref{eq:AF-basic} and recalling that $R$ has row sums bounded, we have
		\begin{align*}
			\A^{R} F(v) 	&\le \gamma\sum_{j \notin v}\bigg( \sum_{i \in v}R_{ij}\bigg) \bigg((|v|+1)^\ell x_j + \big((|v|+1)^{\ell} - |v|^\ell\big)\sum_{i\in v}x_i \bigg)\\
			&\le \gamma(|v|+1)^\ell \langle 1_v, R x\rangle + \gamma |v| \big((|v|+1)^{\ell} - |v|^\ell\big)\langle 1_v, x\rangle. \qedhere
		\end{align*}
\end{proof}

Let us now prove Proposition \ref{pr:expectations}(ii), again assuming without loss of generality that $\sigma = \sqrt{2}$. We begin with by proving Proposition \ref{pr:expectations}(iia). Starting from \eqref{eq:generator} and applying Lemma \ref{le:AF-linear} with $\ell=0$,
\begin{equation*}
\frac{d}{dt} \E_v[\langle 1_{\X^{R}_t},x\rangle ] \le \gamma \E_v[\langle 1_{\X^{R}_t}, R x\rangle ].
\end{equation*}
Applying this with $x$ as basis vectors yields the coordinatewise inequality between vectors,
\begin{equation*}
\frac{d}{dt} \E_v[ 1_{\X^{R}_t} ] \le \gamma R^\top \E_v[  1_{\X^{R}_t}  ].
\end{equation*}
Because $R$ has nonnegative entries, so does the matrix exponential $e^{s R^\top}$ for any $s \ge 0$. Hence, for any $t > s > 0$, we have coordinatewise that
\begin{equation*}
\frac{d}{ds}\big( e^{\gamma sR^\top}\E_v[ 1_{\X^{R}_{t-s}} ]\big) \ge 0.
\end{equation*}
Integrate to find
\begin{equation}
\E_v[ 1_{\X^{`}_t} ] \le e^{\gamma tR^\top}\E_v[ 1_{\X_0} ] =  e^{\gamma tR^\top} 1_v.   \label{pf:exp-linear1}
\end{equation}
Taking the inner product with $x$ proves Proposition \ref{pr:expectations}(iia).

To prove Proposition \ref{pr:expectations}(iib), we use \eqref{eq:generator} and apply Lemma \ref{le:AF-linear} with $\ell=1$ to get
\begin{equation*}
\frac{d}{dt} \E_v[|\X^{R}_t|\langle 1_{\X^{R}_t},x\rangle ] \le \gamma \E_v[ (|\X^{R}_t|+1)\langle 1_{\X^{R}_t},R x\rangle + |\X^{R}_t|\langle 1_{\X^{R}_t}, x\rangle ].
\end{equation*}
Applying this with $x$ as basis vectors yields the coordinatewise inequality between vectors,
\begin{equation*}
\frac{d}{dt} \E_v[ |\X^{R}_t|1_{\X^{R}_t} ] \le \gamma(I + R^\top) \E_v[ |\X^{R}_t| 1_{\X^{R}_t}] + \gamma R^\top \E_v[1_{\X^{R}_t}] .
\end{equation*}
Integrating this as in the previous step and then recalling \eqref{pf:exp-linear1} yields
\begin{align}
\E_v[ |\X^{R}_t|1_{\X^{R}_t} ] &\le e^{\gamma t (I+R^\top)}\E_v[ |\X^{R}_0|1_{\X^{R}_0} ] + \gamma \int_0^t e^{\gamma(t-s) (I+R^\top)}  R^\top \E_v[1_{\X^{R}_s}]\,ds \nonumber \\
	&\le \bigg(e^{\gamma t (I+R^\top)}|v| + \gamma \int_0^t e^{\gamma(t-s) (I+R^\top)}  R^\top  e^{\gamma sR^\top} \,ds\bigg)1_v  \nonumber \\
	&= \bigg(e^{\gamma t (I+R^\top)}|v| + e^{\gamma t(I+R^\top)} R^\top (1-e^{-\gamma t})\bigg)1_v \nonumber \\
	&\le |v|e^{\gamma t(I+R^\top)}(I+R^\top)1_v. \label{pf:exp-linear2}
\end{align}
Taking the inner product with $x$ yields Proposition \ref{pr:expectations}(iib).

To prove Proposition \ref{pr:expectations}(iic), we apply Lemma \ref{le:AF-linear} with $\ell=2$ to get
\begin{align*}
\frac{d}{dt} \E_v[|\X^{R}_t|^2 \langle 1_{\X^{R}_t},x\rangle ] &\le \gamma \E_v[ (|\X^{R}_t|+1)^2\langle 1_{\X^{R}_t},R x\rangle +  |\X^{R}_t|(2|\X^{R}_t|+1)\langle 1_{\X^{R}_t}, x\rangle ] \\
	&= \gamma \E_v[ |\X^{R}_t|^2\langle 1_{\X^{R}_t},(2I+R) x\rangle] + \gamma \E_v[ |\X^{R}_t| \langle 1_{\X^{R}_t},(I+2R) x\rangle] + \gamma \E_v[  \langle 1_{\X^{R}_t},R x\rangle].
\end{align*}
Applying this with $x$ as basis vectors yields the coordinatewise inequality between vectors,
\begin{equation*}
\frac{d}{dt} \E_v[ |\X^{R}_t|^21_{\X^{R}_t} ] \le \gamma(2I + R^\top) \E_v[ |\X^{R}_t|^2 1_{\X^{R}_t}] +  \gamma (I + 2R^\top) \E_v[|\X^{R}_t|1_{\X^{R}_t}] + \gamma R^\top\E_v[1_{\X^{R}_t}] .
\end{equation*}
Integrate this and plug in \eqref{pf:exp-linear1} and \eqref{pf:exp-linear2} to get
\begin{align*}
\E_v[ |\X^{R}_t|^21_{\X_t} ] &\le e^{\gamma t (2I+R^\top)}\E_v[ |\X^{R}_0|^21_{\X^{R}_0} ] \\
& \qquad + \gamma \int_0^t e^{\gamma(t-s)(2I+R^\top)} \Big(  ( I+2R^\top) \E_v[|\X^{R}_s|1_{\X^{R}_s}] +  R^\top\E_v[1_{\X^{R}_s}]\Big)\,ds \\
	&\le e^{\gamma t (2I+R^\top)}|v|^21_{v}   + \gamma |v| \int_0^t e^{\gamma(t-s)(2I+R^\top)} ( I+2R^\top) e^{\gamma s (I+R^\top)}\Big(I +  R^\top  \Big) 1_v  \,ds \\
	&\qquad + \gamma \int_0^t e^{\gamma(t-s)(2I+R^\top)} R^\top e^{\gamma s R^\top} 1_v\,ds \\
	&= e^{\gamma t (2I+R^\top)} |v|^2 1_v  +  |v|(1-e^{-\gamma t})e^{\gamma t (2I+R^\top)}(I + 2R^\top)(I + R^\top)1_v \\
	&\qquad + \frac12(1-e^{-2\gamma t})e^{\gamma t (2I+R^\top)}R^\top 1_v \\
	&\le |v|^2e^{\gamma t (2I+R^\top)}\bigg( I + (I + 2R^\top)( I + R^\top) + R^\top \bigg)1_v \\ 
	&= 2|v|^2 e^{\gamma t (2I+R^\top)}(I+R^\top)^21_v.
\end{align*}
Take the inner product with $x$ to get Proposition \ref{pr:expectations}(iic).

\subsection{Quadratic functions}

We finally prove part (iii) of Proposition \ref{pr:expectations}, which is the most involved. We begin with a lemma estimating the action of $\A$ on relevant functions:

\begin{lemma} \label{le:AF-quadratic}
	Let $ G $ be  an $n \times n$ matrix with nonnegative entries. Let $v \subset [n]$. Then
		\begin{align} 
			\A^{R} [\langle 1_v, G1_v\rangle ] &\le \frac{2\gamma}{\sigma^2} \langle 1_v , R G_{\mathrm{diag}}\rangle  + \frac{2\gamma}{\sigma^2}   \langle 1_v, \big( R G +GR^\top \big) 1_v\rangle , \label{eq: AF.bound.0} \\
			\A^{R} [  |v| \langle 1_v, G1_v\rangle ]  &\le \frac{2\gamma}{\sigma^2} (|v|+ 1) \Big[  \langle 1_v , R G_{\mathrm{diag}}\rangle  + \langle 1_v, \big( R G +GR^\top  \big) 1_v\rangle\Big] + \frac{2\gamma}{\sigma^2}  |v|\langle 1_v,G1_v\rangle .  \label{eq: AF.bound.1}
		\end{align} 
\end{lemma}
\begin{proof} 
As in the proofs of Lemmas \ref{le:AF-polynomial} and \ref{le:AF-linear}, we assume without loss of generality that $\sigma = \sqrt{2}$. We start with \eqref{eq: AF.bound.0}. Let $F(v) = \langle 1_v, G1_v\rangle = \sum_{i,r \in v}G_{ir}$. 	For $j \notin v$ we compute
		\begin{align*}
			F(vj) - F(v) &= \sum_{i,r \in vj}G_{ir} - \sum_{i,r \in v}G_{ir} = G_{jj} + \sum_{r \in v}(G_{rj} + G_{jr}).
		\end{align*}
		Thus, using \eqref{eq:AF-basic} and the nonnegativity of the entries of $R$,
		\begin{align*}
			\A^{R} F(v) &=  \gamma \sum_{j \notin v}\bigg(\sum_{i \in v} R_{ij}\bigg)\bigg(G_{jj} + \sum_{r \in v}(G_{rj} + G_{jr})\bigg) \\
			&\le \gamma \sum_{i \in v}\sum_{j=1}^n R_{ij}G_{jj} + \gamma \sum_{i,r \in v} \sum_{j =1}^n (R_{ij}G_{rj} + R_{ij}G_{jr}) \\
			&= \gamma \langle 1_v , R G_{\mathrm{diag}}\rangle  + \gamma  \langle 1_v, \big( R G +GR^\top \big) 1_v\rangle .
		\end{align*} 
	We next turn to \eqref{eq: AF.bound.1}. Set $F(v)=|v|\langle 1_v, G1_v\rangle$. For $j \notin v$ we compute  
		\begin{align*}
			F(vj) - F(v) &= (|v| +1)\sum_{\ell, r\in vj}G_{\ell r} -  |v| \sum_{\ell, r\in v}G_{\ell r} \\
			&= |v| \bigg(G_{jj} + \sum_{r \in v}(G_{rj} + G_{jr})\bigg) + \sum_{\ell, r\in vj} G_{\ell r}.
		\end{align*}
		Thus
		\begin{align*}
			\A^{R}  F(v) &= \gamma \sum_{j \notin v}\bigg(\sum_{i \in v} R_{ij}\bigg)\bigg(  |v| \bigg(G_{jj} + \sum_{r \in v}(G_{rj} + G_{jr})\bigg) + \sum_{\ell, r\in vj} G_{\ell r} \bigg) \\
			&\le \gamma  |v|  \sum_{i \in v}\sum_{j=1}^n R_{ij}G_{jj} + \gamma  |v| \sum_{i,r \in v} \sum_{j =1}^n \Big(R_{ij}G_{rj} + R_{ij}G_{jr}\Big) 
			+ \gamma \sum_{i\in v}\sum_{j \notin  v} R_{ij} \sum_{\ell, r\in vj} G_{\ell r}  \\
			&= \gamma |v|  \langle 1_v, R G_{\mathrm{diag}}\rangle  + \gamma |v|  \langle 1_v, \big( R G +GR^\top \big) 1_v\rangle + \gamma \sum_{i\in v}\sum_{j\notin  v} R_{ij} \sum_{\ell, r\in vj} G_{\ell r}.
		\end{align*}
		We simplify the last term by splitting the sum into four cases, depending on whether $\ell$ and $r$ equal $j$, or both, or neither:
		\begin{align*}
			\sum_{i\in v}\sum_{j\notin v} R_{ij} \sum_{\ell, r\in vj} G_{\ell r} 
			&\le \sum_{i\in v} \bigg( \sum_{\ell,r\in v} G_{\ell r} + \sum_{\ell\in v} \sum_{j=1}^n R_{ij}G_{\ell j} + \sum_{r\in v} \sum_{j=1}^nR_{ij}G_{j r} + \sum_{j=1}^nR_{ij}G_{j j} \bigg) \\
			&=  |v| \langle 1_v, G 1_v \rangle + \langle 1_v, GR^\top 1_v\rangle  + \langle 1_v, R G 1_v \rangle + \langle 1_v,R G_{\mathrm{diag}}\rangle,
		\end{align*}
		where we used the assumption  \eqref{cond:R row.sum} to remove the $R$ in the first step, and we used nonnegativity of the entries of $R$ and $G$ throughout. 
		Combining this with the previous inequality yields
		\begin{align*}
			\A^{R} F(v) &\le \gamma(|v|+ 1) \langle 1_v,R G_{\mathrm{diag}}\rangle + \gamma (|v|+ 1) \langle 1_v , \big( R G +GR^\top \big) 1_v\rangle + \gamma|v| \langle 1_v, G 1_v \rangle. \qedhere
		\end{align*} 	 
\end{proof}

Let us now prove Proposition \ref{pr:expectations}(iiia). Starting from \eqref{eq:generator} and applying \eqref{eq: AF.bound.0} from Lemma \ref{le:AF-quadratic},
\begin{equation}
\frac{d}{dt} \E_v[\langle 1_{\X^{R}_t},G1_{\X^{R}_t}\rangle ] \le \gamma \E_v\big[\langle 1_{\X^{R}_t}, R G_{\mathrm{diag}}\rangle + \langle 1_{\X^{R}_t}, (R G+ GR^\top)1_{\X^{R}_t}\rangle \big]. \label{pf:quadratic-ineq1}
\end{equation}
Let $\langle A,B\rangle = \tr(AB^\top)$ denote the Frobenius inner product for $n \times n$ matrices.
Let $A_t$ be the $n \times n$ diagonal matrix given by
\begin{equation*}
(A_t)_{ij} = \E_v\bigg[\sum_{r \in \X^{R}_t}R_{ri}\bigg]1_{i = j},
\end{equation*}
which is defined in this way so that, for any matrix $G$,
\begin{equation*}
\langle A_t, G\rangle = \sum_{i=1}^n \E_v\bigg[\sum_{r \in \X^{R}_t}R_{ri}G_{ii}\bigg] = \E_v\big[ \langle 1_{\X^{R}_t},R G_{\mathrm{diag}}\rangle \big].
\end{equation*}
Defining the symmetric matrix $R_t = \E_v[1_{\X^{R}_t}1_{\X^{R}_t}^\top]$, we may write \eqref{pf:quadratic-ineq1} in duality as
\begin{equation*}
\frac{d}{dt}\langle R_t,G\rangle \le \gamma \langle A_t,G\rangle + \gamma \langle R_tR  +R^\top R_t, G \rangle.
\end{equation*}
This holds for every matrix $G$ with nonnegative entries, and we deduce the coordinatewise inequality
\begin{equation*}
\frac{d}{dt}R_t \le \gamma\big( A_t + R_tR + R^\top R_t\big).
\end{equation*}
Because $e^{\gamma s R}$ has nonnegative entries, for each $t \ge s \ge 0$ we deduce
\begin{equation*}
\frac{d}{ds}\big(e^{\gamma sR^\top}R_{t-s} e^{\gamma sR} \big) \ge -\gamma e^{\gamma sR^\top}A_{t-s} e^{\gamma sR} .
\end{equation*}
Integrate to find
\begin{equation}
 R_t \le e^{\gamma tR^\top}R_0 e^{\gamma tR} + \gamma \int_0^t e^{\gamma (t-s)R^\top}A_s e^{\gamma (t-s)R} \,ds. \label{pf:ineq-quadratic2}
\end{equation}
We next take the inner product on both sides with the given matrix $G$ with nonnegative entries.
Note that $R_0=1_v1_v^\top$ and recall that $G_t = e^{\gamma tR}Ge^{\gamma  tR^\top}$, so that
\begin{equation*}
\langle e^{\gamma tR^\top}R_0 e^{\gamma tR},G\rangle = \tr\big(Ge^{\gamma tR^\top}1_v1_v^\top e^{\gamma tR}\big) = \langle 1_v, G_t 1_v\rangle.
\end{equation*} 
Recalling the definition of $A_s$, we have also
\begin{equation*}
\langle e^{\gamma (t-s)R^\top}A_s e^{\gamma (t-s)R},G\rangle = \langle A_s, e^{\gamma (t-s)R}Ge^{\gamma (t-s)R^\top} \rangle = \E_v\big[\langle 1_{\X_s}, R (G_{t-s})_{\mathrm{diag}}\rangle\big].
\end{equation*}
Hence, if we multiply \eqref{pf:ineq-quadratic2} by $G$ and use Proposition \ref{pr:expectations}(iia), we get
\begin{align*}
\E_v[\langle 1_{\X^{R}_t},G1_{\X^{R}_t}\rangle ] &= \langle R_t,G\rangle  \le \langle 1_v, G_t 1_v\rangle  + \gamma \int_0^t \E_v\big[\langle 1_{\X^{R}_s}, R (G_{t-s})_{\mathrm{diag}}\rangle\big] \,ds \\
	&\le \langle 1_v, G_t 1_v\rangle  + \gamma \int_0^t \langle 1_v, e^{\gamma sR} R (G_{t-s})_{\mathrm{diag}}\rangle  \,ds .
\end{align*}
This proves Proposition \ref{pr:expectations}(iiia).

We finally turn to Proposition \ref{pr:expectations}(iiib), adopting a similar strategy.
Starting from \eqref{eq:generator} and applying \eqref{eq: AF.bound.1} from Lemma \ref{le:AF-quadratic},
\begin{equation}
\begin{split}
\frac{d}{dt} \E_v[|\X^{R}_t|\langle 1_{\X^{R}_t},G1_{\X^{R}_t}\rangle ] &\le \gamma \E_v\big[(|\X^{R}_t|+1)\big(\langle 1_{\X^{R}_t}, R G_{\mathrm{diag}}\rangle + \langle 1_{\X^{R}_t}, (RG+ GR^\top)1_{\X^{R}_t}\rangle\big) \big] \\
	&\qquad + \gamma \E_v[|\X^{R}_t|\langle 1_{\X^{R}_t},G1_{\X^{R}_t} \rangle].
\end{split}\label{pf:quadratic3}
\end{equation}
We will translate this into a coordinatewise differential inequality for the matrix $\widetilde{R}_t = \E_v[|\X^{R}_t|1_{\X^{R}_t}1_{\X^{R}_t}^\top]$. Define also $R_t=\E_v[1_{\X^{R}_t}1_{\X^{R}_t}^\top]$ as above, and define a diagonal matrix $\widetilde{A}_t$ by
\[
(\widetilde{A}_t)_{ij} = \E_v\bigg[(|\X^{R}_t|+1)\sum_{r \in \X^{R}_t}R_{ri}\bigg]1_{i = j},
\]
so that, for any matrix $G$,
\[
\langle \widetilde{A}_t, G\rangle = \E_v\big[(|\X^{R}_t|+1) \langle 1_{\X^{R}_t}, RG_{\mathrm{diag}}\rangle\big].
\]
With these definitions, we can write \eqref{pf:quadratic3} as
\begin{equation*}
	\frac{d}{dt}\langle\widetilde{R}_t,G\rangle
	\le
	\gamma \langle \widetilde{A}_t ,G\rangle
	+
	\gamma \langle R_tR + R^\top R_t,G\rangle
	+
	\gamma \langle \widetilde{R}_tR + R^\top \widetilde{R}_t + \widetilde{R}_t,G\rangle,
\end{equation*}
which means coordinatewise that
\[
\frac{d}{dt}\widetilde{R}_t
\le
\gamma \widetilde{A}_t
+
\gamma (R_tR + R^\top R_t)
+
\gamma (\widetilde{R}_tR + R^\top \widetilde{R}_t + \widetilde{R}_t).
\]
We may integrate this as in \eqref{pf:ineq-quadratic2} to get
\begin{equation*}
	\widetilde{R}_t
	\le
	e^{\gamma t} e^{\gamma tR^\top}\widetilde{R}_0 e^{\gamma tR}
	+
	\gamma \int_0^t e^{\gamma (t-s)}e^{\gamma (t-s)R^\top}
	\big(\widetilde{A}_s + R_sR + R^\top R_s\big)
	e^{\gamma (t-s)R} \,ds.
\end{equation*}
Using $\widetilde{R}_0=|v|1_v1_v^\top$, we take the inner product with $G$ to get
\begin{align*}
\E_v[ &|\X^{R} _t|\langle 1_{\X^{R} _t},G1_{\X^{R} _t}\rangle ] = \langle \widetilde{R}_t,G\rangle \\
	&\le e^{\gamma t}|v| \langle e^{\gamma tR^\top}1_v1_v^\top e^{\gamma tR}, G\rangle + \gamma \int_0^t e^{\gamma (t-s)} \Big\langle e^{\gamma (t-s)R^\top} \big(\widetilde{A}_s + R_sR + R^\top R_s\big) e^{\gamma (t-s)R}, G\Big\rangle \,ds \\
	&= e^{\gamma t}|v| \langle 1_v, G_t1_v\rangle + \gamma \int_0^t e^{\gamma (t-s)}\big\langle  \widetilde{A}_s + R_sR + R^\top R_s , G_{t-s}\big\rangle \,ds.
\end{align*}
Using the definition of $\widetilde{A}_s$ and Proposition \ref{pr:expectations}(iib) we have
\begin{align*}
\langle \widetilde{A}_s , G_{t-s}\rangle &= \E_v\big[(|\X^{R}_s|+1) \langle 1_{\X^{R}_s}, R (G_{t-s})_{\mathrm{diag}}\rangle\big] \\
& \le  2\E_v\big[|\X^{R}_s| \langle 1_{\X_s}, R (G_{t-s})_{\mathrm{diag}}\rangle\big] \\
	&\le 2 |v|\langle 1_v,e^{\gamma s (I+R)}(I+R)R (G_{t-s})_{\mathrm{diag}}\rangle.
\end{align*}
Using the definition of $R$ and Proposition \ref{pr:expectations}(iiia) we have
\begin{align*}
\langle R_sR + R^\top R_s , G_{t-s}\rangle &= \langle R_s, R G_{t-s} + G_{t-s}R^\top \rangle \\ 
	&\le \langle 1_v, (R G_{t-s} + G_{t-s}R^\top)_s 1_v\rangle  + \gamma \int_0^s \langle 1_v, e^{\gamma uR} R ((R G_{t-s} + G_{t-s}R^\top)_{s-u})_{\mathrm{diag}}\rangle  \,du  \\
	&= \langle 1_v, (R G_t + G_tR^\top)  1_v\rangle  + \gamma \int_0^s \langle 1_v, e^{\gamma uR} R (R G_{t-u} + G_{t-u}R^\top)_{\mathrm{diag}}\rangle  \,du ,
\end{align*}
where we used the simple identity $(R G_{t-s})_s  =  R (G_{t-s})_s = R G_t$.
Putting it together,
\begin{align}
\E_v[  |\X^{R}_t|\langle 1_{\X^{R}_t},G1_{\X^{R}_t}\rangle ] &\le e^{\gamma t}|v| \langle 1_v, G_t1_v\rangle + 2\gamma |v| \int_0^t e^{\gamma (t-s)} \langle 1_v,e^{\gamma s (I+R)}(I+R)R (G_{t-s})_{\mathrm{diag}}\rangle   \,ds \nonumber \\
	&\quad + \gamma \int_0^t e^{\gamma (t-s)}  \langle 1_v, (R G_t + G_tR^\top)  1_v\rangle \,ds \label{pf:quadratic22} \\
	&\quad + \gamma^2\int_0^t e^{\gamma (t-s)} \int_0^s \langle 1_v, e^{\gamma uR} R (R G_{t-u} + G_{t-u}R^\top)_{\mathrm{diag}}\rangle  \,du  \,ds. \nonumber
\end{align}
The third term on the right-hand side of \eqref{pf:quadratic22} equals 
\begin{align*}
(e^{\gamma t}-1) \langle 1_v, (R G_t + G_tR^\top)  1_v\rangle ,
\end{align*}
and we will discard the $-1$ term.
The second term on the right-hand side of  \eqref{pf:quadratic22} equals 
\begin{align*}
2\gamma e^{\gamma t}|v| \int_0^t  \langle 1_v,e^{\gamma s  R }(I+R)R (G_{t-s})_{\mathrm{diag}}\rangle   \,ds.
\end{align*}
The fourth term on the right-hand side of  \eqref{pf:quadratic22}, after interchanging $du$ and $ds$, equals
\begin{align*}
 \gamma &\int_0^t (e^{\gamma (t-u)} - 1)  \langle 1_v, e^{\gamma uR} R (R G_{t-u} + G_{t-u}R^\top)_{\mathrm{diag}}\rangle  \,du \\
 &\le \gamma \int_0^t e^{\gamma (t-u)}   \langle 1_v, e^{\gamma uR} (I+R)R (R G_{t-u} + G_{t-u}R^\top)_{\mathrm{diag}}\rangle  \,du,
\end{align*}
where the last step used nonnegativity of the entries of $R$.
These  bounds let us combine the first and third terms in \eqref{pf:quadratic22}, as well as the second and fourth, to get
\begin{align*}
\E_v[  |\X^{R}_t|\langle 1_{\X^{R}_t},G1_{\X^{R}_t}\rangle ] &\le |v| e^{\gamma t} \langle 1_v, (R G_t + G_t R^\top + G_t)1_v\rangle \\
	&\quad +  \gamma |v| e^{\gamma t} \int_0^t  \langle 1_v,e^{\gamma s  R }(I+R)R(R G_{t-s} + G_{t-s}R^\top + 2 G_{t-s})_{\mathrm{diag}}\rangle   \,ds.
\end{align*}
This completes the proof of Proposition \ref{pr:expectations}(iiib), and thus of the entire theorem.

\section{Proofs of the concrete bounds} \label{se:proofs-concrete}

This section is devoted to the proofs of the theorems in Section \ref{se:concrete}. They will all follow the same rough outline:
\begin{itemize}
\item We start from Proposition \ref{pr:CTMC}, or its extension Proposition \ref{pr:CTMC-selfimproved}, which bounds entropies in terms of $\E_v[\CC(\X^{R}_t)]$ or $\E_v[\widehat{\CC}(\X^{R}_t)]$.
\item Depending on which theorem we are proving, we bound $\CC$ or $\widehat{\CC}$ from above by a convenient function $F$ for which we can estimate $\E_v[F(\X^{R}_t)]$ using Proposition \ref{pr:expectations}.
\item We simplify the estimate from Proposition \ref{pr:expectations}.
\end{itemize}
A challenge in simplifying the estimates of Proposition \ref{pr:expectations} is that spectral bounds are not often useful in our context. The benchmark example is the mean field case, $\xi_{ij}=1_{i \neq j}/(n-1)$, which has eigenvalues $1$ and $-1/(n-1)$ with respective multiplicities $1$ and $n-1$. Similarly, in the regular graph case (Definition \ref{def:m-regular-case}) or the random walk case (Definition \ref{def:randomwalk-case}), the matrix $\xi$ always has $1$ as an eigenvalue.  In particular, the operator norm $\|\xi\|_{\mathrm{op}}$ might be bounded but is not small in our  cases of interest, and the averaging effects of a dense matrix $\xi$ must be captured by other means.

\subsection{Controlling $h_3$}
We begin by using the first claims in Proposition \ref{pr:CTMC-selfimproved}(1,2) to prove a lemma that explain how to get a bound on the quantity $h_3$, where we recall $h_3$ was a constant bounding the 3-particle entropies in Lemma \ref{lem: init.ent.est}(ii) and Proposition \ref{pr:CTMC-selfimproved}. This will then let us use the sharper second claims of Proposition \ref{pr:CTMC-selfimproved}(1,2). Recall that $\delta=\max_{ij}\xi_{ij}$.

\begin{lemma} \label{le:selfimprovement-prep}
Suppose there exists $C_0 > 0$ such that $H_0(v) \le C_0\delta^2|v|^3$ for all $v \subset [n]$.
\begin{enumerate}[(i)]
\item If Assumption \ref{asssump:common} holds for $T < \infty$, then
\begin{equation*}
H_{[T]}(v)\lesssim \delta^2, \quad \text{for all } v\subset [n] \text{ with } |v|=3. 
\end{equation*}
where the hidden constant depends only on $(T,C_0,\gamma,M,\sigma^2)$.
\item If Assumption \ref{asssump:uniform.in.time} holds, then 
\begin{equation*}
\sup_{t \ge 0} H_t(v)\lesssim \delta^2, \quad \text{for all } v\subset [n] \text{ with } |v|=3.  
\end{equation*} 
where the hidden constant depends only on $(\eta,C_0,\gamma, M,\sigma^2)$. 
\end{enumerate} 
\end{lemma}
\begin{proof} { \ }
	We fix throughout the proof the choice $R = \xi$, which belongs to  $ 	\mathcal{R}$ and also satisfies \eqref{cond:R row.sum}. This allows us to use Propositions \ref{pr:CTMC-selfimproved} and \ref{pr:expectations}. 
\begin{enumerate}[(i)]
\item 
We begin with the trivial inequality
\begin{equation*}
	\CC(v) = \frac{M}{\sigma^2} \sum_{i \in v} \bigg(\sum_{j \in v} \xi_{ij} \bigg)^2
	\le \frac{M\delta^2}{\sigma^2} |v|^3.
\end{equation*}
Applying Proposition \ref{pr:CTMC-selfimproved}(i) and the assumption $H_0(v) \le C_0\delta^2|v|^3$, we have 
\begin{equation*}
H_{[T]}(v) \le \E_v[H_0(\X_T^{R})] + \int_0^T \E_v[\CC(\X_t^{R})]\,dt  \le C_0\delta^2\E_v|\X_T^{R}|^3 + \frac{M\delta^2}{\sigma^2}\int_0^T \E_v|\X_t^{R}|^3\,dt .
\end{equation*}
Using Proposition \ref{pr:expectations}(ic), we get
\begin{equation*}
	H_{[T]}(v) \le 8e^{6\gamma T/\sigma^2}\Big(C_0+\frac{M}{3\gamma}\Big)\delta^2|v|^3.
\end{equation*} 
\item We proceed exactly as for (i), but using part (ii) of Proposition \ref{pr:CTMC-selfimproved} instead of part (i). This yields, with $r=\sigma^2/4\eta$, 
\begin{align*}
H_T(v) &\le e^{-rT}\E_v[H_0(\X_T^{R})] + \int_0^T e^{-rt}\E_v[\CC(\X_t^{R})]\,dt  \\
	&\le C_0\delta^2e^{-rT}\E_v|\X_T^{R}|^3 + \frac{M\delta^2}{\sigma^2}\int_0^T e^{-rt}\E_v|\X_t^{R}|^3\,dt \\
	&\le 8 C_0 e^{(6\gamma/\sigma^2 - r) T} \delta^2|v|^3 + \frac{ 8 M\delta^2}{\sigma^2} |v|^3 \int_0^T e^{(6\gamma/\sigma^2 - r)\,t} \,dt.
\end{align*}
The claim follows because $r > 6\gamma/\sigma^2$ by Assumption \ref{asssump:uniform.in.time}(iii). \qedhere
\end{enumerate} 
\end{proof}

As much as possible, we will unify the proofs of the estimates on $H_{[T]}(v)$ and on $\sup_{t \ge 0}H_t(v)$, with the understanding that, in the uniform-in-time case, Assumption \ref{asssump:uniform.in.time} should be imposed instead of Assumption \ref{asssump:common}, and all hidden constants must be independent of $T$ as well as $(\eta,C_0,\gamma, M,\sigma^2)$.

Let us record a few immediate consequences of Lemma \ref{le:selfimprovement-prep}.
Recall the definition of  $\widehat{\CC}$ from \eqref{def:CChat},
\begin{equation*}
\widehat{\CC}(v) = \frac{ \sqrt{\gamma M h_3}}{\sigma^2} \sum_{i \in v}\bigg(\sum_{j \in v}\xi_{ij}\bigg)^2 + \frac{M}{\sigma^2} \sum_{i,j \in v} \xi_{ij}^2.
\end{equation*}
Here $h_3$ was a constant bounding the 3-particle entropies in Proposition \ref{pr:CTMC-selfimproved}, which by Lemma \ref{le:selfimprovement-prep} can be taken to be $h_3 \lesssim \delta^2$. Hence, we may write
\begin{equation}
\widehat{\CC}(v) \lesssim \delta \sum_{i \in v}\bigg(\sum_{j \in v}\xi_{ij}\bigg)^2 +   \sum_{i,j \in v} \xi_{ij}^2. \label{pf:CChat-bestbound}
\end{equation}
Here the hidden constant could depend on $T$ if we are using Lemma \ref{le:selfimprovement-prep}(i), but it does not depend on $T$ if Lemma \ref{le:selfimprovement-prep}(ii) is used. As a consequence of Lemma \ref{le:selfimprovement-prep}, we may apply Proposition \ref{pr:CTMC-selfimproved} to get the following two bounds, which along with \eqref{pf:CChat-bestbound} will be the starting points for the remaining proofs:
\begin{itemize}
\item If Assumption \ref{asssump:common} holds for $T < \infty$, then
\begin{equation}
	H_{[T]}(v) \le\inf_{R \in \mathcal{R}}\E_v\bigg[H_0(\X^{R}_T) + \int_0^T \widehat{\CC}(\X^{R}_t)\, dt\bigg]. \label{pf:mainCChat}
\end{equation}
\item If Assumption \ref{asssump:uniform.in.time} holds, then, with $r=\sigma^2/4\eta$,
	\begin{equation} 
	H_t(v)  \le \inf_{R \in \mathcal{R}} \E_v \left[ e^{-\sigma^2t/4\eta} H_0(\X^{R}_t) + \int_0^t e^{-\sigma^2s/4\eta } \widehat{\CC}(\X^{R}_s) \, ds\right]. \label{pf:mainCChat.uit}
\end{equation} 
\end{itemize}

\begin{remark} \label{re:reversedentropy2}
In the case of reversed entropy discussed in Section \ref{se:reversedentropy}, if we apply the Remark \ref{re:reversedentropy1}, we find that $\overleftarrow{H}_{[T]}(v)$ obeys the same inequality \eqref{pf:mainCChat} except with $\widehat{C}(\cdot)$ sharpened to $\widetilde{C}(v) = (M/\sigma^2)\sum_{i,j\in v}\xi_{ij}^2$. In fact, there is no need for an estimate of the three-particle entropies, and this is the sense in which the case of reversed entropy is easier. By the Cauchy-Schwarz inequality, we have $\widehat{C}(v) \lesssim (\delta|v|+1)\widetilde{C}(v)$, which explains the claim made in Section \ref{se:reversedentropy} that the reversed entropy bounds save a factor of $(\delta|v|+1)$ compared to the theorems of Section \ref{se:concrete}.
\end{remark}

\subsection{Maximum entropy: Proof of Theorem \ref{theorem:max}}

We now prove Theorem \ref{theorem:max}, first proving \eqref{eq:max-mainbound}. To do this, we again make the choice $R = \xi$, which  allows us to use \eqref{pf:mainCChat}--\eqref{pf:mainCChat.uit} and Proposition  \ref{pr:expectations}. We use a simple upper bound for \eqref{pf:CChat-bestbound}:
\begin{align} \label{eq:C.in.max.case}
\widehat{\CC}(v) \lesssim \delta^3|v|^3 + \delta^2|v|^2.
\end{align}
Combining this with \eqref{pf:mainCChat}, and the assumption $H_0(v) \lesssim \delta^2|v|^2+\delta^3|v|^3$, we get
\begin{equation*}
H_{[T]}(v) \lesssim \delta^3\E_v|\X_T^{R}|^3 + \delta^2\E_v|\X_T^{R}|^2 + \int_0^T \big(\delta^3\E_v|\X_t^{R}|^3 + \delta^2\E_v|\X_t^{R}|^2\big)\,dt.
\end{equation*}
By Proposition \ref{pr:expectations}(i), ignoring the time-dependent constants, we have $\E_v|\X_t^{R}|^p \lesssim |v|^p$ for $p=2,3$. This yields
\begin{equation*}
H_{[T]}(v) \lesssim \delta^3 |v|^3 + \delta^2|v|^2,
\end{equation*}
exactly as claimed in  \eqref{eq:max-mainbound}.

To prove the claimed uniform-in-time estimates of Theorem \ref{theorem:max}, we must be more careful and take into account the time-dependence of the estimates of $\E_v|\X_t^{R}|^p$. Using \eqref{pf:mainCChat.uit}, we have
\begin{equation*}
H_{T}(v) \lesssim e^{-rT}\delta^3\E_v|\X_T^{R}|^3 + e^{-rT}\delta^2\E_v|\X_T^{R}|^2 +  \int_0^T e^{-rt}\big(\delta^3\E_v|\X_t^{R}|^3 + \delta^2\E_v|\X_t^{R}|^2\big)\,dt.
\end{equation*}
Using Proposition \ref{pr:expectations}(i),
\begin{equation*}
H_{T}(v)  \lesssim \delta^3   e^{(6\gamma/\sigma^2-r)T}|v|^3   + \delta^2 e^{(6\gamma/\sigma^2-r)T}|v|^2 +   \int_0^T \big(\delta^3 |v|^3 e^{(6\gamma/\sigma^2-r)t} + \delta^2|v|^2e^{(6\gamma/\sigma^2-r)t} \big)\,dt.
\end{equation*}
This is again $\lesssim \delta^2|v|^2 + \delta^3 |v|^3$ as long as $6\gamma/\sigma^2 < r$, which is true by Assumption \ref{asssump:uniform.in.time}(iii).  \hfill \qedsymbol

\subsection{Average entropy: Proof of Theorem  \ref{theorem:avg-markov}}

Here we prove Theorem \ref{theorem:avg-markov}. We again fix $R = \xi$. Note that the assumption \eqref{asmp:initial-markov} on the initial conditions clearly implies $H_0(v) \lesssim(\delta |v|)^3 + (\delta |v|)^2 \le 2\delta^2|v|^3$ for all $ v \subset [n]$, since $\delta \le 1$. This allows us to apply Lemma \ref{le:selfimprovement-prep} and its consequences outlined at the beginning of the section.
Recall the notation $\delta_i=\max_j \xi_{ij}$ for the row-maximum. We begin by bounding \eqref{pf:CChat-bestbound} by
\begin{align*}
	\widehat{\CC}(v) \lesssim \delta|v|^2 \sum_{i \in v}\delta_i^2 + |v|\sum_{i \in v}\delta_i^2 = (\delta|v|^2+|v|)\langle 1_v,x\rangle,
\end{align*}
where $x=(\delta^2_1,\ldots,\delta^2_n)$. Then, using \eqref{pf:mainCChat}, we have
\begin{align*}
	H_{[T]}(v) &\le \E_v[H_0(\X_T^{R})] + \int_0^T \E_v[\widehat{\CC}(\X_t^{R})]\,dt  \\
	&\lesssim \E_v\Big[(\delta |\X_T^{R}|^3 + |\X_T^{R}|^2) \Big]  \sum_{i=1}^n \pi_i \delta_i^2  + \int_0^T \E_v[(\delta|\X_t^{R}|^2+|\X_t^{R}|)\langle 1_{\X_t^{R}},x\rangle]\,dt,
\end{align*}
where we used also the assumption \eqref{asmp:initial-markov} on the initial conditions.
We control the first term using (ib) and (ic) of Proposition \ref{pr:expectations}, and we control the second term using parts (iib) and (iic):
\begin{align}
	\begin{split}
		H_{[T]}(v)
		&\lesssim \big(\delta e^{6\gamma T/\sigma^2} |v|^3 + e^{6\gamma T/\sigma^2} |v|^2\big)  \sum_{i=1}^n \pi_i \delta_i^2  \\
		&\qquad + \int_0^T \bigg(\delta |v|^2 \big\langle 1_v, e^{2\gamma t (2I+\xi)/\sigma^2}(I+\xi)^2 x\big\rangle + |v| \big\langle 1_v,  e^{2\gamma t (I+\xi)/\sigma^2}(I+\xi) x \big\rangle \bigg)\,dt.
	\end{split} \label{pf:avg1way-11}
\end{align}
Now for any $ k \in [n]$ and $v\subset[n]$ such that $|v| \le k$, 	\eqref{pf:avg1way-11} implies
\begin{align}
	\begin{split}
		H_{[T]}(v)
		&\lesssim \big(\delta e^{6\gamma T/\sigma^2} k^3 + e^{6\gamma T/\sigma^2} k^2\big) \sum_{i=1}^n \pi_i \delta_i^2 \\
		&\qquad + \int_0^T \bigg(\delta k^2 \big\langle 1_v, e^{2\gamma t (2I+\xi)/\sigma^2}(I+\xi)^2 x\big\rangle + k\big\langle 1_v,  e^{2\gamma t (I+\xi)/\sigma^2}(I+\xi) x \big\rangle \bigg)\,dt.
	\end{split} \label{pf:avg1way-11-markov}
\end{align}
To incorporate the random set $\V$, note for any vector $y \in \R^n$ that
	\begin{align} \label{eq:avg-1way-markov-2}
		\E[\langle 1_{\V}, y \rangle ] = \E\Big[\sum_{i=1}^n y_i 1_{i\in\V}\Big] \le k \sum_{i=1}^n y_i \pi_i
		=k\langle\pi, y \rangle
	\end{align}
	by the assumption that $\PP(i\in\V)\le k\pi_i$.
	We apply \eqref{eq:avg-1way-markov-2} in \eqref{pf:avg1way-11-markov} to get
	\begin{align}\label{pf:avg-1way-precol-markov}
		\begin{split}
			\E[H_{[T]}(\V)]
			&\lesssim \big(\delta e^{6\gamma T/\sigma^2} k^3 + e^{6\gamma T/\sigma^2} k^2\big)  \sum_{i=1}^n \pi_i \delta_i^2\\
			&\qquad +  \int_0^T \bigg( \delta k^3 \big\langle \pi , e^{ 2\gamma  t (2I+\xi)/\sigma^2}(I+\xi)^2 x\big\rangle + k^2\big\langle \pi,  e^{2\gamma  t (I+\xi)/\sigma^2}(I+\xi) x \big\rangle \bigg)\,dt.
		\end{split}
	\end{align}
	To bound the two time integrands above, note that the assumption $\pi^\top \xi \le \pi^\top$ implies $\pi^\top \xi^m \le \pi^\top$ for every $m \in \N$, hence 
		\begin{align*}
			\pi^\top e^{2\gamma t \xi/\sigma^2}
			&= \sum_{m=0}^{\infty} \frac{(2\gamma t/\sigma^2)^m}{m!}\,\pi^\top \xi^m
			\le \sum_{m=0}^{\infty} \frac{(2\gamma t/\sigma^2)^m}{m!}\,\pi^\top
			= e^{2\gamma t/\sigma^2}\,\pi^\top.
		\end{align*}
		Thus for any nonnegative vector $y \in \R^n$, 
		\begin{align} \label{pf:inner-product-pi-1}
			\langle \pi, e^{2\gamma t \xi/\sigma^2} y \rangle
			= \pi^\top e^{2\gamma t \xi/\sigma^2} y
			\le e^{2\gamma t/\sigma^2}\,\pi^\top y
			= e^{2\gamma t/\sigma^2}\,\langle \pi, y \rangle.
		\end{align}
		Also, since $x=(\delta^2_1,\ldots,\delta^2_n)$ is nonnegative, $\langle \pi , \xi x\rangle  =  \pi^\top \xi x  \le \langle \pi, x \rangle$. Hence
		\begin{align}  \label{pf:inner-product-pi-2}
			\langle \pi, (I + \xi) x \rangle \le 2 \langle \pi, x \rangle,
			\qquad
			\langle \pi, (I+\xi)^2 x\rangle \le 4\langle \pi, x\rangle.
		\end{align} 
		Using \eqref{pf:inner-product-pi-1}--\eqref{pf:inner-product-pi-2}, we have
		\begin{align*}
			\big\langle \pi,  e^{2\gamma  t (I+\xi)/\sigma^2}(I+\xi) x \big\rangle
			&= e^{2\gamma t/\sigma^2}\big\langle \pi,  e^{2\gamma  t \xi/\sigma^2}(I+\xi) x \big\rangle \\
			&\le e^{2\gamma t/\sigma^2}\,e^{2\gamma t/\sigma^2}\big\langle \pi,  (I+\xi) x \big\rangle
			\le 2  e^{6\gamma t/\sigma^2} \langle \pi, x \rangle.
		\end{align*}
		Similarly,
		\begin{align*}
			\big\langle \pi,  e^{2\gamma  t (2I+\xi)/\sigma^2}(I+\xi)^2 x \big\rangle
			&= e^{6\gamma t/\sigma^2}\big\langle \pi,  e^{2\gamma  t \xi/\sigma^2}(I+\xi)^2 x \big\rangle \\
			&\le e^{6\gamma t/\sigma^2}\,e^{2\gamma t/\sigma^2}\big\langle \pi,  (I+\xi)^2 x \big\rangle
			\le 6\,e^{6\gamma t/\sigma^2}\langle \pi, x\rangle.
		\end{align*}
		Therefore, 
	\begin{align} \label{pf:inner-product-pi-3}
		\delta k^3 \big\langle \pi , e^{ 2\gamma  t (2I+\xi)/\sigma^2}(I+\xi)^2 x\big\rangle
		&\le 6\delta e^{6\gamma t/\sigma^2}k^3 \langle \pi, x\rangle
		= 6\delta e^{6\gamma t/\sigma^2} k^3  \sum_{i=1}^n \pi_i\delta_i^2.
	\end{align}
	A similar argument shows that
	\begin{align} \label{pf:inner-product-pi-4}
		k^2 \big\langle \pi,  e^{2\gamma  t (I+\xi)/\sigma^2}(I+\xi) x \big\rangle
		&\le 2e^{6\gamma t/\sigma^2}k^2 \langle \pi, x\rangle
		= 2e^{6\gamma t/\sigma^2}k^2\sum_{i=1}^n \pi_i \delta_i^2.
	\end{align}
	Plugging \eqref{pf:inner-product-pi-3}--\eqref{pf:inner-product-pi-4} into \eqref{pf:avg-1way-precol-markov} to get
	\begin{equation} 
		\E[H_{[T]}(\V)]
		\lesssim  e^{6\gamma T/\sigma^2} (\delta k + 1) k^2   \sum_{i=1}^n \pi_i \delta_i^2
		+ (\delta k + 1)k^2 \bigg( \sum_{i=1}^n \pi_i \delta_i^2\bigg) \int_0^T e^{6\gamma t/\sigma^2} \,dt.
		\label{pf:avg1way-3}
	\end{equation}
	This completes the proof of the first claim \eqref{eq:markov-mainbound} of Theorem \ref{theorem:avg-markov}.
	
To prove the uniform-in-time claim, we make minor modifications: Use \eqref{pf:mainCChat.uit} in place of \eqref{pf:mainCChat} to get
\begin{equation*}
	H_T(v)  \lesssim e^{-rT} \E_v\Big[(\delta |\X_T^{R}|^3 + |\X_T^{R}|^2) \Big] \sum_{i=1}^n \pi_i\delta_i^2 + \int_0^T e^{-rt}\E_v[(\delta|\X_t^{R}|^2+|\X_t^{R}|)\langle 1_{\X_t^{R}},x\rangle]\,dt,
\end{equation*}
for $r=\sigma^2/4\eta$. Repeat the argument to bound $\E[H_T(\V)]$ for each $T>0$ by the same right-hand side as \eqref{pf:avg1way-3}, except with the two exponentials replaced by $e^{(6\gamma/\sigma^2 -r)T}$ and $e^{(6\gamma/\sigma^2 -r)t}$, respectively. Because $r > 6\gamma/\sigma^2$ by Assumption \ref{asssump:uniform.in.time}(iii), the claim follows.
\hfill \qedsymbol

\subsection{Sharper average entropy: Proof of Theorem \ref{theorem:avg.2way.asym}} \label{se:proof:sharper}

Here we prove Theorem \ref{theorem:avg.2way.asym}, starting with the bound on $H_{[T]}(v)$ claimed in \eqref{eq:avg.2way.asym}.  In contrast to the proofs of Theorem \ref{theorem:max} and Theorem \ref{theorem:avg-markov}, we make a different choice of the matrix $R$ given as follows. For each $i,j = 1, \dots, n$, define 
\begin{align*}
	S_i := \sum_{\ell=1}^n (\xi_{i\ell}^2 + \xi_{\ell i}^2), \qquad R_{ij}
	=
	\begin{cases}
		\dfrac{S_i \xi_{ij}}{S_i + \xi_{ij}}, & \text{if } \xi_{ij} > 0, \\
		0, & \text{if } \xi_{ij} = 0 .
	\end{cases}
\end{align*}
We claim  that $R \in \mathcal{R}$. Indeed, if $S_i=0$, then
$\xi_{ij}=0$ for all $j$ and $\sum_{j=1}^n \xi_{ij}^2/R_{ij}=0$ by the convention $0/0=0$ in our definition of $\mathcal{R}$. Otherwise $S_i>0$, and
\begin{align*}
	\sum_{j=1}^n \frac{\xi_{ij}^2}{R_{ij}}
	&=
	\frac{1}{S_i}\sum_{j:\,\xi_{ij}>0}\xi_{ij}(S_i+\xi_{ij})
	\le
	\frac{1}{S_i}\Big(S_i\sum_{j=1}^n\xi_{ij}+\sum_{j=1}^n\xi_{ij}^2\Big)
	\le 2,
\end{align*}
where we used \eqref{cond:row.sum} and $\sum_{j=1}^n\xi_{ij}^2\le S_i$.

Also, let us define
\begin{align}
p_\xi :=	\sum_{i=1}^n S_i^2   = \sum_{i=1}^n\bigg(\sum_{j=1}^n (\xi_{ij}^2 +\xi_{ji}^2)\bigg)^2.
\end{align}
Using the assumption that the row and column sums of $\xi$ are bounded by 1, it is easy to see from the definition  that $p_\xi \le 6\delta^2n$. Using $\delta \le 1$ and the assumption \eqref{eq: init.chao.assum} on the initial condition,
\begin{align*}
H_0(v) &\lesssim \frac{\delta |v|^3+|v|^2}{n^2}\sum_{i,j=1}^n\xi_{ij}^2 + \frac{\delta |v|^2 + |v|}{n}p_\xi \\
		&\le \frac{\delta |v|^3+|v|^2}{n^2} \delta^2 n^2  + \frac{\delta |v|^2 + |v|}{n} 6\delta^2 n \lesssim \delta^2|v|^3 .
\end{align*}
This lets us apply Lemma \ref{le:selfimprovement-prep} along with its consequences described at the beginning of the section. Applying \eqref{pf:CChat-bestbound} and bounding the first term therein using convexity, we have
\begin{align*}
\widehat{\CC}(v) \lesssim \delta |v| \sum_{i,j \in v} \xi_{ij}^2 +   \sum_{i,j \in v} \xi_{ij}^2 = (\delta |v|+1) \langle 1_v, \widehat{\xi} 1_v\rangle,
\end{align*}
where we recall that $\widehat{\xi}_{ij}=\xi_{ij}^2$ is the entrywise (Hadamard) square of $\xi$.
Now, by Lemma \ref{le:selfimprovement-prep}, we may apply \eqref{pf:mainCChat}. Using the fact that $R \in \mathcal{R}$,  along with the assumed bound on $H_0(v)$, we get
\begin{equation}
\begin{split}
H_{[T]}(v) & \lesssim \E_v\Big[ \big(\delta |\X_T^{R}|^3+|\X_T^{R}|^2\big)\frac{1}{n^2}\sum_{i,j=1}^n\xi_{ij}^2 + \big(\delta |\X_T^{R}|^2 + |\X_T^{R}|\big)\frac{p_\xi}{n}\Big] \\
	&\quad + \int_0^T \E_v\big[(\delta |\X_t^{R}|+1) \langle 1_{\X_t^{R}}, \widehat{\xi} 1_{\X_t^{R}}\rangle  \big]\,dt.
\end{split} \label{pf:avg-main1}
\end{equation}
We next apply Proposition \ref{pr:expectations} to estimate each term. This is justified because $R \le \xi$ entrywise, and so the assumption \eqref{cond:row.sum} implies \eqref{cond:R row.sum}. The first expectation is straightforward to bound using Proposition \ref{pr:expectations}(i):
\begin{align*}
\E_v\Big[\big(\delta |\X_T^{R}|^3+|\X_T^{R}|^2\big)\frac{1}{n^2}\sum_{i,j=1}^n\xi_{ij}^2 + \big(\delta |\X_T^{R}|^2 + |\X_T^{R}|\big)\frac{p_\xi}{n}\Big] \lesssim e^{3\gamma T}\bigg(\frac{\delta|v|^3 + |v|^2}{n^2}\sum_{i,j=1}^n\xi_{ij}^2 + \frac{\delta|v|^2+|v|}{n}p_\xi \bigg).
\end{align*}
In particular, for any $F: 2^{[n]} \to \R$, writing
	\begin{align} \label{eq:average.notation}
		\avg_{|v|=k}=\frac{1}{{n \choose k}} \sum_{v\subset [n] : |v|=k}
	\end{align}
	 to denote the average over all choices of $v \subset [n]$ of cardinality $k$, we have
\begin{equation}
\avg_{|v|=k}\E_v\Big[\big(\delta |\X_T^{R}|^3+|\X_T^{R}|^2\big)\frac{1}{n^2}\sum_{i,j=1}^n\xi_{ij}^2 + \big(\delta |\X_T^{R}|^2 + |\X_T^{R}|\big)\frac{p_\xi}{n}\Big] \lesssim e^{3\gamma T}(\delta k+1)\bigg(\frac{k^2}{n^2}\sum_{i,j=1}^n\xi_{ij}^2 + \frac{k}{n}p_\xi\bigg). \label{pf:avg-main2}
\end{equation}
We next bound the second expectation in \eqref{pf:avg-main1}, by applying Proposition \ref{pr:expectations}(iii). We do this in two steps.

{ \ }

\noindent\textbf{Step 1.}
We first show that
\begin{equation}
\avg_{|v|=k}\E_v\big[  \langle 1_{\X_t^{R}}, \widehat{\xi} 1_{\X_t^{R}}\rangle  \big] \lesssim e^{2\gamma t}\bigg(\frac{k^2}{n^2} \sum_{i,j=1}^n  \xi_{ij}^2 + \frac{k }{n } p_\xi \bigg). \label{pf:avg-step1}
\end{equation}
We apply Proposition \ref{pr:expectations}(iiia) with $G=\widehat{\xi}$, recalling that $\widehat{\xi}_{ij}=\xi_{ij}^2$ is the entrywise square. We get
\begin{equation}
\E_v\big[  \langle 1_{\X_t^{R}}, \widehat{\xi} 1_{\X_t^{R}}\rangle  \big] \le \langle 1_v,\widehat{\xi}_t 1_v\rangle + \gamma \int_0^t  \Big\langle 1_v, R e^{\gamma (t-u)R} (\widehat{\xi}_u)_{\mathrm{diag}}\Big\rangle \,du, \label{pf:avg-step1-1}
\end{equation}
where $\widehat{\xi}_u = e^{\gamma u R}\widehat{\xi}e^{\gamma u R^\top}$. We now average over all choices of $v \subset [n]$ of size $k$. 
The principle is that for any vector $y \in \R^n$, we have
	\begin{equation}
	\avg_{|v|=k} \langle 1_v, y\rangle =\avg_{|v|=k} \sum_{i=1}^n y_i 1_{i \in v} = \sum_{i=1}^n y_i \, \avg_{|v|=k}  1_{i \in v} = \frac{k}{n}\sum_{i=1}^n y_i = \frac{k}{n}\langle 1, y\rangle , \label{eq:avg-1way}
	\end{equation}
	where $1$ is the all-ones vector.
	Indeed, the identity $\avg_{|v|=k}  1_{i \in v}=k/n$ is simply saying that the probability of a fixed $i \in [n]$ belonging to a uniformly random set $v \subset [n]$ of size $k$ is $k/n$.
Applying \eqref{eq:avg-1way}, the average of the second term in \eqref{pf:avg-step1-1} becomes 
\begin{equation*}
\avg_{|v|=k}\Big\langle 1_v, R e^{\gamma (t-u)R} (\widehat{\xi}_u)_{\mathrm{diag}}\Big\rangle = \frac{k}{n}\Big\langle 1 , R e^{\gamma (t-u)R} (\widehat{\xi}_u)_{\mathrm{diag}}\Big\rangle.
\end{equation*}
Recalling that $1$ denotes the all-ones vector, the column sum bound $\xi^\top 1 \le 1$ together with the entrywise inequality $R \le \xi$ imply that the column sum bound  $R^\top 1 \le 1$. Therefore,
\begin{equation}
\avg_{|v|=k}\Big\langle 1_v, R e^{\gamma (t-u)R} (\widehat{\xi}_u)_{\mathrm{diag}} \Big\rangle  \le \frac{k}{n}e^{\gamma (t-u) }\Big\langle 1 ,  (\widehat{\xi}_u)_{\mathrm{diag}}\Big\rangle = \frac{k}{n}e^{\gamma (t-u) }\tr(\widehat{\xi}_u).\label{pf:avg-step1-2}
\end{equation}
For the first term in \eqref{pf:avg-step1-1}, we use the identity 
\begin{equation}\label{pf:avg-step1-2.1}
\avg_{|v|=k} 1_{i,j \in v} = \frac{k(k-1)}{n(n-1)} 1_{i\neq j} + \frac{k}{n}1_{i=j} = \frac{k(k-1)}{n(n-1)} + \frac{k(n-k)}{n(n-1)}1_{i = j},
\end{equation}
valid for $i,j \in [n]$. Indeed, this simply says that the probability of both $i$ and $j$ belonging to a uniformly random set $v \subset [n]$ of size $k$ is $k(k-1)/n(n-1)$  if $i \neq j$, or $k/n$  if $i=j$. As a consequence, for any $n \times n$ matrix $G$,
\begin{equation}
\avg_{|v|=k} \langle 1_v, G1_v\rangle =\avg_{|v|=k} \sum_{i,j=1}^n G_{ij} 1_{i,j \in v}  =  \frac{k(k-1)}{n(n-1)}\sum_{i,j=1}^n G_{ij} + \frac{k(n-k)}{n(n-1)}\tr(G). \label{eq:avg-2way}
\end{equation}
Apply this to the first term in \eqref{pf:avg-step1-1} and simplify using the bounds $(k-1)/(n-1)\le k/n$ and $(n-k)/(n-1) \le 1$:
\begin{equation}
\avg_{|v|=k}  \langle 1_v,\widehat{\xi}_t 1_v\rangle \le \frac{k^2}{n^2}\sum_{i,j=1}^n (\widehat{\xi}_t)_{ij} + \frac{k }{n }\tr(\widehat{\xi}_t).\label{pf:avg-step1-3}
\end{equation}
The first term on the right-hand side can be controlled using the column sum bound $R^\top 1 \le 1$:
\begin{align*}
\sum_{i,j=1}^n (\widehat{\xi}_t)_{ij} &=\big\langle 1, e^{\gamma t R}\widehat{\xi}e^{\gamma t R^\top} 1\big\rangle \le e^{2\gamma t} \langle 1, \widehat{\xi}1\rangle =e^{2\gamma t} \sum_{i,j=1}^n  \xi_{ij}^2.
\end{align*}
Plug this into \eqref{pf:avg-step1-3}, and then plug the result along with \eqref{pf:avg-step1-2} into \eqref{pf:avg-step1-1}, to get
\begin{equation}
\avg_{|v|=k}\E_v\big[  \langle 1_{\X_t^{R}}, \widehat{\xi} 1_{\X_t^{R}}\rangle  \big] \lesssim  e^{2\gamma t} \frac{k^2}{n^2} \sum_{i,j=1}^n  \xi_{ij}^2 + \frac{k }{n }\tr(\widehat{\xi}_t) + \frac{k}{n} \gamma \int_0^t e^{\gamma(t-u)}\tr(\widehat{\xi}_u)\,du . \label{pf:avg-step1-4}
\end{equation}
To complete the proof of \eqref{pf:avg-step1}, it suffices to show that
\begin{equation}
\tr(\widehat{\xi}_t) \le 2e^{2\gamma t}p_\xi, \qquad t \ge 0. \label{pf:avg-trxihat<pxi}
\end{equation}
To this end, we make use of a Taylor series formula. For an $n \times n$ matrix $G$, we have
\begin{equation}
G_t = e^{\gamma t R}G e^{\gamma t R^\top} = \sum_{m=0}^\infty \frac{(\gamma t)^m}{m!} \Gamma_m(G), \qquad \Gamma_m(G) := \sum_{\ell=0}^m {m \choose \ell} R^\ell G(R^\top)^{m-\ell}, \label{pf:avg-taylorseries}
\end{equation}
which is easily derived using a Cauchy product calculation:
\begin{align}
e^{  t R} G e^{ tR^\top} &= \bigg(\sum_{r=0}^\infty  \frac{t^r}{r!} R^r\bigg)\bigg(\sum_{r=0}^\infty  \frac{t^r}{r!} G(R^\top)^r\bigg) = \sum_{m=0}^\infty \sum_{r=0}^m  \frac{t^m}{r!(m-r)!}R^r G (R^\top)^{m-r} \nonumber \\
	&= \sum_{m=0}^\infty \frac{t^m}{m!} \BB_m(G), \qquad \text{for } t \in \R . \label{pf:Gamma-CauchyProduct}
\end{align}
The diagonal entries of $\xi$, and hence those of $R$, are zero, which implies that $\tr (\Gamma_0(\widehat{\xi}))=\tr (\widehat{\xi})=0$. Hence, 
\begin{equation}
\tr(\widehat{\xi}_t) = \gamma t \tr \big(\Gamma_1(\widehat{\xi})\big) + \sum_{m=2}^\infty \frac{(\gamma t)^m}{m!} \tr \big(\Gamma_m(\widehat{\xi})\big). \label{pf:avg-trxihat-taylor}
\end{equation}
The $m=1$ term is estimated as
\begin{align}
\tr \big(\Gamma_1(\widehat{\xi})\big) &= \tr \big(\widehat{\xi} R^\top +  R \widehat{\xi} \big) = \sum_{i,j=1}^{n}  \big( \xi_{ij}^2 R_{ij}  + \xi_{ji}^2 R_{ij}\big) = \sum_{i=1}^n S_i \sum_{j=1}^n  \frac{\xi_{ij}^3 + \xi_{ij}\xi_{ji}^2}{S_i + \xi_{ij}} \nonumber \\
& \le  \sum_{i=1}^n S_i \sum_{j=1}^n \big(\xi^2_{ij} + \xi_{ji}^2  \big) =  p_\xi,
 \label{pf:avg-trxihat-taylor1}
\end{align}
where we used $S_i \ge 0$ in the second-to-last step.
The $m \ge 2$ terms are estimated as follows. Write
\begin{equation*}
\tr\big(R^\ell \widehat{\xi} (R^\top)^{m-\ell}\big)  = \sum_{i,j=1}^n \xi_{ij}^2 \big( (R^\top)^{m-\ell} R^\ell\big)_{ji}.
\end{equation*}
Let $e_1,\ldots,e_n$ denote the standard basis in $\R^n$.
Note that since $R$ has row and column sums bounded by 1 it must also have  $\|R\|_{\mathrm{op}} \le 1$. 
For $m > \ell > 0$, it follows that
\begin{equation*}
\big( (R^\top)^{m-\ell} R^\ell\big)_{ji}  \le |R^\ell e_i| |R^{m-\ell}e_j| \le 
|R e_i| |R e_j| \le \frac12(|R e_i|^2 + |R e_j|^2).
\end{equation*}
If $\ell = m$ we have
\begin{equation*}
\big( (R^\top)^{m-\ell} R^\ell\big)_{ji} = \langle R^\top e_j, R^{m-2}R  e_i\rangle \le  |R e_i||R^\top e_j| \le \frac12(|R e_i|^2 + |R^\top e_j|^2).
\end{equation*}
If $\ell =0$ we have
\begin{equation*}
\big( (R^\top)^{m-\ell} R^\ell\big)_{ji} = \langle  R e_j,  (R^\top)^{m-2} R^\top e_i\rangle \le  |R e_j||R^\top e_i| \le \frac12(|R e_j|^2 + |R^\top e_i|^2).
\end{equation*}
Hence, for $m \ge 2$, we split off and then recombine the $\ell \in \{0,m\}$ cases to get
\begin{align}
\tr\big(\Gamma_m(\widehat{\xi})\big) &= \sum_{\ell=0}^m {m \choose \ell} \tr\big(R^\ell \widehat{\xi} (R^\top)^{m-\ell}\big) \nonumber \\
	&\le \frac12 \sum_{i,j=1}^n \xi_{ij}^2(|R e_i|^2 + |R e_j|^2 + |R^\top e_i|^2 + |R^\top e_j|^2) + \frac12\sum_{\ell=1}^{m-1} {m \choose \ell} \sum_{i,j=1}^n \xi_{ij}^2(|R e_i|^2 + |R e_j|^2) \nonumber \\
	&\le 2^m\sum_{i,j=1}^n \xi_{ij}^2(|R e_i|^2 + |R e_j|^2 + |R^\top e_i|^2 + |R^\top e_j|^2) \nonumber \\
	&= 2^{m }\sum_{i,j,r=1}^n \xi_{ij}^2 ( R_{ri}^2 + R_{rj}^2 +  R_{ir}^2 + R_{jr}^2) \nonumber \\
	&\le 2^{m }\sum_{i,j,r=1}^n \xi_{ij}^2 ( \xi_{ri}^2 + \xi_{rj}^2 +  \xi_{ir}^2 + \xi_{jr}^2)   \\
	&= 2^{m } p_\xi, \label{pf:avg-trxihat-taylor2}
\end{align}
where we used the entrywise inequality $R \le \xi$ in the second-to-last step.
Plug this and \eqref{pf:avg-trxihat-taylor1} into \eqref{pf:avg-trxihat-taylor} to get
\begin{equation*}
\tr(\widehat{\xi}_t) \le \gamma  t  p_\xi+ \sum_{m=2}^\infty \frac{(2\gamma t)^m}{m!} p_\xi \le 2 e^{2\gamma t} p_\xi,  
\end{equation*}
completing the proof of \eqref{pf:avg-trxihat<pxi} and thus of Step 1.

{ \ }

\noindent\textbf{Step 2.}
We next show that
\begin{equation}
\avg_{|v|=k}\E_v\big[ |\X_t^{R}| \langle 1_{\X_t^{R}}, \widehat{\xi} 1_{\X_t^{R}}\rangle  \big] \lesssim e^{3\gamma t}\bigg(\frac{k^3}{n^2} \sum_{i,j=1}^n  \xi_{ij}^2 + \frac{k^2 }{n } p_\xi \bigg). \label{pf:avg-step2}
\end{equation}
Applying Proposition \ref{pr:expectations}(iiib) with $G=\widehat{\xi}$, we have 
\begin{equation}
\begin{split}
\E_v\big[ |\X_t^{R}| \langle 1_{\X_t^{R}}, \widehat{\xi} 1_{\X_t^{R}} \rangle  \big] &\le |v| e^{\gamma t} \Big\langle 1_v,\Big(  R \widehat{\xi}_t +  \widehat{\xi}_tR^\top + \widehat{\xi}_t\Big) 1_v\Big\rangle \\
&\quad + \gamma |v| e^{\gamma t} \int_0^t  \langle 1_v,e^{\gamma (t-s)  R }(I+R)R(R \widehat{\xi}_s + \widehat{\xi}_sR^\top + 2 \widehat{\xi}_s)_{\mathrm{diag}}\rangle   \,ds.
\end{split} \label{pf:avg-step2-1}
\end{equation}
We now average over all choices of $v \subset [n]$ of size $k$. Starting with the second term of \eqref{pf:avg-step2-1}, we use the identity \eqref{eq:avg-1way} along with the coordinatewise inequality $1^\top R \le 1^\top$ (due to the column sums being bounded by 1) to get
\begin{align}
\avg_{|v|=k} &\gamma |v| e^{\gamma t} \int_0^t  \langle 1_v,e^{\gamma (t-s)  R }(I+R)R(R \widehat{\xi}_s + \widehat{\xi}_sR^\top + 2\widehat{\xi}_s)_{\mathrm{diag}}\rangle   \,ds \nonumber \\
	&= \frac{k^2}{n} \gamma e^{\gamma t} \int_0^t  \langle 1 ,e^{\gamma (t-s)  R }(I+R)R(R \widehat{\xi}_s + \widehat{\xi}_sR^\top + 2\widehat{\xi}_s)_{\mathrm{diag}}\rangle   \,ds \nonumber \\
	&\le \frac{k^2}{n} 2\gamma e^{\gamma t} \int_0^t   e^{\gamma (t-s)  } \langle 1, (R \widehat{\xi}_s + \widehat{\xi}_sR^\top +2 \widehat{\xi}_s)_{\mathrm{diag}}\rangle   \,ds \nonumber \\
	&= \frac{k^2}{n} 2\gamma e^{\gamma t} \int_0^t   e^{\gamma (t-s)  } \tr(R \widehat{\xi}_s + \widehat{\xi}_sR^\top + 2\widehat{\xi}_s)  \,ds \label{pf:avg-step2-1.5}
\end{align}
Turning to the first term in  \eqref{pf:avg-step2-1}, 
we start by using \eqref{eq:avg-2way} to get
\begin{align*}
\avg_{|v|=k}|v| e^{\gamma t} \Big\langle 1_v,\Big(  R \widehat{\xi}_t +  \widehat{\xi}_tR^\top + 2\widehat{\xi}_t\Big) 1_v\Big\rangle &\le e^{\gamma t}\frac{k^3}{n^2}\sum_{i,j=1}^n\Big(  R \widehat{\xi}_t +  \widehat{\xi}_tR^\top + 2\widehat{\xi}_t\Big)_{ij} \\
	&\qquad + e^{\gamma t}\frac{k^2}{n}\tr\Big(  R \widehat{\xi}_t +  \widehat{\xi}_tR^\top + 2\widehat{\xi}_t\Big).
\end{align*}
Using the row and column sum bounds, $R 1 \le 1$ and $R^\top 1 \le 1$,
\begin{align*}
e^{\gamma t}\frac{k^3}{n^2}\sum_{i,j=1}^n\Big(  R \widehat{\xi}_t +  \widehat{\xi}_tR ^\top +2 \widehat{\xi}_t\Big)_{ij} &= e^{\gamma t}\frac{k^3}{n^2}\Big\langle 1,\big(  R \widehat{\xi}_t +  \widehat{\xi}_tR^\top + 2\widehat{\xi}_t\big)1\Big\rangle \\
	&\le 4e^{\gamma t}\frac{k^3}{n^2} \big\langle 1,\widehat{\xi}_t1\big\rangle  = 4\frac{k^3}{n^2}  e^{\gamma t}\big\langle 1,e^{\gamma t R}\widehat{\xi} e^{\gamma t R^\top }1\big\rangle  \\
	&\le 4e^{3\gamma t}\frac{k^3}{n^2 }\langle 1, \widehat{\xi} 1\rangle  = 4 e^{3\gamma t}\frac{k^3}{n^2}  \sum_{i,j=1}^n\xi_{ij}^2.
\end{align*}
Plugging this and \eqref{pf:avg-step2-1.5} into \eqref{pf:avg-step2-1}, we find
\begin{align*}
\avg_{|v|=k}\E_v\big[ |\X_t^{R}| \langle 1_{\X_t^{R}}, \widehat{\xi} 1_{\X_t^{R}}\rangle  \big] &\lesssim e^{3\gamma t}\frac{k^3}{n^2} \sum_{i,j=1}^n\xi_{ij}^2 +  e^{\gamma t}\frac{k^2}{n}\tr\Big(  R \widehat{\xi}_t +  \widehat{\xi}_tR^\top + 2 \widehat{\xi}_t\Big) \\
	&\quad + \frac{k^2}{n}  \gamma e^{\gamma t} \int_0^t   e^{\gamma (t-s)  } \tr(R \widehat{\xi}_s + \widehat{\xi}_s R^\top + 2 \widehat{\xi}_s)  \,ds.
\end{align*}
Recalling from \eqref{pf:avg-trxihat<pxi} that $\tr(\widehat{\xi}_s) \le 2e^{2\gamma s}p_\xi$, 
the proof of \eqref{pf:avg-step2} will be complete once we show that
\begin{equation}
\tr\Big(  R \widehat{\xi}_t +  \widehat{\xi}_tR^\top  \Big) \le 2 e^{2\gamma t}p_\xi, \qquad t \ge 0. \label{pf:avg-step2-2}
\end{equation} 
To do so, we will again make use of the Taylor series \eqref{pf:avg-taylorseries}, by writing
\begin{align*}
\tr(R \widehat{\xi}_t + \widehat{\xi}_tR^\top) &=   \sum_{m=0}^\infty \frac{(\gamma t)^m}{m!} \tr \big(R\Gamma_m(\widehat{\xi}) + \Gamma_m(\widehat{\xi})R^\top\big) = \sum_{m=0}^\infty \frac{(\gamma t)^m}{m!} \tr \big(  \Gamma_{m+1}(\widehat{\xi})\big).
\end{align*}
The last step used the identity
\begin{equation*}  
	R \BB_m(G) + \BB_m(G)R^\top = \sum_{r=0}^m {m \choose r}\Big[ R^{r+1}G(R^\top)^{m-r} + R^{r}G(R^\top)^{m-r+1}\Big] = \BB_{m+1}(G),
\end{equation*}
which follows from a more general fact that for any sequence $\{a_n\}$,
\begin{equation*} 
	\sum_{r=0}^m \binom{m}{r} (a_{r+1} + a_r )
	 = \sum_{r=0}^{m+1}\binom{m+1}{r} a_r.
\end{equation*}
 Recalling the estimates \eqref{pf:avg-trxihat-taylor1} and \eqref{pf:avg-trxihat-taylor2}, we have
\begin{align*}
\tr( R \widehat{\xi}_t + \widehat{\xi}_t R^\top)  &\le  p_\xi + \sum_{m=1}^\infty \frac{(\gamma t)^m}{m!}2^{m +1}   p_\xi \le 2 e^{2\gamma t}p_\xi.
\end{align*}
This proves \eqref{pf:avg-step2-2}, thus completing the proof of Step 2.

{ \ }

With Steps 1 and 2 established, we now put them together with \eqref{pf:avg-main2} to yield a bound for \eqref{pf:avg-main1}. Specifically, adding \eqref{pf:avg-main2} plus \eqref{pf:avg-step1} plus $\delta$ times \eqref{pf:avg-step2}, we deduce from \eqref{pf:avg-main1} that
\begin{align*}
  \overline{H}^k_{[T]}  &\lesssim (\delta k+1)\bigg(\frac{k^2}{n^2}\sum_{i,j=1}^n\xi_{ij}^2 + \frac{k}{n}p_\xi\bigg).
\end{align*}
This proves the claim \eqref{eq:avg.2way.asym} of Theorem \ref{theorem:avg.2way.asym}. To prove the uniform-in-time claim, we make minor modifications: Use \eqref{pf:mainCChat.uit} in place of \eqref{pf:mainCChat} to get the following alternative to \eqref{pf:avg-main1}:
\begin{equation*}
\begin{split}
H_T(v) & \lesssim e^{-rT}\E_v\big[ \delta |\X_T^{R}|^3+|\X_T^{R}|^2\big]\frac{1}{n^2}\sum_{i,j=1}^n\xi_{ij}^2 + e^{-rT}\E_v\big[\delta |\X_T^{R}|^2 + |\X_T^{R}|\big]\frac{p_\xi}{n} \\
	&\quad + \int_0^T e^{-rt}\E_v\big[(\delta |\X_t^{R}|+1) \langle 1_{\X_t^{R}}, \widehat{\xi} 1_{\X_t^{R}}\rangle  \big]\,dt,
\end{split}  
\end{equation*}
where $r=\sigma^2/4\eta$.
In the estimates above,  the largest exponential term was $e^{6\gamma t/\sigma^2}$. Hence, because $r > 6\gamma/\sigma^2$ in assumption \ref{asssump:uniform.in.time}(iii),  we end up with the same bound for $\sup_{T > 0}H_T(v)$. \hfill \qedsymbol

\subsection{Setwise entropy: Proof of Theorem \ref{theorem:pointwise}}

In this section we prove Theorem \ref{theorem:pointwise}. We again fix $R = \xi$.  Recall that our assumption therein on the initial condition is that $H_0(v) \le C_0 q_\xi(v)$ for all $v \subset [n]$, where $q_\xi(v)$ can be written as
\begin{equation}
	q_\xi(v) = (\delta|v|+1)\langle 1_v, \widehat{\xi} 1_v\rangle + \delta (\delta|v|+1)\langle 1_v, (\xi^\top \xi + \xi\xi^\top) 1_v\rangle + \delta^3|v|^2  +\delta^2|v| , \label{def:qxi-rewrite}
\end{equation}
where $\widehat{\xi}_{ij}=\xi_{ij}^2$ is the entrywise square of $\xi$.

We first claim that
\begin{align} \label{eq:q.upper.bound.by.delta}
	q_{\xi}(v)\le 8\delta^2|v|^3, \quad v \subset [n],
\end{align}
which will thus allow us to apply Lemma \ref{le:selfimprovement-prep} and its consequences outlined at the beginning of the section.
The row and column sum bounds imply $\|\xi\|_{\mathrm{op}} \le 1$, which yields
\begin{align*}
	\langle 1_v, \xi\xi^\top 1_v\rangle = |\xi^\top 1_v|^2 \le |1_v|^2 = |v|^2.
\end{align*}
The same bound holds for $\langle 1_v, \xi^\top\xi 1_v\rangle$.  Using also $\sum_{i,j\in v}\xi^2_{ij}\le \delta^2|v|^2$,  we deduce 
\begin{align*}
	q_{\xi}(v) &\le (\delta|v|+1)\delta^2|v|^2 + 2\delta(\delta|v| +1) |v|^2  + \delta^3|v|^2  +\delta^2|v|  \le 8\delta^2|v|^3,
\end{align*}
where the last step just used  $\delta \le 1$. This establishes \eqref{eq:q.upper.bound.by.delta}.

Bounding the first term in \eqref{pf:CChat-bestbound} using convexity, we have
\begin{align*}
	\widehat{\CC}(v) \lesssim \delta |v| \sum_{i,j \in v} \xi_{ij}^2 +   \sum_{i,j \in v} \xi_{ij}^2 = (\delta |v|+1) \langle 1_v, \widehat{\xi} 1_v\rangle .
\end{align*}
Now, by Lemma \ref{le:selfimprovement-prep}, we may apply \eqref{pf:mainCChat} and \eqref{def:qxi-rewrite} to get
\begin{align} 
	H_{[T]}(v) & \lesssim \E_v\Big[(\delta|\X_T^{R}| +1)\langle 1_{\X_T^{R}}, \widehat{\xi} 1_{\X_T^{R}}\rangle + \delta (\delta|\X_T^{R}|+1)\langle 1_{\X_T^{R}}, (\xi^\top \xi + \xi\xi^\top) 1_{\X_T^{R}}\rangle  + \delta^3 |\X_T^{R}|^2  + \delta^2 |\X_T^{R}|\Big] \nonumber  \\
	&\quad + \int_0^T \E_v\big[(\delta |\X_t^{R}|+1) \langle 1_{\X_t^{R}}, \widehat{\xi} 1_{\X_t^{R}}\rangle  \big]\,dt.   \label{pf:setwise-main1}
\end{align}
We next apply Proposition \ref{pr:expectations} to estimate each term.
This will be done in five steps.

{ \ }

\noindent\textbf{Step 1.}
Using Proposition \ref{pr:expectations}(ia, ib), we have
\begin{equation}
	\E_v |\X_T^{R}|  \le e^{2\gamma T/\sigma^2}  |v| ,\qquad \E_v |\X_T^{R}|^2  \le 2e^{6\gamma T/\sigma^2}  |v|^2   .\label{pf:setwise-exp1}
\end{equation} 

{ \ }

\noindent\textbf{Step 2.}
We next show that
\begin{equation}
	\E_v\big[  \langle 1_{\X_t^{R}}, \widehat{\xi} 1_{\X_t^{R}}\rangle  \big] \lesssim e^{6\gamma t/\sigma^2}\Big(\langle 1_v,\widehat{\xi}  1_v\rangle + \delta \big\langle 1_v,\big(RR^\top + R^\top R\big) 1_v \big\rangle + \delta^2|v|\Big). \label{pf:setwise-exp2}
\end{equation}
Start by applying Proposition \ref{pr:expectations}(iiia) with $G=\widehat{\xi}$ to get
\begin{equation}
	\E_v\big[  \langle 1_{\X_t^{R}}, \widehat{\xi} 1_{\X_t^{R}}\rangle  \big] \le \langle 1_v,\widehat{\xi}_t 1_v\rangle + \frac{\gamma}{\sigma^2} \int_0^t  \Big\langle 1_v, R e^{2\gamma (t-u)R/\sigma^2} (\widehat{\xi}_u)_{\mathrm{diag}}\Big\rangle \,du, \label{pf:setwise-**}
\end{equation}
where $\widehat{\xi}_t = e^{2\gamma t R/\sigma^2}\widehat{\xi}e^{2\gamma t R^\top/\sigma^2}$. 
To estimate this, we write
\begin{align*}
	\widehat{\xi}_t &= \widehat{\xi} + (e^{2\gamma t R/\sigma^2}-I)\widehat{\xi} +  e^{2\gamma t R/\sigma^2} \widehat{\xi}(e^{2\gamma t R^\top/\sigma^2}-I).
\end{align*}
Using the Cauchy-Schwarz inequality, $\|R\|_{\mathrm{op}} \le 1$, and the coordinatewise inequality $\widehat{\xi} \le \delta R$, 
\begin{align*}
	\langle 1_v,(e^{2\gamma t R/\sigma^2}-I)\widehat{\xi} 1_v\rangle
	&= \frac{2\gamma}{\sigma^2} \int_0^{ t} \langle 1_v, R e^{2\gamma u R/\sigma^2} \widehat{\xi} 1_v\rangle\,du \\
	&\le |R^\top 1_v|\,| \widehat{\xi} 1_v|\, \frac{2\gamma}{\sigma^2} \int_0^t  e^{2\gamma u/\sigma^2}\,du
	\le \delta |R^\top 1_v|\,| R 1_v|\, e^{2\gamma t/\sigma^2} \\
	&\le \frac12 \delta e^{2\gamma t/\sigma^2}\big(|R^\top 1_v|^2+  |  R  1_v|^2\big).
\end{align*}
Similarly,
\begin{align*}
	\langle 1_v,e^{2\gamma t R/\sigma^2} \widehat{\xi}(e^{2\gamma t R^\top/\sigma^2}-I) 1_v\rangle
	&= \frac{2\gamma}{\sigma^2} \int_0^{ t} \langle 1_v,e^{2\gamma t R/\sigma^2} \widehat{\xi} e^{2\gamma u R^\top/\sigma^2} R^\top1_v\rangle\,du \\
	&\le \delta\, \frac{2\gamma}{\sigma^2} \int_0^{ t} \langle 1_v,e^{2\gamma t R/\sigma^2}  R  e^{2\gamma u R^\top/\sigma^2} R^\top1_v\rangle\,du \\
	&\le  \delta\, \frac{2\gamma}{\sigma^2}\, |R^\top 1_v|^2\int_0^t  e^{2\gamma (t+u)/\sigma^2}\,du  
	\le   \delta |R^\top 1_v|^2 e^{6\gamma t/\sigma^2} . 
\end{align*}
Combining the above three displays, 
\begin{align}
	\langle 1_v,\widehat{\xi}_t 1_v\rangle
	&\lesssim
	\langle 1_v,\widehat{\xi}  1_v\rangle
	+
	\delta e^{6\gamma t/\sigma^2}\,
	\big\langle 1_v,\big(RR^\top + R^\top R\big) 1_v \big\rangle. \label{pf:setwise-hatxi1}
\end{align}
Finally, letting $1$ denote the vector of all ones, use the coordinatewise inequalities $\widehat{\xi} \le \delta^2 1 1^\top$ and $R 1\le 1$ to get
\begin{align*}
	(\widehat{\xi}_u)_{\mathrm{diag}}
	&= \big( e^{2\gamma u R/\sigma^2} \widehat{\xi} e^{2\gamma uR^\top/\sigma^2} \big)_{\mathrm{diag}}
	\le \delta^2 \big( e^{2\gamma u R/\sigma^2} 11^\top e^{2\gamma uR^\top/\sigma^2} \big)_{\mathrm{diag}}
	\le \delta^2 e^{6\gamma u/\sigma^2} 1. 
\end{align*}
Hence,
\begin{align*}
	\big\langle 1_v, R e^{2\gamma (t-u)R/\sigma^2} (\widehat{\xi}_u)_{\mathrm{diag}}\big\rangle
	&\le  \delta^2 e^{6\gamma u/\sigma^2} \big\langle 1_v, R e^{2\gamma (t-u)R/\sigma^2} 1\big\rangle
	\le \delta^2 e^{6\gamma t/\sigma^2} |v|.
\end{align*}
Plug this and \eqref{pf:setwise-hatxi1} into \eqref{pf:setwise-**} to deduce \eqref{pf:setwise-exp2}.

{ \ }

\noindent\textbf{Step 3.}
We next show that
\begin{equation}
	\E_v\big[|\X_t^{R}|  \langle 1_{\X_t^{R}}, \widehat{\xi} 1_{\X_t^{R}}\rangle  \big] \lesssim  |v| e^{6\gamma t/\sigma^2} \Big(\langle 1_v,\widehat{\xi}  1_v\rangle + \delta \big\langle 1_v,\big(RR^\top + R^\top R\big) 1_v \big\rangle + |v|\delta^2\Big). \label{pf:setwise-exp3}
\end{equation}
Start by applying Proposition \ref{pr:expectations}(iiib) with $G=\widehat{\xi}$ to get 
\begin{equation}
	\begin{split}
		\E_v\big[ |\X_t^{R}| \langle 1_{\X_t^{R}}, \widehat{\xi} 1_{\X_t^{R}}\rangle  \big]
		&\le |v| e^{2\gamma t/\sigma^2} \Big\langle 1_v,\Big(  R \widehat{\xi}_t +  \widehat{\xi}_tR^\top + \widehat{\xi}_t\Big) 1_v\Big\rangle   \\
		&\quad + \frac{2\gamma}{\sigma^2} |v| e^{2\gamma t/\sigma^2} \int_0^t  \langle 1_v,e^{2\gamma (t-s)  R /\sigma^2}(I+R)R(R \widehat{\xi}_s + \widehat{\xi}_sR^\top + 2\widehat{\xi}_s)_{\mathrm{diag}}\rangle   \,ds .
	\end{split} \label{pf:setwise-exp3-1}
\end{equation}
Using $(\widehat{\xi}_s)_{\mathrm{diag}} \le \delta^2 e^{6\gamma s/\sigma^2}1$ (and similarly for $(R\widehat{\xi}_s)_{\mathrm{diag}}$, $(\widehat{\xi}_sR^\top)_{\mathrm{diag}}$), the integral term is $\lesssim e^{6\gamma t/\sigma^2}|v|^2\delta^2$.
For the first term, use \eqref{pf:setwise-hatxi1} and $\widehat{\xi} \le \delta R$ to conclude it is $\lesssim e^{6\gamma t/\sigma^2}|v|\big(\langle 1_v,\widehat{\xi}1_v\rangle+\delta\langle 1_v,(R R^\top+R^\top R)1_v\rangle\big)$.
This yields \eqref{pf:setwise-exp3}.

{ \ } \noindent\textbf{Step 4.} Similarly to Step 2, we will next show that \begin{equation} \E_v\big[ \langle 1_{\X_t^{R}}, (\xi^\top \xi + \xi \xi^\top) 1_{\X_t^{R}}\rangle \big] \lesssim e^{6\gamma t/\sigma^2}\Big(\langle 1_v,(\xi\xi^\top + \xi^\top \xi) 1_v \rangle + \delta |v|\Big). \label{pf:setwise-exp4} \end{equation} Start by applying Proposition \ref{pr:expectations}(iiia) with $G=\xi^\top \xi + \xi \xi^\top$ to get \begin{equation} \E_v\big[ \langle 1_{\X_t^{R}}, (\xi^\top \xi + \xi \xi^\top) 1_{\X_t^{R}}\rangle \big] \le \langle 1_v,G_t 1_v\rangle + \frac{\gamma}{\sigma^2} \int_0^t \Big\langle 1_v, R e^{2\gamma (t-u)R/\sigma^2} (G_u)_{\mathrm{diag}}\Big\rangle \,du, \label{pf:setwise-step4-1} \end{equation} where $G_t = e^{2\gamma t R/\sigma^2}Ge^{2\gamma t R^\top/\sigma^2}$. Using $\|R\|_{\mathrm{op}}\le 1$ and bounding as in Step 2 gives \begin{align} \langle 1_v,G_t 1_v\rangle \lesssim e^{6\gamma t/\sigma^2}\langle 1_v,(\xi^\top \xi + \xi \xi^\top)1_v\rangle, \label{pf:setwise-step4-2} \end{align} and using $(G_u)_{\mathrm{diag}}\lesssim \delta e^{6\gamma u/\sigma^2}1$ yields the $\delta|v|$ term, proving \eqref{pf:setwise-exp4}. 

{ \ }

\noindent\textbf{Step 5.} Similarly to Step 3, we will next show that \begin{equation} \E_v\big[|\X_t^{R}| \langle 1_{\X_t^{R}}, (\xi^\top \xi + \xi \xi^\top) 1_{\X_t^{R}}\rangle \big] \lesssim |v| e^{6\gamma t/\sigma^2}\Big( \langle 1_v,(\xi\xi^\top + \xi^\top \xi) 1_v \rangle + \delta |v|\Big). \label{pf:setwise-exp5} \end{equation} This follows by applying Proposition \ref{pr:expectations}(iiib) with $G=\xi^\top\xi+\xi\xi^\top$ (so the front factor is $e^{2\gamma t/\sigma^2}$ and $G_t$ carries another $e^{6\gamma t/\sigma^2}$), exactly as in Step 3. 

{ \ }

\noindent\textbf{Step 6.} In this step we put together Steps 1--5 to produce a bound for \eqref{pf:setwise-main1}. Indeed, note that the bound \eqref{pf:setwise-exp3} from Step 3 is $|v|$ times the bound \eqref{pf:setwise-exp2} from Step 2, and similarly the bound \eqref{pf:setwise-exp5} from Step 5 is $|v|$ times the bound \eqref{pf:setwise-exp4} from Step 4. Keeping track of the factors of $\delta$ in \eqref{pf:setwise-main1}, we get \begin{align*} H_{[T]}(v) &\lesssim \delta^2 |v| +\delta^3 |v|^2 + (\delta|v|+1)\Big(\langle 1_v,\widehat{\xi} 1_v\rangle + \delta \langle 1_v,(\xi\xi^\top + \xi^\top \xi) 1_v \rangle + \delta^2|v|\Big) \\ &\quad + \delta(\delta|v|+1) \Big( \langle 1_v,(\xi\xi^\top + \xi^\top \xi) 1_v \rangle + \delta|v| \Big). \end{align*} Combining terms, the right-hand side is $\lesssim q_\xi(v)$, and the proof of the first claim of Theorem \ref{theorem:pointwise} is complete.

{ \ }

\noindent\textbf{Step 7.} Next, we explain the uniform-in-time part of Theorem \ref{theorem:pointwise}. This requires only some minor adaptations of the above arguments, most importantly keeping track of exponents. Using \eqref{pf:mainCChat.uit} instead of \eqref{pf:mainCChat}, we get the following analogue of \eqref{pf:setwise-main1}, with $r= \sigma^2/4\eta$: \begin{equation} \begin{split} H_{T}(v) & \lesssim e^{-rT}\E_v\Big[(\delta|\X_T^{R}| +1)\langle 1_{\X_T^{R}}, \widehat{\xi} 1_{\X_T^{R}}\rangle + \delta (\delta|\X_T^{R}|+1)\langle 1_{\X_T^{R}}, (\xi^\top \xi + \xi\xi^\top) 1_{\X_T^{R}}\rangle + \delta^2 |\X_T^{R}|\Big] \\ &\quad + \int_0^T e^{-rt}\E_v\big[(\delta |\X_t^{R}|+1) \langle 1_{\X_t^{R}}, \widehat{\xi} 1_{\X_t^{R}}\rangle \big]\,dt. \end{split} \label{pf:setwise-main1-unif} \end{equation} This is the same as the right-hand side of \eqref{pf:setwise-main1} aside from the exponential terms. Checking through Steps 2--5 above, the largest exponential factor was $e^{6\gamma t/\sigma^2}$, and thus the resulting bound on \eqref{pf:setwise-main1-unif} is uniform in $T > 0$ because $r > 6\gamma/\sigma^2$ by Assumption \ref{asssump:uniform.in.time}(iii). \hfill \qedsymbol
\section{Proofs for Gaussian example} \label{se:gaussianproofs}

In this section we prove Theorem \ref{theorem:Gaussian.avg.entropy}, Proposition \ref{pr:Gaussiancase-clique}, and Proposition \ref{pr:Gaussiancase-max}.
Let us write $\lambda_{\mathrm{max}}(A)$ and $\lambda_{\mathrm{min}}(A)$ for the largest and smallest eigenvalues of a symmetric matrix $A$.
We start with bounds for the relative entropy between two centered Gaussian measures, which essentially performs a leading-order (quadratic) Taylor expansion of the entropy in terms of the covariance matrices. In the following, we write $\alpha_-=\max(-\alpha,0)$ for the negative part of a number $\alpha$.

\begin{proposition} \label{prop.static.second.moment.lower.bound}
	Consider two  centered nondegenerate Gaussian measures $\gamma_0$ and $\gamma_1$ on $\R^k$ with covariance matrices $\Sigma_0$ and $\Sigma_1$.
	\begin{enumerate}[(i)]
		\item If $-1<\alpha\le\lambda_{\min}(\Sigma_0^{-1}\Sigma_1 - I)$, we have
		\begin{align*}
			H\left(\gamma_1 \, |\,  \gamma_0 \right) \le  \Big(\frac12 + \frac{\alpha_{-}}{3(1+ \alpha)^3}\Big)\tr((\Sigma_0^{-1}\Sigma_1 - I)^2).
		\end{align*}
		\item If $\lambda_{\min}(\Sigma_0^{-1}\Sigma_1 - I) > -1$ and $\lambda_{\max}(\Sigma_0^{-1}\Sigma_1 - I) \le 1$,
		\begin{align*}
			H\left(\gamma_1 \, |\,  \gamma_0 \right) \ge  \frac{1}{6}\tr((\Sigma_0^{-1}\Sigma_1 - I)^2).
		\end{align*}
	\end{enumerate}
\end{proposition}
\begin{proof}[Proof of Proposition \ref{prop.static.second.moment.lower.bound}.]
We make   use of the following basic fact: If  $f,g : [-\rho,\rho] \to \R$ are continuous functions satisfying $f\le g$ pointwise, then 
\begin{equation}
	\tr( f(A)) \le \tr (g(A)) \label{ineq:tr-f-g}
\end{equation}
for any symmetric matrix $A$ with eigenvalues contained in $[-\rho,\rho]$. 
	We start from the following well known explicit formula:
	\begin{align*}
		H\left(\gamma_1 \vert \gamma_0 \right) &= \frac{1}{2}\left[ \tr\left(\Sigma_0^{-1}\Sigma_1\right) - k + \log\frac{\det(\Sigma_0) }{\det(\Sigma_1 )} \right]\\
		&= \frac12\Big(\tr(\Sigma_0^{-1}\Sigma_1- I) - \log\det(\Sigma_0^{-1}\Sigma_1)\Big) \\
		&= \frac12\tr\,h(\Sigma_0^{-1}\Sigma_1 - I),
	\end{align*}
	where we used $\log\det = \tr \log$, and the scalar function $h$ is defined by $h(x) :=x - \log(1+x)$.
	Note that $h(0)=h'(0)=0$. 
	With a bit of calculus, we have the following upper and lower bounds on $h$. For $-1 < \alpha \le x$, we have
	\begin{align}
		h(x) \le x^2\Big(\frac12 + \frac{\alpha_{-}}{3(1+ \alpha)^3}\Big). \label{h-upperbound}
	\end{align}
	Using the fact that the fourth derivative of $h$ is positive, we have for $1 \ge x > -1$ that
	\begin{align}
		h(x) \ge \frac12 x^2 - \frac13 x^3 = x^2\bigg(\frac12 - \frac13 x\bigg) \ge \frac16 x^2. \label{h-lowerbound}
	\end{align}
	Combine these inequalities with \eqref{ineq:tr-f-g} completes the proof.
\end{proof}

Recall that the laws $P_T$ and $Q_T$ of the SDE systems \eqref{def:SDEgaussian} and \eqref{def:SDEgaussian-indproj} are the centered Gaussian measures on $\R^n$ with covariance matrices $\Sigma_T$ and $TI$, respectively, where $\Sigma_T := \int_0^T e^{t \xi}e^{t \xi^\top} dt$. 
Recall the notation $\rho=\|\xi\|_{\mathrm{op}}$.
The identity 
\begin{equation}
\frac{1}{T}\Sigma_T  - I = \frac{1}{T}\int_0^T(e^{t \xi}e^{t \xi^\top}- I) \,dt \label{pf:Gaussian-Sigma-I}
\end{equation}
implies that
\begin{equation}
\lambda_{\mathrm{min}}(T^{-1}\Sigma_T-I) \ge e^{-2\rho T} - 1 , \qquad \lambda_{\mathrm{max}}(T^{-1}\Sigma_T-I) \le  e^{2\rho T}-1 . \label{pf:Sigma-evals0}
\end{equation}
Indeed, the second inequality is clear. For the first, observe that $e^{t \xi}e^{t \xi^\top}$ is (symmetric) positive semi-definite, and it is invertible with inverse given by $e^{-t \xi}e^{-t \xi^\top}$. Use concavity of $\lambda_{\mathrm{min}}(\cdot)$ to get
\begin{align*}
\lambda_{\mathrm{min}}(T^{-1}\Sigma_T-I) \ge \frac{1}{T}\int_0^T(\lambda_{\mathrm{min}}(e^{t \xi}e^{t \xi^\top})- 1) \,dt.
\end{align*}
For any  positive definite matrix $A$, it is well known that $\Vert A^{-1} \Vert_{\mathrm{op}}=1/\lambda_{\min}(A)$. 	
Applying this to $A=e^{t \xi}e^{t \xi^\top}$, we have $\lambda_{\min}(e^{t \xi}e^{t \xi^\top})=1/\Vert e^{-t \xi}e^{-t \xi^\top}\Vert_{\mathrm{op}}$, and the first claim of \eqref{pf:Sigma-evals0} follows by noting that $\| e^{-t \xi}e^{-t \xi^\top}\|_{\mathrm{op}} \le e^{2\rho t}$.

Marginalizing, the law $P^v_t$ is the centered Gaussian with covariance matrix denoted $\Sigma_t^v$; in general, for an $n \times n$ matrix $A$, we write $A^v$ for the $|v| \times |v|$ principal submatrix of $A$ corresponding to the indices in $v$.
Using \eqref{pf:Sigma-evals0} and Cauchy's interlacing theorem, we have
\begin{align}
\lambda_{\mathrm{min}}(T^{-1}\Sigma^v_T-I) \ge  e^{-2\rho T}-1, \qquad \lambda_{\mathrm{max}}(T^{-1}\Sigma^v_T-I) \le e^{2\rho T}-1, \label{pf:Sigma-evals}
\end{align}
for any $v \subset [n]$.
Note that  $\lambda_{\mathrm{max}}(T^{-1}\Sigma^v_T-I) \le 1$ when $T \le \log(2)/2\rho$.
Hence, applying Proposition \ref{prop.static.second.moment.lower.bound} with $\alpha =  e^{-2\rho T}-1$, and using $\frac12 + \frac{\alpha_-}{3(1+\alpha)^3}=  \frac{1}{2} + \frac13 e^{6 \rho T}(1-e^{-2\rho T}) \le e^{6\rho T} $, we have
\begin{align}
H(P^v_T\,|\,Q^v_T) &\le e^{6\rho T}  \tr \Big(\Big(\frac{1}{T}\Sigma^v_T - I\Big)^2\Big), \quad\quad \forall T \ge 0, \label{pf:Guassian-UB} \\
H(P^v_T\,|\,Q^v_T) &\ge \frac16\tr \Big(\Big(\frac{1}{T}\Sigma^v_T - I\Big)^2\Big), \quad\quad\quad  \ \forall T \le \log(2)/2\rho. \label{pf:Guassian-LB}
\end{align} 
These inequalities will be the starting point for the proofs below. We will also make use of a Taylor expansion, used also within the proof of Theorem \ref{theorem:avg.2way.asym} (see \eqref{pf:Gamma-CauchyProduct}):

\begin{lemma} \label{le:Gaussianseries}
We have
\begin{equation*}
\frac{1}{T}\Sigma_T - I = \sum_{m=1}^{\infty} \frac{T^m}{(m+1)!} \BB_m, \qquad \text{where} \quad \BB_m := \sum_{r=0}^m {m \choose r} \xi^r (\xi^\top)^{m-r}, \ m \in \N.
\end{equation*}
\end{lemma}
\begin{proof}
We have the Cauchy product identity
\begin{align*}
e^{  t\xi} e^{ t\xi^\top} &= \bigg(\sum_{r=0}^\infty  \frac{t^r}{r!} \xi^r\bigg)\bigg(\sum_{r=0}^\infty  \frac{t^r}{r!}  (\xi^\top)^r\bigg) =  \sum_{m=0}^\infty \frac{t^m}{m!} \BB_m , 
\end{align*}
for $t \in \R$. Thus, using $\Gamma_0=I$ and Fubini,
\begin{align*}
\frac{1}{T}\Sigma_T - I &= \frac{1}{T}\int_0^T(e^{  t\xi} e^{ t\xi^\top} - I)\,dt = \frac{1}{T}\int_0^T\sum_{m=1}^\infty \frac{t^m}{m!} \BB_m \, dt = \sum_{m=1}^{\infty} \frac{T^m}{(m+1)!} \BB_m. \qedhere
\end{align*}
\end{proof}

\subsection{Proof of Proposition \ref{pr:Gaussiancase-clique}}
Starting from \eqref{pf:Guassian-LB} and applying Lemma \ref{le:Gaussianseries},
\begin{align*}
H(P^v_T\,|\,Q^v_T) &\ge \frac16\tr \bigg(\bigg(\sum_{m=1}^{\infty} \frac{T^m}{(m+1)!} \BB_m^v\bigg)^2\bigg) \ge \frac{T^2}{24}\tr\big((\Gamma_1^v)^2\big),
\end{align*}
where the second inequality follows from the fact that all entries of $\Gamma_m$ are nonnegative. Using $\Gamma_1=\xi + \xi^\top$, 
\begin{equation*}
\tr\big((\Gamma_1^v)^2\big) = \sum_{i,j \in v}(\xi_{ij} + \xi_{ji})^2 \ge 2\sum_{i,j \in v} \xi_{ij}^2,
\end{equation*} 
where we again used $\xi_{ij} \ge 0$. \hfill\qedsymbol

\subsection{Proof of Theorem \ref{theorem:Gaussian.avg.entropy}}

We start from a general calculation for any symmetric $n \times n$ matrix $A$, where we recall the notation $\avg_{|v|=k}$ defined in \eqref{eq:average.notation}.
As was noted in \eqref{pf:avg-step1-2.1}, for any indices $i,j \in [n]$ we have
\begin{align*}
\avg_{|v|=k} 1_{i,j \in v} = \frac{k(k-1)}{n(n-1)}1_{i\neq j} + \frac{k}{n}1_{i=j} = \frac{k(k-1)}{n(n-1)} + \frac{k(n-k)}{n(n-1)} 1_{i=j}.
\end{align*}
This implies
\begin{align}
\avg_{|v|=k} \tr((A^v)^2) &= \avg_{|v|=k} \sum_{i,j \in v} A_{ij}^2 = \sum_{i,j=1}^n A_{ij}^2  \big(\avg_{|v|=k} 1_{i,j \in v} \big) \nonumber \\
	&= \frac{k(k-1)}{n(n-1)}\tr(A^2) + \frac{k(n-k)}{n(n-1)}\sum_{i=1}^n A_{ii}^2 .  \label{avgtrace} 
\end{align}
Using \eqref{pf:Guassian-UB} and \eqref{pf:Guassian-LB}, we deduce that
\begin{align}
\avg_{|v|=k} H(P^v_T\,|\,Q^v_T) &\le e^{6\rho T}  \bigg(\frac{k(k-1)}{n(n-1)}\tr((T^{-1}\Sigma_T - I)^2) +  \frac{k(n-k)}{n(n-1)}\sum_{i=1}^n (T^{-1}(\Sigma_T)_{ii}-1)^2\bigg), \label{avgbound-UB} \\
\avg_{|v|=k} H(P^v_T\,|\,Q^v_T) &\ge \frac16 \bigg(\frac{k(k-1)}{n(n-1)}\tr((T^{-1}\Sigma_T - I)^2) +  \frac{k(n-k)}{n(n-1)}\sum_{i=1}^n (T^{-1}(\Sigma_T)_{ii}-1)^2\bigg). \label{avgbound-LB}
\end{align} 
It remains to express the right-hand sides in terms of $\xi$.

We start with the upper bound for the trace term. Let $(e_1,\ldots,e_n)$ denote the standard basis in $\R^n$. Using Lemma \ref{le:Gaussianseries},
\begin{align}
\tr((T^{-1}\Sigma_T - I)^2) &= \sum_{i=1}^n | (T^{-1}\Sigma_T - I) e_i|^2 = \sum_{i=1}^n \bigg|\sum_{m=1}^{\infty} \frac{T^m}{(m+1)!} \BB_m e_i\bigg|^2 \label{pf:gaussian-tr-ub1} \\
	&\le \sum_{i=1}^n\bigg(\sum_{m=1}^{\infty} \frac{T^m}{(m+1)!}  |\BB_m e_i|\bigg)^2, \nonumber
\end{align}
To bound $|\Gamma_me_i|^2$, we note first that for $0 \le r \le m$,
\begin{align*}
|\xi^r(\xi^\top)^{m-r}e_i| \le \rho^{m-1}\big(|\xi^\top e_i|1_{r < m} + |\xi e_i|1_{r  =  m}\big).
\end{align*} 
Discarding the indicators, we find for $m \ge 1$ that
\begin{align*}
|\BB_m e_i| &\le \sum_{r=0}^m {m \choose r}\rho^{ m-1 }\big(|\xi^\top e_i|  + |\xi e_i| \big) \le 2^m\rho^{m-1}\big(|\xi^\top e_i|  + |\xi e_i| \big).
\end{align*}
Thus,
\begin{align}
\tr((T^{-1}\Sigma_T - I)^2)  &\le \sum_{i=1}^n\bigg(\sum_{m=1}^{\infty} \frac{T^m}{(m+1)!}  2^m\rho^{m-1}\big(|\xi^\top e_i|  + |\xi e_i| \big)\bigg)^2 \nonumber \\
	&\le 4T^2 e^{4\rho T}\sum_{i=1}^n \big(|\xi^\top e_i|   + |\xi e_i| \big)^2 \nonumber \\
	&\le  16 T^2 e^{4\rho T}\sum_{i,j=1}^n \xi_{ij}^2 . \label{pf:gaussian-ubfinal1}
\end{align}
The lower bound for the trace term is similar: Using nonnegativity of the entries of $\BB_m$ and $\xi$, from  \eqref{pf:gaussian-tr-ub1} we deduce
\begin{align}
\tr((T^{-1}\Sigma_T - I)^2)  &\ge \sum_{i=1}^n \frac{T^2}{4}|\Gamma_1e_i|^2 = \frac{T^2}{4}\sum_{i=1}^n |(\xi+\xi^\top)e_i|^2 \nonumber \\
	&\ge \frac{T^2}{4}\sum_{i=1}^n \big(| \xi e_i|^2 + |\xi^\top e_i|^2\big) = \frac{T^2}{2}\sum_{i,j=1}^n \xi_{ij}^2 .  \label{pf:gaussian-lbfinal1}
\end{align}

Let us next turn to upper bounding the $(\Sigma_T)_{ii}$ term in \eqref{avgbound-UB}. We start from
\begin{align}
\sum_{i=1}^n (T^{-1}(\Sigma_T)_{ii}-1)^2 &= \sum_{i=1}^n   \bigg(\sum_{m=2}^{\infty} \frac{T^m}{(m+1)!} (\BB_m )_{ii}\bigg)^2,   \label{pf:gaussian-tr-ub2} 
\end{align}
where we note that the inner summation starts at $m=2$ because $\Gamma_1=\xi+\xi^\top$ is zero on the diagonal. 
For each $i$, $m \ge 2$, and $0 < r < m$,  we have  by Young's inequality
\begin{align*}
	|\langle e_i, \xi^r (\xi^\top)^{m-r} e_i\rangle  | \le \rho^{m-2} |\xi^\top e_i|^2.
\end{align*}
Thus, noting that $(\xi^m)^\top = (\xi^\top)^m$, 
\begin{equation*}
	|(\BB_m)_{ii}| \le 2 (\xi^m)_{ii} +  \sum_{r=1}^{m-1} {m \choose r}\rho^{m-2} |\xi^\top e_i|^2 \le  2 (\xi^m)_{ii}  +  2^m \rho^{m-2}  \sum_{j=1}^n\xi_{ij}^2.
\end{equation*}
This yields
\begin{align}
\sum_{i=1}^n (T^{-1}(\Sigma_T)_{ii}-1)^2 &\le  \sum_{i=1}^n  \bigg( 2\sum_{m=2}^{\infty} \frac{T^m}{(m+1)!} (\xi^m)_{ii} +  \sum_{m=2}^{\infty} \frac{T^m}{(m+1)!}  2^m\rho^{ m-2 }\sum_{j=1}^n \xi_{ij}^2 \bigg)^2 \nonumber \\
&\le 8  \sum_{i=1}^n \bigg(\sum_{m=2}^{\infty} \frac{T^m}{(m+1)!} (\xi^m)_{ii} \bigg)^2 + 32\sum_{i=1}^n \bigg(\sum_{m=2}^{\infty} \frac{T^m}{(m+1)!}  (2\rho)^{ m-2}\sum_{j=1}^n \xi_{ij}^2\bigg)^2 \nonumber\\
	&\le 8 D_T(\xi) + 32 T^4 e^{4\rho T}\sum_{i=1}^n\bigg(\sum_{j=1}^n\xi_{ij}^2\bigg)^2 . \label{pf:gaussian-ubfinal2}
\end{align}
where $D_T(\xi)$ is defined as in \eqref{eq:quantity.for.Gaussian}. The lower bound for the $(\Sigma_T)_{ii}$ term is similar: Starting again from \eqref{pf:gaussian-tr-ub2}, 
\begin{align}
	\sum_{i=1}^n (T^{-1}(\Sigma_T)_{ii}-1)^2 & = \sum_{i=1}^{n}  \bigg(2\sum_{m=2}^{\infty} \frac{T^m}{(m+1)!} (\xi^m)_{ii} + \sum_{m=2}^\infty \frac{T^m}{(m+1)!} \sum_{r=1}^{m-1} \binom{m}{r} \big(\xi^r (\xi^\top)^{m-r}\big)_{ii}\bigg)^2 \nonumber\\
	&\ge \sum_{i=1}^{n}  \bigg(2\sum_{m=2}^{\infty} \frac{T^m}{(m+1)!} (\xi^m)_{ii} + \frac{T^2}{3} (\xi \xi^\top)_{ii}\bigg)^2 \nonumber\\
	&\ge 4  D_T(\xi) + \frac{T^4}{9}\sum_{i=1}^n \bigg(\sum_{j=1}^{n} \xi_{ij}^2\bigg)^2, \label{pf:gaussian-lbfinal2}
\end{align}
where the first inequality follows from discarding all of the $m > 2$ terms in the second sum, and the last inequality follows from the nonnegativity of the entries of $\xi$.

To complete the proof of \eqref{eq:Gaussian.avg.ent}, we plug \eqref{pf:gaussian-ubfinal1} with \eqref{pf:gaussian-ubfinal2} into \eqref{avgbound-UB} to get the upper bound, and we combine \eqref{pf:gaussian-lbfinal1} with \eqref{pf:gaussian-lbfinal2} into \eqref{avgbound-LB} to get the lower bound. 

Finally, we demonstrate the claim \eqref{eq:quantity.for.Gaussian.bounded}. By similar argument with Young's inequality above,
\begin{align*}
	 (\xi^m)_{ii} = \langle e_i, \xi^m e_i\rangle \le \rho^{m-2}|\xi^\top e_i||\xi e_i | 
	 \le \rho^{m-2}\big(|\xi^\top e_i|^2 + |\xi e_i |^2\big) 
\end{align*}
for $m\ge 2$. Therefore, we have the upper bound
\begin{align*}
	\sum_{i=1}^n \bigg(\sum_{m=2}^{\infty} \frac{T^m}{(m+1)!} (\xi^m)_{ii} \bigg)^2
	&\le \sum_{i=1}^n \bigg(\sum_{m=2}^{\infty} \frac{T^m\rho^{m-2}}{(m+1)!} \sum_{j=1}^n\Big( \xi_{ij}^2+\xi_{ji}^2\Big)  \bigg)^2 \\
	&\le 2T^4 e^{2\rho T} \sum_{i=1}^n \bigg(\sum_{j=1}^n \xi_{ij}^2 \bigg)^2 + 2T^4 e^{2\rho T}\sum_{i=1}^n \bigg(\sum_{j=1}^n \xi_{ji}^2 \bigg)^2,
\end{align*}
and the lower bound follows from discarding the $m>2$ terms:
\begin{align*}
	\sum_{i=1}^n \bigg(\sum_{m=2}^{\infty} \frac{T^m}{(m+1)!} (\xi^m)_{ii} \bigg)^2
	\ge \frac{T^4}{36} \sum_{i=1}^n \big( (\xi^2)_{ii} \big)^2
	= \frac{T^4}{36} \sum_{i=1}^n \bigg(\sum_{j=1}^n \xi_{ij}\xi_{ji} \bigg)^2.  
\end{align*}
{ \ } $\quad$ \vskip-1.2cm \hfill\qedsymbol

\subsection{Proof of Proposition \ref{pr:Gaussiancase-max}}

 Use \eqref{pf:Guassian-UB} and Lemma \ref{le:Gaussianseries} to write
\begin{equation*}
H(P^v_T\,|\,Q^v_T) \le  e^{6\rho T}   \sum_{i,j \in v} \bigg(\sum_{m=1}^{\infty} \frac{T^m}{(m+1)!} (\BB_m)_{ij}\bigg)^2.
\end{equation*}
We first note that every entry of $\xi^r$ is bounded by $\delta$, for each $r \in \N$.
Indeed, this is true for $r=1$ by definition of $\delta$, and if we assume it is true for some $r$ then we prove it for $r+1$ by using  the assumption that row sums of $\xi$ are bounded by 1:
\begin{equation*}
(\xi^{r+1})_{ij} = \sum_{k=1}^n \xi_{ik}(\xi^r)_{kj} \le \delta\sum_{k=1}^n \xi_{ik} \le \delta.
\end{equation*}
Similarly, every entry of $\xi^r(\xi^\top)^{m-r}$ is bounded by $\delta$, for any integers $m \ge r \ge 0$ with $m \ge 1$. We deduce that $(\BB_m)_{ij} \le \delta 2^m$ for $m \ge 1$.
Thus,
\begin{equation*}
H(P^v_T\,|\,Q^v_T) \le  e^{6 \rho T}   \sum_{i,j \in v} \bigg(\sum_{m=1}^{\infty} \frac{T^m}{(m+1)!} \delta 2^m \bigg)^2 \le e^{10 \rho  T}  \delta^2|v|^2.
\end{equation*}
{ \ } $\quad$ \vskip-1.2cm \hfill\qedsymbol

\appendix

\section{Proofs for examples}
\subsection{Convex potentials: Example \ref{ex:convex}} \label{se:proofs:convexpotentials}
Recall in this setting that $ b_0^{i}(t,x) = -\nabla U(x) $ and $ b^{ij}(t,x,y) = - \nabla W(x - y) $ for all $i$. 
We need to check that Assumption \ref{asssump:uniform.in.time} holds, which includes Assumption \ref{asssump:common} in particular.
Assumption \ref{asssump:common}(i), on the wellposedness of the main SDE systems \eqref{eq.SDE.n.particle.sys} and \eqref{eq.independent.projection.sys}, follows from the Lipschitz continuity of $(\nabla U, \nabla W)$, with the independent projection being discussed in \cite[Proposition 4.1]{lacker2023independent}. Assumption \ref{asssump:common}(ii,iii) follow trivially from boundedness of $\nabla W$, with $\gamma = 2\||\nabla W|^2\|_\infty$ and $M =2\gamma$.

We turn next to Assumption \ref{asssump:uniform.in.time}(iv).
The assumed  $\nabla W(x - \cdot) \in L^1(Q^j_t)$ for all $(t,x) \in [0, \infty) \times \R^d$ and $j \in [n]$, as well as the local boundedness of $(t,x) \mapsto \langle Q^j_t, \nabla W(x-\cdot)\rangle$. Finally, for the integrability requirements \eqref{assump.unit.int}, note that the assumed LSI for $Q_0$ implies that $Q_0$ has finite moments of every order. It was shown in \cite[Proposition 4.1]{lacker2023independent} that Lipschitz coefficients finite moments at time zero lead to the moment bound $\sup_{t \in [0,T]}\E|Y^j_t|^p< \infty$ for any $p \ge 1$, $T> 0$, and $j \in [n]$. Similarly, the Lipschitz continuity of $ (\nabla U, \nabla W) $ and the assumption that $ P_0 $ has finite moments of all orders implies the moment bound $\sup_{t \in [0,T]}\E|X^j_t|^p< \infty$ for any $p \ge 1$, $T> 0$, and $j \in [n]$. The the integrability requirements \eqref{assump.unit.int} are then consequences of the linear growth of $ (\nabla U, \nabla W) $.  

We lastly explain why the LSI of Assumption \ref{asssump:uniform.in.time}(ii) holds.
The independent projection \eqref{eq.independent.projection.sys} can be written as  
\begin{align} \label{eq:ip.conv.pot}
	d Y^{i}_t = \Big(-\nabla U(Y^{i}_t) - \sum_{j \neq i} \xi_{ij} \nabla W \ast Q^{j}_t (Y^{i}_t )\Big) dt + \sigma dB^i_t, \quad i \in [n].
\end{align}
Fix $ i \in [n] $. The drift of $Y^i_t$ at time $t$ is the gradient of the function
\begin{align*}
\Psi_t(x) = U(x) +  \sum_{j \neq i} \xi_{ij}  W \ast Q^{j}_t(x), \qquad x \in \R^d,
\end{align*}
which is easily checked to satisfy $\nabla^2\Psi_t(x) \ge \nabla^2U(x) \succeq \lambda I$, using the assumed $\lambda$-convexity of $U$ and convexity of $W$.
This verifies the curvature condition of \cite[Proposition 3.12]{malrieu2001logarithmic} and we can follow the arguments therein to deduce that $ Q^i_t $ satisfies a LSI with constant
\[
\frac{\sigma^2}{\lambda}(1 - e^{-4 \lambda t/\sigma^2} ) +  \frac{\eta_0}{4}e^{-4\lambda t/\sigma^2}\le \max(\eta_0/4, \sigma^2/\lambda) =: \eta.
\]

\subsection{Models on the torus: Example \ref{ex:torus}} \label{se:proofs:torus}
Checking Assumption \ref{asssump:uniform.in.time} in this example is almost the same as in the proof of \cite[Corollary 2.9]{lacker2023sharp}, and we just sketch the main differences.
The well-posedness Assumption \ref{asssump:common}(i) is straightforward, as are Assumption \ref{asssump:common}(ii,iii) and \ref{asssump:uniform.in.time}(iv) by the boundedness of $K$. The only changes are in checking the LSI, Assumption \ref{asssump:uniform.in.time}, and mainly identifying the constant therein. To this end, we give the following lemma, adapted from \cite[Corollary 2.9]{lacker2023sharp}, which in turn borrowed key ideas from the proofs of \cite[Proposition 3.1]{carrillo2020long} and \cite[Theorem 2]{guillin2024uniform}.

\begin{lemma} \label{lem.torus.prep}
	For each $t > 0$ and $i \in [n]$, the density of $Q^i_t$ is  $C^2$ and obeys the pointwise bound
	\begin{align*}
		\frac{1}{\lambda e^{r}} \le Q^{i}_T(x)  \le \frac{\lambda}{1 - r e^{r}}, \quad \text{where} \quad 	r 		 =   \frac{\sqrt{2 \log \lambda} \Vert \diver K \Vert_\infty}{ 2\sigma^2\pi^2 -  \Vert \diver K \Vert_\infty } .
	\end{align*}
	Moreover, it holds that $r < 1/2$, and $Q_t=Q^1_t\otimes \cdots \otimes Q^n_t$ satisfies the LSI
	\begin{align}
		H(\cdot \,|\, Q^i_t) \le \eta I(\cdot \,|\, Q^i_t), \quad \text{where} \quad \eta := \lambda^2(1 - 2r)^{-1}(8\pi^2)^{-1}.
	\end{align}
	\begin{proof}
	 { \ }
	 
	 	 \noindent\textbf{Step 1.}
	Let $\bm{1}$ denote the uniform (Lebesgue) probability measure on $\mathbb{T}^d$. We first adapt the argument of \cite[Proposition 3.1]{carrillo2020long} to show that for each $t > 0$ and $i \in [n]$
	\begin{align}
			H (Q^i_t \, |\,\mathbf{1}) \le e^{-2ct} \log \lambda, \quad c := 2\sigma^2\pi^2 -  \Vert \diver K \Vert_\infty   .  \label{eq: entropy.decay}
		\end{align}
		 Note that $Q^i_t$ satisfies the Fokker-Planck equation $\partial_t Q^i = -\diver(b^iQ^i) + (\sigma^2/2)\Delta Q^i$ with $b^i_t = \sum_{j }\xi_{ij}K * Q^j_t$. A standard computation followed by integration by parts yields
	\begin{equation*}
	\frac{d}{dt}H(Q^i_t\,|\,\bm{1}) = -\int Q^i_t \, \diver \,b^i_t  - \frac{\sigma^2}{2}\int Q^i_t|\nabla \log Q^i_t|^2  .
	\end{equation*}
	Here and below, the integrals are all with respect to the uniform probability measure on the torus.
Using the log-Sobolev inequality for the uniform measure on $\mathbb{T}^d$, we have
\begin{equation}
	\int Q^i_t|\nabla \log Q^i_t|^2 \ge 8\pi^2 H(Q^i_t\,|\,\bm{1}). \label{ineq:LSItorus}
	\end{equation}
	Indeed, see \cite{emery1987simple} for proof of this LSI in dimension $d=1$, which tensorizes to general dimension. Using the form of $b$,
	\begin{align*}
	-\int Q^i_t \,\diver b^i_t &= -\sum_{j\neq i}\xi_{ij}\int  Q^i_t \, \diver K * Q^j_t = -\sum_{j\neq i}\xi_{ij}\int (Q^i_t-\bm{1}) \diver K * (Q^j_t-\bm{1})  \\
		&\le \|Q^i_t-\bm{1}\|_{\mathrm{TV}} \| \diver K \|_\infty \sum_{j\neq i}\xi_{ij}\|Q^j_t-\bm{1}\|_{\mathrm{TV}}.
	\end{align*}
	Combining the three previous displays and using Gronwall's inequality,
	\begin{equation*}
	e^{4 \sigma^2\pi^2 t} H(Q^i_t\,|\,\bm{1}) \le H(Q^i_0\,|\,\bm{1}) + \| \diver K \|_\infty \int_0^t e^{4 \sigma^2\pi^2 s} \|Q^i_s-\bm{1}\|_{\mathrm{TV}} \sum_{j }\xi_{ij}\|Q^j_s-\bm{1}\|_{\mathrm{TV}} \,ds.
	\end{equation*}
	Letting $\widehat{H}_t = \max_{i \in [n]}H(Q^i_t\,|\,\bm{1})$ and using Pinsker's inequality along with $\sum_j \xi_{ij} \le 1$, we deduce
	\begin{equation*}
	e^{4 \sigma^2\pi^2 t} \widehat{H}_t \le \widehat{H}_0 + 2\| \diver K \|_\infty \int_0^t e^{4 \sigma^2\pi^2 s} \widehat{H}_s \, ds.
	\end{equation*}
	Applying Gronwall's inequality again, along with the assumption $Q^i_0 \le \lambda$, we find
	\begin{equation*}
	\widehat{H}_t \le  e^{(2\| \diver K \|_\infty - 4\sigma^2\pi^2)t} \widehat{H}_0 \le e^{(2\| \diver K \|_\infty - 4\sigma^2\pi^2)t}  \log \lambda .
	\end{equation*}
	
	 { \ }
	 
	 	 \noindent\textbf{Step 2.}
		We next prove the pointwise bound on $Q^i_t$. Fix $T > 0$ and $x_i \in \T^d$, and let $(Z^i)_{t \in [0,T]}$ be unique strong solution of the SDE system
		\begin{equation*}
			d Z^{i}_t = -\sum_{j \neq i} \xi_{ij} K \ast Q^{j}_{T - t} (Z^{i}_t) dt + dB^i_t, \quad Z^{i}_0 = x_i.
		\end{equation*}
		Using Ito's formula and the Fokker-Planck equation for $Q^i$, and taking expectations, we have
		\begin{equation} \label{eq: Ito.PDE}
			\E \left[Q^{i}_{T-t}(Z^{i}_t)\right] = Q^i_T(x) + \sum_{j} \xi_{ij} \E \int_0^t Q^{i}_{T-s}(Z^{i}_s) \,  \diver K \ast Q^{j}_{T-s}(Z^{i}_s) ds, \quad t \in [0,T].
		\end{equation}
		Noting that $ \diver K \ast \mathbf{1} \equiv 0 $, we have for any $ u \in [0,T] $ that 
		\begin{equation*}
			\Vert \diver K \ast Q^{j}_u \Vert_\infty \le \Vert \diver K \ast (Q^{j}_u - \mathbf{1}) \Vert_\infty \le \Vert \diver K \Vert_\infty \Vert Q^{j}_u - \mathbf{1} \Vert_{\text{TV}} \le \sqrt{2 \log \lambda} \Vert \diver K \Vert_\infty e^{-c u},
		\end{equation*}
		where the last step uses Pinsker's inequality and \eqref{eq: entropy.decay}. Setting $ a = \sqrt{2 \log \lambda} \Vert \diver K \Vert_\infty  $, and using $\sum_j \xi_{ij} \le 1$, this implies
		\begin{equation*}
			\E \left[Q^{i}_{T-t}(Z^{i}_t)\right] \le Q^i_T(x) + a \int_0^t e^{-c (T - s)} \E \left[Q^{i}_{T-s}(Z^{i}_s)\right]  ds .
		\end{equation*}
		By  Gronwall's inequality, we obtain for $t \in [0,T]$
		\begin{equation*}
			\E \left[Q^{i}_{T-t}(Z^{i}_t)\right] \le Q_T^i(x) \exp \left(a e^{-cT}  \int_0^t e^{cs} ds\right) \le Q_T^i(x) e^{a/c}.
		\end{equation*}
		Setting $t = T$ and using the lower bound $Q^i_0 \ge \lambda^{-1}$ yields 
		\begin{equation*}
			Q_T^{i}(x)  \ge e^{-a/c}	\E \left[Q^{i}_{0}(Z^{i}_t)\right]   \ge e^{-a/c}\lambda^{-1}.
		\end{equation*}
		Similarly, using \eqref{eq: Ito.PDE}, we can deduce 
		\begin{align*}
			Q^{i}_T(x) &\le \E \left[Q^{i}_{0}(Z^{i}_t)\right]  + a  \int_0^T e^{-c (T - s)}    \E[Q^{i}_{T - s}(Z^{i}_s)] ds \\
			& \le \lambda + a  Q_T^i(x) e^{a/c}\int_0^T e^{-c (T - s)}\, ds.
		\end{align*}
		Therefore, 
		\begin{equation*}
			Q^{i}_T(x) \le   \lambda \big/\big(1 -(a/c)e^{a/c}\big).
		\end{equation*}
		Combining gives us the claimed bounds on the density. It was assumed in \eqref{eq:K.smallness.assum} that $\|\diver K\|_\infty < 2\sigma^2\pi^2/(1+\sqrt{2\log \lambda})$, which ensures that $r < 1/2$. Lastly, by noting that
		\begin{align}
			\frac{\sup Q^{i}_t}{\inf Q^{i}_t} \le  \frac{\lambda^2 e^{r}}{1 - r e^{r}} \le \frac{\lambda^2}{1 - 2r}.
		\end{align}
		we have by Holley-Stroock \cite[Proposition 5.1.6]{bakry2014analysis} that $Q^{i}_t$ satisfies the claimed LSI.
	\end{proof}
	\end{lemma}
 
 \begin{remark}
 The above proof corrects two small errors in the argument of \cite[Corollary 2.9]{lacker2023sharp}. First, the constant $c$ (and thus the denominator of $r$) was missing the factor of 2, carrying forward a typo from \cite[Equation (3.3)]{carrillo2020long} in which the constant in the LSI \eqref{ineq:LSItorus} was misquoted as $4\pi^2$ instead of $8\pi^2$. Second, the factor $(8\pi^2)^{-1}$ was missing from $\eta$ in \cite[Corollary 2.9]{lacker2023sharp}, due to a misapplication of Holley-Stroock at  the end of the proof.
 \end{remark}

\bibliographystyle{amsplain} 
\bibliography{biblio}

\providecommand{\bysame}{\leavevmode\hbox to3em{\hrulefill}\thinspace}
\providecommand{\MR}{\relax\ifhmode\unskip\space\fi MR }
\providecommand{\MRhref}[2]{%
  \href{http://www.ams.org/mathscinet-getitem?mr=#1}{#2}
}
\providecommand{\href}[2]{#2}
\begin{thebibliography}{10}

\bibitem{auffinger201750}
A.~Auffinger, M.~Damron, and J.~Hanson, \emph{50 years of first-passage
  percolation}, vol.~68, American Mathematical Soc., 2017.

\bibitem{ayi2024mean}
N.~Ayi, N.~P. Duteil, and D.~Poyato, \emph{Mean-field limit of non-exchangeable
  multi-agent systems over hypergraphs with unbounded rank}, arXiv preprint
  arXiv:2406.04691 (2024).

\bibitem{backhausz2022action}
{\'A}.~Backhausz and B.~Szegedy, \emph{Action convergence of operators and
  graphs}, Canadian Journal of Mathematics \textbf{74} (2022), no.~1, 72--121.

\bibitem{bakry2014analysis}
D.~Bakry, I.~Gentil, and M.~Ledoux, \emph{Analysis and geometry of {M}arkov
  diffusion operators}, vol. 103, Springer, 2014.

\bibitem{basak2017universality}
A.~Basak and S.~Mukherjee, \emph{Universality of the mean-field for the {P}otts
  model}, Probability Theory and Related Fields \textbf{168} (2017), 557--600.

\bibitem{bayraktar2023graphon}
E.~Bayraktar, S.~Chakraborty, and R.~Wu, \emph{Graphon mean field systems}, The
  Annals of Applied Probability \textbf{33} (2023), no.~5, 3587--3619.

\bibitem{bayraktar2022stationarity}
E.~Bayraktar and R.~Wu, \emph{Stationarity and uniform in time convergence for
  the graphon particle system}, Stochastic Processes and their Applications
  \textbf{150} (2022), 532--568.

\bibitem{arous1999increasing}
G.~{Ben Arous} and O.~Zeitouni, \emph{Increasing propagation of chaos for mean
  field models}, Annales de l'Institut Henri Poincare (B) Probability and
  Statistics, vol.~35, Elsevier, 1999, pp.~85--102.

\bibitem{benjamini2011recurrence}
I.~Benjamini and O.~Schramm, \emph{Recurrence of distributional limits of
  finite planar graphs}, Electronic Journal of Probability \textbf{6} (2001), 1
  -- 13.

\bibitem{bet2024weakly}
G.~Bet, F.~Coppini, and F.~R. Nardi, \emph{Weakly interacting oscillators on
  dense random graphs}, Journal of Applied Probability \textbf{61} (2024),
  no.~1, 255--278.

\bibitem{bhamidi2019weakly}
S.~Bhamidi, A.~Budhiraja, and R.~Wu, \emph{Weakly interacting particle systems
  on inhomogeneous random graphs}, Stochastic Processes and their Applications
  \textbf{129} (2019), no.~6, 2174--2206.

\bibitem{blei2017variational}
D.M. Blei, A.~Kucukelbir, and J.D. McAuliffe, \emph{Variational inference: {A}
  review for statisticians}, Journal of the American statistical Association
  \textbf{112} (2017), no.~518, 859--877.

\bibitem{borgs2019I}
C.~Borgs, J.~Chayes, H.~Cohn, and Y.~Zhao, \emph{{An $L^p$ theory of sparse
  graph convergence {I}: {L}imits, sparse random graph models, and power law
  distributions}}, Transactions of the American Mathematical Society
  \textbf{372} (2019), no.~5, 3019--3062.

\bibitem{borgs2019II}
C.~Borgs, J.~T. Chayes, H.~Cohn, and Y.~Zhao, \emph{{An $L^{p}$ theory of
  sparse graph convergence {II}: {LD} convergence, quotients and right
  convergence}}, The Annals of Probability \textbf{46} (2018), no.~1, 337 --
  396.

\bibitem{bresch2022new}
D.~Bresch, P.-E. Jabin, and J.~Soler, \emph{A new approach to the mean-field
  limit of {V}lasov-{F}okker-{P}lanck equations}, arXiv preprint
  arXiv:2203.15747 (2022).

\bibitem{bresch2023mean}
D.~Bresch, P.-E. Jabin, and Z.~Wang, \emph{Mean field limit and quantitative
  estimates with singular attractive kernels}, Duke Mathematical Journal
  \textbf{172} (2023), no.~13, 2591--2641.

\bibitem{bris2022note}
P.~L. Bris and C.~Poquet, \emph{A note on uniform in time mean-field limit in
  graphs}, arXiv preprint arXiv:2211.11519 (2022).

\bibitem{brunickshreve}
G.~Brunick and S.~Shreve, \emph{{Mimicking an Itô process by a solution of a
  stochastic differential equation}}, The Annals of Applied Probability
  \textbf{23} (2013), no.~4, 1584 -- 1628.

\bibitem{carrillo2020long}
J.A. Carrillo, R.S. Gvalani, G.A. Pavliotis, and A.~Schlichting,
  \emph{Long-time behaviour and phase transitions for the mckean--vlasov
  equation on the torus}, Archive for Rational Mechanics and Analysis
  \textbf{235} (2020), no.~1, 635--690.

\bibitem{cattiaux2024entropy}
P.~Cattiaux, \emph{Entropy on the path space and application to singular
  diffusions and mean-field models}, arXiv preprint arXiv:2404.09552 (2024).

\bibitem{chaintron2021propagation}
L.-P. Chaintron and A.~Diez, \emph{Propagation of chaos: a review of models,
  methods and applications. {II}. {A}pplications}, arXiv preprint
  arXiv:2106.14812 (2021).

\bibitem{chaintron2022propagation}
\bysame, \emph{Propagation of chaos: a review of models, methods and
  applications. {I}. {M}odels and methods}, arXiv preprint arXiv:2203.00446
  (2022).

\bibitem{chatterjee2016nonlinear}
S.~Chatterjee and A.~Dembo, \emph{Nonlinear large deviations}, Advances in
  Mathematics \textbf{299} (2016), 396--450.

\bibitem{chiba2016mean}
H.~Chiba and G.~S. Medvedev, \emph{The mean field analysis for the {K}uramoto
  model on graphs {I}. {T}he mean field equation and transition point
  formulas}, arXiv preprint arXiv:1612.06493 (2016).

\bibitem{coppini2020law}
F.~Coppini, H.~Dietert, and G.~Giacomin, \emph{A law of large numbers and large
  deviations for interacting diffusions on {E}rd{\H{o}}s--{R}{\'e}nyi graphs},
  Stochastics and Dynamics \textbf{20} (2020), no.~02, 2050010.

\bibitem{de2023attractive}
A.~C. de~Courcel, M.~Rosenzweig, and S.~Serfaty, \emph{The attractive log gas:
  Stability, uniqueness, and propagation of chaos}, arXiv preprint
  arXiv:2311.14560 (2023).

\bibitem{de2023sharp}
A.~C. {de Courcel}, M.~Rosenzweig, and S.~Serfaty, \emph{Sharp uniform-in-time
  mean-field convergence for singular periodic {R}iesz flows}, Annales de
  l'Institut Henri Poincar{\'e} C (2023).

\bibitem{delattre2016note}
S.~Delattre, G.~Giacomin, and E.~Lu{\c{c}}on, \emph{A note on dynamical models
  on random graphs and {F}okker--{P}lanck equations}, Journal of Statistical
  Physics \textbf{165} (2016), 785--798.

\bibitem{dembo1991information}
A.~Dembo, T.~M. Cover, and J.~A. Thomas, \emph{Information theoretic
  inequalities}, IEEE Transactions on Information theory \textbf{37} (1991),
  no.~6, 1501--1518.

\bibitem{du2023sequential}
K.~Du, Y.~Jiang, and X.~Li, \emph{Sequential propagation of chaos}, arXiv
  preprint arXiv:2301.09913 (2023).

\bibitem{duerinckx2016mean}
M.~Duerinckx, \emph{Mean-field limits for some {R}iesz interaction gradient
  flows}, SIAM Journal on Mathematical Analysis \textbf{48} (2016), no.~3,
  2269--2300.

\bibitem{durrett2010random}
R.~Durrett, \emph{Random graph dynamics}, vol.~20, Cambridge university press,
  2010.

\bibitem{eden1961two}
M.~Eden, \emph{A two-dimensional growth process}, Dynamics of fractal surfaces
  \textbf{4} (1961), 223--239.

\bibitem{elek2012measure}
G.~Elek and B.~Szegedy, \emph{A measure-theoretic approach to the theory of
  dense hypergraphs}, Advances in Mathematics \textbf{231} (2012), no.~3-4,
  1731--1772.

\bibitem{emery1987simple}
M.~{\'E}mery and J.~E. Yukich, \emph{A simple proof of the logarithmic
  {S}obolev inequality on the circle}, S{\'e}minaire de probabilit{\'e}s de
  Strasbourg \textbf{21} (1987), 173--175.

\bibitem{estrada2008communicability}
E.~Estrada and N.~Hatano, \emph{Communicability in complex networks}, Physical
  Review E \textbf{77} (2008), no.~3, 036111.

\bibitem{gkogkas2022graphop}
M.~A. Gkogkas and C.~Kuehn, \emph{Graphop mean-field limits for {K}uramoto-type
  models}, SIAM Journal on Applied Dynamical Systems \textbf{21} (2022), no.~1,
  248--283.

\bibitem{gozlan2010transport}
N.~Gozlan and C.~L{\'e}onard, \emph{Transport inequalities. a survey}, arXiv
  preprint arXiv:1003.3852 (2010).

\bibitem{guillin2024uniform}
A.~Guillin, P.~Le~Bris, and P.~Monmarch{\'e}, \emph{Uniform in time propagation
  of chaos for the 2{D} vortex model and other singular stochastic systems},
  Journal of the European Mathematical Society (2024), 1--28.

\bibitem{han2022entropic}
Y.~Han, \emph{Entropic propagation of chaos for mean field diffusion with
  {$L^p$} interactions via hierarchy, linear growth and fractional noise},
  arXiv preprint arXiv:2205.02772 (2022).

\bibitem{hess2023higher}
E.~Hess-Childs and K.~Rowan, \emph{Higher-order propagation of chaos in {$L^2$}
  for interacting diffusions}, arXiv preprint arXiv:2310.09654 (2023).

\bibitem{hu2024mimicking}
K.~Hu, K.~Ramanan, and W.~Salkeld, \emph{A mimicking theorem for processes
  driven by fractional {B}rownian motion}, arXiv preprint arXiv:2405.08803
  (2024).

\bibitem{jabin2021mean}
P.-E. Jabin, D.~Poyato, and J.~Soler, \emph{Mean-field limit of
  non-exchangeable systems}, arXiv preprint arXiv:2112.15406 (2021).

\bibitem{jabin2016mean}
P.-E. Jabin and Z.~Wang, \emph{Mean field limit and propagation of chaos for
  {V}lasov systems with bounded forces}, Journal of Functional Analysis
  \textbf{271} (2016), no.~12, 3588--3627.

\bibitem{jabin2018quantitative}
\bysame, \emph{Quantitative estimates of propagation of chaos for stochastic
  systems with ${W}^{-1,\infty}$ kernels}, Inventiones mathematicae
  \textbf{214} (2018), 523--591.

\bibitem{jabin2023mean}
P.-E. Jabin and D.~Zhou, \emph{The mean-field limit of sparse networks of
  integrate and fire neurons}, arXiv preprint arXiv:2309.04046 (2023).

\bibitem{jabir2019rate}
J.-F. Jabir, \emph{Rate of propagation of chaos for diffusive stochastic
  particle systems via {G}irsanov transformation}, arXiv preprint
  arXiv:1907.09096 (2019).

\bibitem{jackson2023approximately}
J.~Jackson and D.~Lacker, \emph{Approximately optimal distributed stochastic
  controls beyond the mean field setting}, arXiv preprint arXiv:2301.02901
  (2023).

\bibitem{jackson2024concentration}
J.~Jackson and A.~Zitridis, \emph{Concentration bounds for stochastic systems
  with singular kernels}, arXiv preprint arXiv:2406.02848 (2024).

\bibitem{kuehn2022vlasov}
C.~Kuehn and C.~Xu, \emph{Vlasov equations on digraph measures}, Journal of
  Differential Equations \textbf{339} (2022), 261--349.

\bibitem{lacker2018strong}
D.~Lacker, \emph{On a strong form of propagation of chaos for
  {M}c{K}ean-{V}lasov equations}, Electronic Communications in Probability
  \textbf{23} (2018), 1 -- 11.

\bibitem{lacker2022hierarchies}
\bysame, \emph{Hierarchies, entropy, and quantitative propagation of chaos for
  mean field diffusions}, 2022.

\bibitem{lacker2022quantitative}
\bysame, \emph{Quantitative approximate independence for continuous mean field
  {G}ibbs measures}, Electronic Journal of Probability \textbf{27} (2022),
  1--21.

\bibitem{lacker2023independent}
\bysame, \emph{Independent projections of diffusions: {G}radient flows for
  variational inference and optimal mean field approximations}, arXiv preprint
  arXiv:2309.13332 (2023).

\bibitem{lacker2023sharp}
D.~Lacker and L.~Le~Flem, \emph{Sharp uniform-in-time propagation of chaos},
  Probability Theory and Related Fields (2023), 1--38.

\bibitem{lacker2023local}
D.~Lacker, K.~Ramanan, and R.~Wu, \emph{Local weak convergence for sparse
  networks of interacting processes}, The Annals of Applied Probability
  \textbf{33} (2023), no.~2, 843--888.

\bibitem{lacker2023marginal}
\bysame, \emph{Marginal dynamics of interacting diffusions on unimodular
  {G}alton--{W}atson trees}, Probability Theory and Related Fields \textbf{187}
  (2023), no.~3, 817--884.

\bibitem{lacker2022case}
D.~Lacker and A.~Soret, \emph{A case study on stochastic games on large graphs
  in mean field and sparse regimes}, Mathematics of Operations Research
  \textbf{47} (2022), no.~2, 1530--1565.

\bibitem{lovasz2012large}
L.~Lov{\'a}sz, \emph{Large networks and graph limits}, vol.~60, American
  Mathematical Soc., 2012.

\bibitem{malrieu2001logarithmic}
F.~Malrieu, \emph{Logarithmic {S}obolev inequalities for some nonlinear
  {PDE}'s}, Stochastic processes and their applications \textbf{95} (2001),
  no.~1, 109--132.

\bibitem{medvedev2014nonlinear}
G.~Medvedev, \emph{The nonlinear heat equation on dense graphs and graph
  limits}, SIAM Journal on Mathematical Analysis \textbf{46} (2014), no.~4,
  2743--2766.

\bibitem{medvedev2018continuum}
\bysame, \emph{The continuum limit of the {K}uramoto model on sparse random
  graphs}, arXiv preprint arXiv:1802.03787 (2018).

\bibitem{mishura2020existence}
Y.~Mishura and A.~Veretennikov, \emph{Existence and uniqueness theorems for
  solutions of {M}c{K}ean--{V}lasov stochastic equations}, Theory of
  Probability and Mathematical Statistics \textbf{103} (2020), 59--101.

\bibitem{monmarche2024time}
P.~Monmarch{\'e}, Z.~Ren, and S.~Wang, \emph{Time-uniform log-{S}obolev
  inequalities and applications to propagation of chaos}, arXiv preprint
  arXiv:2401.07966 (2024).

\bibitem{newman2018networks}
M.~Newman, \emph{Networks}, Oxford university press, 2018.

\bibitem{oliveira2019interacting}
R.~Oliveira and G.~Reis, \emph{Interacting diffusions on random graphs with
  diverging average degrees: {H}ydrodynamics and large deviations}, Journal of
  Statistical Physics \textbf{176} (2019), no.~5, 1057--1087.

\bibitem{oliveira2020interacting}
R.~Oliveira, G.~Reis, and L.~Stolerman, \emph{{Interacting diffusions on sparse
  graphs: hydrodynamics from local weak limits}}, Electronic Journal of
  Probability \textbf{25} (2020), 1 -- 35.

\bibitem{otto2000generalization}
F.~Otto and C.~Villani, \emph{Generalization of an inequality by {T}alagrand
  and links with the logarithmic {S}obolev inequality}, Journal of Functional
  Analysis \textbf{173} (2000), no.~2, 361--400.

\bibitem{peszek2023heterogeneous}
J.~Peszek and D.~Poyato, \emph{Heterogeneous gradient flows in the topology of
  fibered optimal transport}, Calculus of Variations and Partial Differential
  Equations \textbf{62} (2023), no.~9, 258.

\bibitem{ramanan2023interacting}
K.~Ramanan, \emph{Interacting stochastic processes on sparse random graphs},
  arXiv preprint arXiv:2401.00082 (2023).

\bibitem{richardson1973random}
D.~Richardson, \emph{Random growth in a tessellation}, Mathematical Proceedings
  of the Cambridge Philosophical Society, vol.~74, Cambridge University Press,
  1973, pp.~515--528.

\bibitem{rosenzweig2023modulated}
M.~Rosenzweig and S.~Serfaty, \emph{Modulated logarithmic {S}obolev
  inequalities and generation of chaos}, arXiv preprint arXiv:2307.07587
  (2023).

\bibitem{ross2014introduction}
S.M. Ross, \emph{Introduction to probability models}, Academic press, 2014.

\bibitem{serfaty2020mean}
S.~Serfaty, \emph{{Mean field limit for {C}oulomb-type flows}}, Duke
  Mathematical Journal \textbf{169} (2020), no.~15, 2887 -- 2935.

\bibitem{van2001first}
R.~Van Der~Hofstad, G.~Hooghiemstra, and P.~Van~Mieghem, \emph{First-passage
  percolation on the random graph}, Probability in the Engineering and
  Informational Sciences \textbf{15} (2001), no.~2, 225--237.

\bibitem{wang2024sharp}
S.~Wang, \emph{Sharp local propagation of chaos for mean field particles with
  ${W}^{-1,\infty}$ kernels}, arXiv preprint arXiv:2403.13161 (2024).

\bibitem{wang2022mean}
Z.~Wang, X.~Zhao, and R.~Zhu, \emph{Mean-field limit of non-exchangeable
  interacting diffusions with singular kernels}, arXiv preprint
  arXiv:2209.14002 (2022).

\end{thebibliography}
\end{document}